\documentclass[a4paper,11pt]{article}
\usepackage[latin1]{inputenc}
\usepackage[T1]{fontenc}
\usepackage[british,UKenglish,USenglish,american]{babel}
\usepackage{amsmath}
\usepackage{amsfonts}
\usepackage{hyperref}
\usepackage[T1]{fontenc}
\usepackage[latin1]{inputenc} 
\usepackage{graphics}
\usepackage{makeidx}
\usepackage{fancyhdr}
\usepackage{bbm}
\usepackage{bm} 
\usepackage{amsthm} 
\usepackage{amssymb}
\usepackage{color}
\numberwithin{equation}{section}
\usepackage[mathscr]{eucal}
\newcommand{\be}{\begin{equation}}
\newcommand{\ee}{\end{equation}}
\newcommand{\benn}{\begin{equation*}}
\newcommand{\eenn}{\end{equation*}}
\newcommand{\bea}{\begin{eqnarray}}
\newcommand{\eea}{\end{eqnarray}}
\newcommand{\beann}{\begin{eqnarray*}}
\newcommand{\eeann}{\end{eqnarray*}}

\newtheorem{theorem}{Theorem}[section]

\newtheorem{lemma}[theorem]{Lemma}

\newtheorem{definition}[theorem]{Definition}
\newcommand{\vertiii}[1]{{\left\vert\kern-0.25ex\left\vert\kern-0.25ex\left\vert #1
		\right\vert\kern-0.25ex\right\vert\kern-0.25ex\right\vert}}
\newtheorem{remark}[theorem]{Remark}

\def\R{\mathbb{R}}

\def\cB{\mathcal{B}}

\title{Existence of smooth stable manifolds for a class of parabolic SPDEs with fractional noise\thanks{This paper was supported by National Natural Science Foundation of China (12271177, 11871225) and Natural Science Foundation of Guangdong Province (2023A1515010622). }}
\author{Xiaofang Lin\thanks{School of Mathematics, South China University of Technology, Guangzhou 510640, P.R. China. E-Mail: xflinmath@163.com}, Alexandra 
	Neam\c tu\thanks{University of Konstanz, Department of Mathematics and Statistics, Universit\"atsstra\ss{}e 10 78464 Konstanz, Germany. E-Mail: alexandra.neamtu@uni-konstanz.de}, and 
		Caibin Zeng\thanks{School of Mathematics, South China University of Technology, Guangzhou 510640, P.R. China. E-Mail: macbzeng@scut.edu.cn}~~\thanks{Corresponding author.}}

\begin{document}

\maketitle

\begin{abstract}
  Little seems to be known about the invariant manifolds for stochastic partial differential equations (SPDEs) driven by nonlinear multiplicative noise.  Here we contribute to this aspect and analyze the Lu-Schmalfu{\ss} conjecture [Garrido-Atienza, et al., J. Differential Equations, 248(7):1637--1667, 2010] on the existence of stable manifolds for a class of parabolic SPDEs driven by nonlinear mutiplicative fractional noise. We emphasize that stable manifolds for SPDEs are infinite-dimensional objects, and the classical Lyapunov-Perron method cannot be applied, since the 
  Lyapunov-Perron operator does not give any information about the backward orbit.
  However, by means of interpolation theory, we construct a suitable function space in which the  discretized Lyapunov-Perron-type operator 
  has a unique fixed point. Based on this we further prove the existence and smoothness of local stable manifolds for such SPDEs.

\end{abstract}

\newcommand\keywords[1]{\textbf{Keywords}: #1}
\newcommand\msc[1]{\textbf{Mathematics Subject Classification (2020)}: #1}
 

\msc{60H15, 60G22, 37H05, 37L55.} 
\section{Introduction}

Invariant manifolds including center, stable and unstable manifolds are powerful tools in the study of properties and structures for deterministic and random dynamical systems, such as the derivation of normal forms and understanding the nature of  bifurcations \cite{Hirsch1977,Deng1990,chow_li_wang_1994,Wanner1995}.  Recently, a random version of the $C^1$ Hartamn theorem was proved in \cite{LZZ2020} by constructing  a smooth weak-stable invariant manifold which provides a way to overcome difficulties from nonuniformity and measurability. In general, there are two approaches to construct  invariant manifolds: Hadamard's graph transform method \cite{Hadamard1901} and the Lyapunov-Perron method \cite{Perron1928,Liapounof}. The contribution list in this respect is too extensive to be completely written down herein. We refer to \cite{DLS2003,DLS2004,LL2010,mohammed2008,Caraballo2010} and the references specified therein.\\
In this work we investigate stable manifolds for  SPDEs given by
\begin{equation}\label{spde:intro}
    du(t)= [A u(t) + F(u(t))]dt + G(u(t))d\omega(t),
\end{equation}
where $\omega$ is a Hilbert space-valued fractional Brownian  motion (fBm) with Hurst index $H>\frac{1}{2}$. The assumptions on the linear operator $A$ and on the nonlinear terms $F$ and $G$ are specified in Subsection~\ref{assumptions}.
To obtain random invariant manifolds, one needs to  define a random dynamical system (RDS) from the stochastic (partial) differential equations considered. The Kolmogorov continuity theorem guarantees the existence of a finite dimensional RDS, however, there is no analogous result in the infinite dimensional setting \cite[page 246]{Arnold1998}. In the literature,  the conjugacy approach is mostly applied to overcome this difficulty, see \cite{DLS2003,DLS2004,LZ2021} and references among them. It allows one to convert the SPDEs into pathwise deterministic systems. For instance, a coordinate transform based on the stationary Ornstein-Uhlenbeck process is employed. This transformation can be only employed if the noise is additive or linear multiplicative. For nonlinear multiplicative noise, pathwise integrals are extensively used to avoid this problem. In particular,  using pathwise stochastic integrals via fractional calculus \cite{Zahle1998} or rough path theory \cite{FH2020,Gu2004}, one can immediately infer the existence of an RDS.  In this respect,  the pathwise solution and an RDS gererated by  a SPDE driven by fBm with $H>\frac{1}{2}$ has been derived in \cite{GLS2010} and a local unstable invariant manifold for this RDS has been established in \cite{GarridoAtienza2010}. Similar results in this direction have been obtained using rough path theory. More precisely, stable and unstable manifolds for singular delay equations have been obtained in~\cite{GVR21} and center manifolds in~\cite{KN22,KN21}.  Moreover, first results on a probabilistic construction of random dynamical systems for SPDEs are available in~\cite{G22}. Besides,  the framework of mean-square RDS \cite{Kloeden20121422} is also practicable to the case of general multiplicative noise, under which the theory of mean-square unstable manifolds  and stable invariant sets have been  developed in \cite{Wang2021,LI2022} for densely and non-densely defined linear operators.

In contrast to unstable or center manifolds, {stable manifolds} for SPDEs are {\em infinite-dimensional} objects, therefore their construction is more involved. Their existence was still an open problem for SPDEs driven by fBm with $H>\frac{1}{2}$ since the celebrated work  \cite{GarridoAtienza2010}, in which the Lyapunov-Perron method is applied, the associated Lyapunov-Perron operator is related to the backward orbit. The existence of a fixed point for this operator implies  the existence of a backward solution.  It has been proved in \cite{GLS2010} that the forward solution belongs to a subspace of $D((-A)^\alpha)$ for $\alpha>0$ (see Subsection~\ref{assumptions} for a precise statement)
and it is also H\"older continuous on the state space. 
For stable manifolds, this construction suggests that the backward solution is H\"older regular on the state space and should also belong to this domain 
for a certain $\alpha>0$. 

However, one cannot seek stable manifolds in such a space, because the Lyapunov-Perron operator does not give any information on the backward solution, see Subsection~\ref{assumptions} for more details. Therefore one cannot a-priori make a similar ansatz as for unstable manifolds. 

To overcome this difficulty, inspired by \cite{GH2019,GHN21}, one of the main novel contributions of this work is to introduce the function space $C^{\beta}_{-\beta}([0,T];\mathcal{B})$. This contains continuous mappings with values in the state space, which we denote by $\cB$, and $\beta$-H\"older continuous mappings with values in $\cB_{-\beta}$, where $\beta>\frac{1}{2}$. Here $\{\cB_\gamma\}_{\gamma\in\R}$ is a monotone family of interpolation spaces associated to the linear part of the equation, as discussed in~\cite{GHN21}. A similar approach combined with rough paths tools has been followed in~\cite{KN22} to construct center manifolds for~\eqref{spde:intro}, where the Hurst index of the fBm $H\in(\frac{1}{3},\frac{1}{2}]$. Even though we deal with the case $H>\frac{1}{2}$, the fBm considered in this manuscript is infinite-dimensional, in contrast to~\cite{KN22}. Moreover, we only assume that the covariance operator $Q$ of the fBm is of trace-class. This assumption is less restrictive than the condition imposed in~\cite{GarridoAtienza2010} ($\text{tr}Q^{1/2}<\infty$), which constructs unstable manifolds for~\eqref{spde:intro}. 
In order to establish the existence of stable manifolds, we first show the global existence of a solution  and therefore obtain by common methods an RDS. 
Putting these deliberations together we construct a Lyapunov-Perron type operator~\cite{GarridoAtienza2010,KN22}. Under a certain spectral gap condition, this operator is proved to have a unique fixed point in the function space $C^{\beta}_{-\beta}([0,T];\cB)$. This fixed point is further used to characterize the stable manifold. Due to the nonlinear diffusion coefficient $G$ and in order to control the growth of the noise, a truncation argument is additionally required for the fixed point argument. 
In conclusion our results provide a positive answer to the Lu-Schmalfu{\ss} conjecture \cite{GarridoAtienza2010,Lu2011}. {The pathwise integrals occurring in the Lyapunov-Perron method are constructed using fractional calculus as in~\cite{Chen2014}.~Nevertheless one should obtain similar results using Young's integral, see Remark~\ref{rem:young}.}
{Moreover we believe that these results can be generalized to the rough case, i.e.~$H\in(\frac{1}{3},\frac{1}{2}]$ using rough paths theory. The methods presented in this work combined with the tools developed in~\cite{KN22} are expected to entail the existence of stable manifolds provided that the driving noise is finite-dimensional. For infinite-dimensional trace-class noise, as considered here, similar arguments could be conducted  based on the results obtained in~\cite{HN20}. This case is technically more involved and therefore postponed to another work}. \\
 
This work is structured as follows. In Section \ref{sec2}, we introduce basic concepts on RDS, Hilbert space-valued fBm and a pathwise integral with respect to this process via fractional calculus. Furthermore, we formulate the equation we consider and prove the global well-posedness and the generation of an RDS (Theorem~\ref{th-exits-solu}). Section \ref{se3} contains our main results. In Theorem~\ref{thm:manifold} we prove the existence of a local stable manifold for~\eqref{spde:intro} in $C^{\beta}_{-\beta}([0,T];\mathcal{B})$. Under additional assumptions on the coefficients, the smoothness of this manifold is established in Subsection~\ref{smoothness}. Assertions regarding the smoothness of invariant manifolds have been derived in~\cite{DLS2004, LL2010} for linear multiplicative (Brownian) noise. To our best knowledge, this is the first work that provides a complete proof of the smoothness of random invariant manifolds for SPDEs with nonlinear multiplicative fractional noise.

\section{Preliminaries}
\label{sec2}

In this section, we present some basic concepts on RDS, fBm and the construction of the integral concerning $\beta$-H\"older continuous integrator from \cite{Arnold1998,Castaing1977,MV1968,Zahle1998,GarridoAtienza2010} and references therein. For the sake of completeness,  we also  establish the  well-posedness of SPDEs with fractional noise in a different function space as in \cite{NR2002,Chen2014,GLS2010}. Some  key estimates for the computations of stable manifolds will be provided in this setting. 

\subsection{Random dynamical systems}
 Given a metric dynamical system  $(\Omega,\mathcal{F},\mathbb{P},\{\theta_t\}_{t\in\mathbb{R}})$, a function $X:\Omega\rightarrow \mathbb{R}$ is called a random variable if it is $(\mathcal{F};\mathscr{B}(\mathbb{R}))$-measurable. It is  tempered if 
\begin{equation*}
\lim_{t\to\pm\infty}\frac{\log X(\theta_{t}\omega)}{t}=0,~~\mathbb{P}\text{-almost surely.}
\end{equation*}
In particular, it is tempered from above if 
\begin{equation*}
\lim_{t\to\pm\infty}\frac{\log^+ X(\theta_{t}\omega)}{t}=0,~~\mathbb{P}\text{-almost surely,}
\end{equation*}
and tempered from below if $1/X$ is tempered from above.~{Here we recall that $\log^+ x:=\max\{\log x, 0\}$.}
Let $(\mathcal{B}, |\cdot|, \langle\cdot,\cdot\rangle)$ be a separable Hilbert space. A  set-valued mapping $M:\Omega\rightarrow 2^{\mathcal{B}}\setminus\emptyset$, $\omega\mapsto M(\omega)$ is said to be a random set if $M(\omega)$ is a closed set and $\omega\mapsto \inf_{y\in M(\omega)}|x-y|$ is a random variable for each $x\in \mathcal{B}$.
\begin{definition}
A mapping $\varphi:\mathbb{R}^+\times\Omega\times \mathcal{B}\to \mathcal{B}$ defines a  random dynamical system  if
	\begin{itemize}
		\item [$(\mathrm{i})$] $\varphi$ is $(\mathscr{B}(\mathbb{R}^{+})\otimes\mathcal{F}\otimes\mathscr{B}(\mathcal{B}),\mathscr{B}(\mathcal{B}))$-measurable;
		\item [$(\mathrm{ii})$] the mapping $\varphi(t,\omega):=\varphi(t,\omega,\cdot): \mathcal{B}\to \mathcal{B}$ forms a cocycle over $\{\theta_t\}_{t\in\mathbb{R}}$ meaning that
		\begin{equation*}
		\begin{aligned}
		\varphi(0,\omega)&=\text{\rm{id}}_\mathcal{B}, \quad \text{ for all }\omega\in\Omega,\\
		\varphi(t+s,\omega)&=\varphi(t,\theta_s\omega)\circ\varphi(s,\omega),
		\quad \text{ for all }t, s \in \mathbb{R}^{+} \text{~and all~}  \omega\in\Omega.
		\end{aligned}
		\end{equation*}
	\end{itemize}
	\end{definition}
	\begin{definition}\label{def3.1}
		A random set $M$ is called an invariant set for $\varphi$ if
	\begin{equation*}
	\varphi(t,\omega,M(\omega))\subset M(\theta_t\omega) \text{~for~} t\ge0.
	\end{equation*} It is called a  random Lipschitz manifold  if it can be characterized by the graph of a Lipschitz mapping, meaning that there exists a function $m$ such that
\begin{align}
M(\omega)=\{x\in \mathcal{B} : x=\xi+m(\xi,\omega)\}.
\end{align}
Here $m(\cdot,\omega):\mathcal{B}^-\to \mathcal{B}^+$ is Lipschitz continuous and $\mathcal{B}^-$ and $ \mathcal{B}^+$ are unstable and stable subspaces as defined in Section~\ref{assumptions}.
	\end{definition}
	\begin{definition}\label{def3.4}
	A random  manifold $M(\omega)$ is called a  local stable manifold at zero for $\varphi$  if $m$ is defined in a  neighborhood $U\subset \mathcal{B}^-$ of zero and there exists a neighborhood $W(\omega)\subset \mathcal{B}$ of zero such that
	\begin{itemize}
		\item [$(1)$] for $x\in M(\omega)\cap W(\omega)$,
	\begin{equation*}
	\lim_{t\to+\infty}\varphi(t,\omega,x)=0 \quad\text{exponentially},
	\end{equation*}
		\item [$(2)$] 
		\begin{equation*}
	\lim_{|x|\to 0}t_0(\omega,x)=\infty, 
	\end{equation*}
where $x\in M(\omega)$ and 
\begin{align*}
    t_0(\omega,x)=\inf\{t\in\mathbb{R}^+:\varphi(t,\omega,x)\notin M(\theta_{t}\omega)\}.
\end{align*}
		\end{itemize}
\end{definition}
\subsection{Hilbert space-valued fractional Brownian motion and stochastic integral}
\label{sec2.2}
The two-sided one-dimensional fBm $(\beta^H(t))_{t\in\mathbb{R}}$, with Hurst parameter $0<H<1$ is a continuous centered Gaussian process with covariance  
\begin{equation*}
\mathbb{E}\left[\beta^H(t)\beta^H(s)\right]=\frac{1}{2}\left(|s|^{2H}+|t|^{2H}-|t-s|^{2H}\right),\quad\text{for}~s,t\in\mathbb{R}.
\end{equation*}
In order to define a Hilbert space-valued fbm, we describe its covariance operator. We denote by $\{e_i\}_{i\in\mathbb{N}}$  the complete orthonormal basis of the separable Hilbert space $\cB$ and
consider a bounded symmetric, trace-class operator $Q$ on $\cB$. This means that there exists a sequence of non-negative constants $\{\mu_i\}_{i\in\mathbb{N}}$ such that $Qe_i=\mu_i e_i$ and the trace $\text{tr}Q=\sum_{i=1}^{\infty}\mu_i<\infty$. Thus one can define a $\mathcal{B}$-valued fBm with Hurst parameter $H$ as
\begin{equation*}
B^H(t)=\sum_{i=1}^\infty\sqrt{\mu_i}e_i\beta_i^H(t),~~t\in\mathbb{R},
\end{equation*}
where $\{\beta_i^H(t)\}_{i\in\mathbb{N}}$ is a sequence of stochastically independent one-dimensional fBms. We next construct a
canonical probability space associated to $B^H$. Denote by $\Omega=C_{0}(\mathbb{R};\mathcal{B})$ the space of continuous functions $\omega:\mathbb{R}\to \mathcal{B}$ such that $\omega(0)=0$, equipped with the compact open topology, $\mathcal{F}$ represents the associated Borel-$\sigma$-algebra of $\Omega$ and $\mathbb{P}$ is the Gaussian measure generated by the two-sided fractional Brownian motion $B^H$. On the probability space $(\Omega, \mathcal{F}, \mathbb{P})$, one can introduce $\{\theta_t\}_{t\in\mathbb{R}}$ the flow of Wiener shifts, given by $\theta_{t}\omega(\cdot)=\omega(\cdot+t)-\omega(\cdot)$. Then it follows from \cite[Corollary 1]{GS2011} that the quadruple $(\Omega,\mathcal{F},\mathbb{P},\{\theta_t\}_{t\in\mathbb{R}})$ is an ergodic metric dynamical system. 

Since $B^H$ has a $\beta'$-H\"{o}lder continuous version on any compact interval for $\beta'<H$, there exists a subset $\Omega'$ with full measure satisfying that for every interval $[-n,n]$, $n\in\mathbb{N}$, $$\vertiii{\omega}_{\beta',-n,n}=\sup_{-n\leq s<t\leq n}\frac{|\omega(t)-\omega(s)|}{(t-s)^{\beta'}}< \infty.$$

Note that it was proved in \cite[Lemma 15]{Chen2014} that $\Omega'$ is $\{\theta_t\}_{t\in\mathbb{R}}$-invariant. In the following we restrict this ergodic metric dynamical system to $\Omega'$ and denote this restriction using the old symbols $(\Omega,\mathcal{F},\mathbb{P},\{\theta_t\}_{t\in\mathbb{R}})$.

We now introduce a suitable stochastic integral with respect to $\omega$. To do this, we first recall some tools from fractional calculus (consult the monograph \cite{SKM1993} for more details).
Let $L_2(\mathcal{B};\mathcal{B})$ be the space of Hilbert-Schmidt operators from $\mathcal{B}$ to $\mathcal{B}$ endowed with the norm
\begin{equation*}
  \|L\|_{L_2(\mathcal{B};\mathcal{B})}^2=\sum_{i=1}^\infty |Le_i|^2,
\end{equation*}
for $L\in L_2(\mathcal{B};\mathcal{B})$.  It is well-known that this is a separable Hilbert space.
For $0<\alpha<1$, let $\mathcal{H}_1,\mathcal{H}_2$ be two separable Hilbert spaces, $f:[0,T]\rightarrow \mathcal{H}_1$ and $g:[0,T]\rightarrow \mathcal{H}_2$ be
the H\"{o}lder continuous functions with the exponents $\alpha$ and $1-\alpha$, respectively. Then the Weyl right and left hand side fractional derivatives
are defined as follows:
\begin{align*}
D_{0+}^{\alpha}f[t]&=\frac{1}{\Gamma(1-\alpha)}\left(\frac{f(t)}{t^{\alpha}}+\alpha\int_{0}^{t}\frac{f(t)-f(q)}
{(t-q)^{1+\alpha}}dq\right)\in \mathcal{H}_1,\\
D_{T-}^{1-\alpha}g_{T-}[t]&=\frac{(-1)^{1-\alpha}}{\Gamma(\alpha)}\left(\frac{g(t)-g(T-)}{(T-t)^{1-\alpha}}
+(1-\alpha)\int_{t}^{T}\frac{g(t)-g(q)}{(q-t)^{2-\alpha}}dq\right)\in \mathcal{H}_2,
\end{align*}
where $g(T-)$ is the left limit of $g$ at $T$. We now take $\mathcal{H}_1=L_2(\mathcal{B};\mathcal{B})$ and $\mathcal{H}_2=\mathcal{B}$ and assume $1-\beta'<\alpha<\beta$, $Z\in C^\beta([0,T];L_2(\mathcal{B};\mathcal{B}))$ and
$\omega\in C^{\beta'}([0,T];\mathcal{B})$ such that
\begin{equation*}
  {t}\mapsto \left\|D_{0+}^{\alpha}Z[t]\right\|_{L_2(\mathcal{B};\mathcal{B})}\left|D_{T-}^{1-\alpha}\omega_{T-}[t]\right|\in L^1([0,T];\mathbb{R}).
\end{equation*}
Then it was proved in \cite[Lemma 3]{Chen2014} that for $0\le t \le T$ the integral
\begin{equation}\label{eq2.2}
\int_{0}^{T}Z(t)d\omega(t)=(-1)^\alpha\sum_{j\in\mathbb{N}}
\left(\sum_{i\in\mathbb{N}}\int_{0}^{T}D_{0+}^{\alpha}\langle e_j,Z(\cdot)e_i\rangle[t]D_{T-}^{1-\alpha}\langle e_i,\omega(\cdot)\rangle[t]dt\right)e_j
\end{equation}
is well-defined. Also, the estimate of its norm follows
\begin{align}\label{eq2.3}
\left|\int_{0}^{T}Z(t)d\omega(t)\right|_{\cB}\leq\int_{0}^{T}\left\|D_{0+}^{\alpha}Z[t]\right\|_{L_2(\mathcal{B};\mathcal{B})}
\left|D_{T-}^{1-\alpha}\omega_{T-}[t]\right|dt,
\end{align}
and the additivity of such integral yields
\begin{equation}\label{eq2.4}
\int_{0}^{T}Z(t)d\omega(t)=\int_{-\tau}^{T-\tau}Z(t+\tau)d\theta_\tau\omega(t), \qquad \forall~ \tau\in\mathbb{R}.
\end{equation}
\begin{remark}
Note that the trace-classs assumption on the noise is less restrictive than the condition $\text{tr}Q^{1/2}<\infty$ imposed in~\cite{GarridoAtienza2010}. {Alternatively, one can use Young's integral or rough paths theory as in~\cite{HN20} to define~\eqref{eq2.2}. }Here we use for simplicity fractional calculus relying on the definition of the stochastic integral introduced in~\cite{Chen2014}.  
\end{remark}

\subsection{Equation of study}\label{assumptions}
We consider the following  stochastic evolution equation on the separable Hilbert space $\cB$
\begin{equation}\label{eq2.5}
du(t)=[Au(t)+F(u(t))]dt+G(u(t))d\omega(t),
\end{equation}
where $\omega$ denotes the infinite dimensional fBm  with Hurst index $H>\frac{1}{2}$ and the stochastic integral with respect to $\omega$ is interpreted in the sense of (\ref{eq2.2}). We impose the following assumptions:
\begin{itemize}
  \item [$(\bf A_1$)] $A$ is  the generator of an analytic semigroup $S(\cdot)$ on a monotone family of interpolation spaces $\{\mathcal{B}_{\theta}\}_{\theta\in\mathbb{R}}$. {We recall that 
  a family of separable Banach spaces $(\cB_\theta,|\cdot|_\theta)_{\theta\in\mathbb{R}}$ is called a monontone family of interpolation spaces if for $\beta_1\leq \beta_2$, the space $\cB_{\beta_2}\subset \cB_{\beta_1}$ with dense and continuous embedding and the following interpolation inequality holds for $\theta\leq \beta\leq \gamma$ and $x\in  \cB_{\gamma}$:
\begin{align}\label{interpolation:ineq}
    |x|^{\gamma-\theta}_\beta \lesssim |x|^{\gamma-\beta}_\theta |x|^{\beta-\theta}_\gamma.
\end{align}
  We set $\cB_0:=\cB$. } Furthermore the spectrum of $A$ is discrete and, in particular,  does not contain $0$. The eigenvalues of $A$ are real and can be ordered as $\lambda_1>\lambda_2>\cdots>\lambda_n\cdots$, and $\lim_{n\to\infty}\lambda_n=-\infty$.
  \item [$(\bf A_2$)] The mapping $F\in C^1_b(\cB;\cB)$ with $F(0)=DF(0)=0$ and there exists a constant $L_F>0$ such that
\begin{align}
&|F(u)-F(v)|\leq L_F|u-v|~\text{ for all } u,v\in\cB\label{F2}.
\end{align}
  \item [$(\bf A_3$)] The mapping $G\in C^2_b(\mathcal{B};L_2(\mathcal{B};\mathcal{B}))$ with $G(0)=DG(0)=0$,~i.e. is a twice continuously Fr\'{e}chet-differentiable  operator with bounded first derivative $DG$ and second derivative $D^2G$.  There also exists a constant $L_G>0$ such that
\begin{align}
&\|G(u)-G(v)\|_{L_2(\mathcal{B};\mathcal{B})}\leq L_G|u-v|_{-\beta}~\text{ for all } u,v\in\cB\label{G0}\\
&\|DG(u)h-DG(v)h\|_{L_2(\mathcal{B};\mathcal{B}_{-\beta})}\leq L_G|u-v|_{-\beta}|h|_{-\beta},~\text{ for all } u,v,h\in\cB.\label{G4}
\end{align}
  \item [$(\bf A_4$)] $\frac{1}{2}<\beta<\beta'<H<1$ and $1-\beta'<\alpha<\beta$.
\end{itemize}
Note that due to (\ref{G4}) we have for $u_1,v_1,u_2,v_2\in\mathcal{B}$ that
\begin{align*}
&\|G(u_1)-G(v_1)-G(u_2)+G(v_2)\|_{L_2(\mathcal{B};\mathcal{B}_{-\beta})}\\
&\quad\leq L_G(|u_1-u_2|_{-\beta}+|v_1-v_2|_{-\beta})|u_1-v_1|_{-\beta}
+(|u_2|_{-\beta}+|v_2|_{-\beta})|u_1-v_1-u_2{\color{blue}+}v_2|_{-\beta}.
\end{align*}
\begin{remark}Note that $\mathcal{B}$ is continuously embedded into $\mathcal{B}_{-\beta}$, therefore we can always  find a constant  $C>0$ such that $|u|_{-\beta}\leq C|u|$. Hence, there exists constants, still denoted by $L_F$ and $L_G$ for the  simplicity of notation, such that 
$|F(u)-F(v)|_{-\beta}\leq L_F|u-v|.$
Similarly for $G$ we have \begin{align}
&\|G(u)-G(v)\|_{L_2(\mathcal{B};\mathcal{B})}\leq L_G|u-v|,\label{G1}\\
&\|G(u)-G(v)\|_{L_2(\mathcal{B};\mathcal{B}_{-\beta})}\leq L_G|u-v|,\label{G2}\\
&\|G(u)-G(v)\|_{L_2(\mathcal{B};\mathcal{B}_{-\beta})}\leq L_G|u-v|_{-\beta}.\label{G3}
\end{align}
\end{remark}
\begin{remark}
Under weaker assumptions on $G$, the global existence (and not uniqueness) of solutions for~\eqref{eq2.5} was proved in \cite[Theorem 8, Lemma 10]{GSV2019}.
\end{remark}
For a better comprehension, we discuss some important consequences of the imposed assumptions. First, based on $(\bf A_1$) we can decompose the space $\mathcal{B}$ into a direct sum of closed invariant subspaces:
\begin{equation*}
  \mathcal{B}=\mathcal{B}^+\oplus \mathcal{B}^-,
\end{equation*}
where the finite dimensional linear subspace $\mathcal{B}^+$ is spanned by the eigenvectors with eigenvalues larger than 0, and $\mathcal{B}^-$ is the linear subspace spanned by the eigenvectors with eigenvalues less than 0. Let us denote by $\pi^{\pm}$ the orthogonal projections associated with this splitting, and $S^{\pm}$ the restriction of the semigroup $S$ to $\mathcal{B}^{\pm}$. In other words, $S^+$ is generated by the bounded operator $\pi^+A$ on $\mathcal{B}^+$ and $S^-$ is generated by the strictly negative operator $\pi^-A$ on $\mathcal{B}^-$. The projections commute with the semigroup, namely it holds that $\pi^{\pm}S=S\pi^{\pm}$ and there exist positive constants $\hat{\mu}$ and $c_S$ and a negative constant $\check{\mu}$ such that
\begin{subequations}
  \begin{align}
  \|S^+(t)x\|&\le c_S e^{\hat{\mu}t}\|x\|,\quad t\le0,~x\in \mathcal{B},\label{eq.S-posi}\\
  \|S^-(t)x\|&\le c_S e^{\check{\mu}t}\|x\|,\quad t\ge0,~x\in \mathcal{B}.\label{eq.S-nega}
\end{align}
\end{subequations}
Here $\hat{\mu}$ can be any positive number less than the smallest positive eigenvalue of $A$ and  $\check{\mu}$ can be any negative number larger than the largest negative eigenvalue of $A$. Moreover {it follows for e.g. from~\cite[p.~118]{Triebel}} that for any $\sigma\in\mathbb{R}$, the interpolation space $\mathcal{B}_\sigma$ admits a  decomposition given by $\mathcal{B}_\sigma=\mathcal{B}^+\oplus \mathcal{B}^-_{\sigma}$, where $\mathcal{B}^-_{\sigma}=\mathcal{B}_{\sigma}\cap \mathcal{B}^-$.  {We refrain from giving an additional index to the finite dimensional space $\mathcal{B}^+$.~We recall that $\mathcal{B}_{0}=\mathcal{B}$ and $\mathcal{B}_{\sigma_2}\subset\mathcal{B}_{\sigma_1}$ for $\sigma_1\leq \sigma_2$ with dense and continuous embeddings}. Let $L(\mathcal{B}_{\eta};\mathcal{B}_{\gamma})$ be the space of continuous linear operators from $\mathcal{B}_{\eta}$ to $\mathcal{B}_{\gamma}$. Thanks to the analyticity of the semigroup, it follows from \cite{Pazy1983} that for $0\le\sigma \le1$, $x\in\mathcal{B}_{\gamma+\sigma}$
\begin{equation}\label{eq2.7}
|S(t)x|_{\gamma+\sigma}\leq c_St^{-\sigma}|x|_{\gamma},
\end{equation}
and
\begin{equation}\label{eq2.8}
|(S(t)-\text{id})x|_{\gamma}\leq c_St^{\sigma}|x|_{\gamma+\sigma}.
\end{equation}
It is straightforward from the above inequalities that
\begin{equation}\label{eq2.9}
\|S(t-r)-S(t-\tau)\|_{L(\mathcal{B}_\eta,\mathcal{B}_\gamma)}\leq c_S(r-\tau)^\nu (t-r)^{-\nu-\gamma+\eta},\quad \text{for~} 0\le\eta-\nu\le \gamma,
\end{equation}
and
\begin{equation}\label{eq2.10}
\|S(t-r)-S(s-r)-S(t-\tau)+S(s-\tau)\|_{L(\mathcal{B};\mathcal{B})}\leq c_S(t-s)^{\kappa}(r-\tau)^{\iota}(s-r)^{-(\kappa+\iota)},
\end{equation}
for $\kappa,\iota\ge0$, {$\tau\leq r\leq s\leq t$}. 

We interpret \eqref{eq2.5} in the mild form
\begin{equation}\label{mild sol}
u(t)=S(t)u_0+\int_0^t S(t-\tau)F(u(\tau))d\tau+\int_0^t S(t-\tau)G(u(\tau))d\omega(\tau),
\end{equation}
where the initial condition $u(0)=u_0\in \mathcal{B}$.
Existence of solutions to  \eqref{eq2.5} was first studied in \cite{Chen2014} (see also \cite[Theorem 3.2]{Gao2014}) in $C^{\beta,\sim}([0,T];\mathcal{B})$, the
space of continuous functions $u:[0,T]\to \mathcal{B}$ such that
\begin{equation*}
\|u\|_{\beta,\sim,[0,T]}=\sup_{t\in[0,T]}|u(t)|+\sup_{0\leq s<t\leq T}s^{\beta}\frac{|u(t)-u(s)|}{(t-s)^{\beta}}.
\end{equation*}
If the initial value $u_0\in \mathcal{B}_{\beta}$,  the solution $u\in C^{\beta}([0,T];\mathcal{B})$.

In \cite{GarridoAtienza2010}, working with the scale $\cB_\beta:=D(\pi^+id_\cB+\pi^-(-A\pi^-)^\beta)$ (equipped with the graph norm $|x|_{\cB_\beta}:=|\pi^+ x| + |\pi^{-}(-A\pi^{-})^\beta x|$), the unstable invariant manifold is obtained  in $\mathcal{B}_{\delta}$ for $\delta\in[0,1-\xi),~\xi\in [\alpha,1-\alpha)$. As already stated, the operator $A\pi^{-}$ is strictly negative {definite} on $\cB^{-}$ ($A$ itself is not, due to the assumed spectral decomposition). Therefore it makes sense to construct its fractional powers $(-A\pi^{-})^\beta$ for $\beta>0$.
Further, we remark that it is reasonable to have this prediction on the manifold (i.e.~that it should belong to $\cB_\delta)$ before even obtaining it, because if the backward orbit exists, that is for any $u_0\in \mathcal{B}$, $t>0$, there is a $u_{-t}\in \mathcal{B}$ such that $\varphi(t,\theta_{-t}\omega,u_{-t})=u_0$, the fact that $\varphi(t,\theta_{-t}\omega,u_{-t})\in \mathcal{B}_{\delta},\delta\in[0,1-\xi),~\xi\in [\alpha,1-\alpha)$ implies $u_0\in \mathcal{B}_{\delta}$. The condition of $u_0\in \mathcal{B}_{\delta}$ provides a possibility to obtain the fixed point of 
 the Lyapunov-Perron operator. For stable manifolds,  one cannot  a-priori {prescribe} that the initial conditions are in  $\mathcal{B}_{\delta}$, which leads to the possibility of non-convergence of the Lyapunov-Perron operator. To overcome this problem, we introduce a modified space $C^{\beta}_{-\beta}([0,T];\mathcal{B})$ endowed with the norm 
\begin{equation*}
\|u\|_{\beta,-\beta,[0,T]}=\sup_{t\in[0,T]}|u(t)|+\sup_{0\leq s<t\leq T}\frac{|u(t)-u(s)|_{-\beta}}{(t-s)^{\beta}},
\end{equation*}
and study the stable manifold in this function space. For the sake of completeness, we will prove the existence of solutions of~\eqref{eq2.5} in $C^{\beta}_{-\beta}([0,T];\mathcal{B})$. To this aim and in order to get global-in-time existence, we introduce an equivalent norm on $C^{\beta}_{-\beta}([0,T];\mathcal{B})$ given by 
\begin{align}\label{norm:c}
\|u\|_{\beta,\rho,-\beta,[T_1,T_2]}=\sup_{t\in[T_1,T_2]}e^{-\rho (t-T_1)}|u(t)|+\sup_{T_1\leq s<t\leq T_2}\frac{e^{-\rho(t-T_1)}|u(t)-u(s)|_{-\beta}}{(t-s)^{\beta}},~~\rho\geq0.
\end{align}
\begin{remark}\label{rem:young}
    Note that our space $C^\beta_{-\beta}([0,T];\cB)=C([0,T];\cB)\cap C^{\beta}([0,T];\cB_{-\beta})$. This is consistent with the approach in~\cite{GHN21} for the Young integral. This can be defined for a process $y\in C([0,T];\cB)\cap C^\beta([0,T];\cB_{-\beta})$ (if $\omega$ is $\beta$-H\"older continuous with $\beta>\frac{1}{2})$) as 
    \begin{align*}
    \int_0^t S(t-r) y(r)d\omega(r) = \lim\limits_{|\mathcal{P}|\to 0}\sum\limits_{[u,v]\in\mathcal{P}} S(t-u) y(u) (\omega(v) - \omega(u))
    \end{align*}
    and estimated by 
    \begin{align*}
    \Big | \int_s^t S(t-r) y(r)d\omega(r) - y(s)(\omega(t)-\omega(s)) \Big|_{-\beta} \lesssim (t-s)^{2\beta} \|y\|_{C^\beta_{-\beta}([0,T];\cB)}\vertiii{\omega}_{\beta,0,T}.
    \end{align*}
\end{remark}
\subsection{Existence and uniqueness of a global solution}

We prove that the SPDE~\eqref{eq2.5} has a unique global-in-time solution in the function space $C^{\beta}_{-\beta}([0,T];\mathcal{B})$, which generates a random dynamical system. For our aims, we begin with some preliminary considerations for the fractional integral and present several estimates of the stochastic convolution required in~\eqref{mild sol}. 
Based on the definition of the fractional derivative, on the assumptions on the nonlinear term $G$ and regarding~\eqref{norm:c} we obtain 
the following preliminary lemmas.
\begin{lemma}
For $\beta>\alpha$ we have 
\begin{align*}
   \|D_{s+}^{\alpha}S(t-\cdot)G(u(\cdot))[r]\|_{L_2(\mathcal{B};\mathcal{B}_{-\beta})}  \leq c_SC_{\alpha,\beta}L_G\|u\|_{\beta,\rho,-\beta}[(r-s)^{-\alpha}+(r-s)^{\beta-\alpha}]e^{\rho r}
\end{align*}
as well as 
\begin{align*}
&\|D_{s+}^{\alpha}S(t-\cdot)G(u(\cdot))[r]\|_{L_2(\mathcal{B};\mathcal{B})} \leq c_SC_{\alpha,\beta}L_G\|u\|_{\beta,\rho,-\beta}[(r-s)^{-\alpha}+(t-r)^{-\beta}(r-s)^{\beta-\alpha}]e^{\rho r}.
\end{align*}
\end{lemma}

\begin{proof}
    By immediate computations we derive 
    \begin{align*}
&\|D_{s+}^{\alpha}S(t-\cdot)G(u(\cdot))[r]\|_{L_2(\mathcal{B};\mathcal{B}_{-\beta})}\\
 &\quad\leq \frac{1}{\Gamma(1-\alpha)}\bigg(\frac{\|S(t-r)G(u(r))\|_{L_2(\mathcal{B};\mathcal{B}_{-\beta})}}{(r-s)^{\alpha}}\\
 &\qquad+\alpha\int_{s}^{r}\frac{\|S(t-r)G(u(r))-S(t-q)G(u(q))\|_{L_2(\mathcal{B};\mathcal{B}_{-\beta})}}{(r-q)^{\alpha+1}}dq\bigg)\\
&\quad\leq c_SC_{\alpha}L_G|u(r)|(r-s)^{-\alpha}+C_{\alpha}\int_{s}^r\frac{\|S(t-r)-S(t-q)\|_{L(\mathcal{B};\mathcal{B}_{-\beta})}\|G(u(r))\|_{L_2(\mathcal{B};\mathcal{B})}}{(r-q)^{1+\alpha}}dq\\
&\qquad+C_{\alpha}\int_{s}^r\frac{\|S(t-q)\|_{L(\mathcal{B}_{-\beta};\mathcal{B}_{-\beta})}\|G(u(r))-G(u(q))\|_{L_2(\mathcal{B};\mathcal{B}_{-\beta})}}{(r-q)^{1+\alpha}}dq\\
&\quad\leq c_SC_{\alpha}L_G(r-s)^{-\alpha}e^{\rho r}\|u\|_{\beta,\rho,-\beta}+c_SC_{\alpha}L_G\|u\|_{\beta,\rho,-\beta}\int_{s}^r(r-q)^{\beta-\alpha-1}e^{\rho r}dq\\
&\quad\leq c_SC_{\alpha,\beta}L_G\|u\|_{\beta,\rho,-\beta}[(r-s)^{-\alpha}+(r-s)^{\beta-\alpha}]e^{\rho r},
\end{align*}
which proves the first statement. 
Similarly, taking the $\|\cdot\|_{L_2(\mathcal{B};\mathcal{B})}$-norm and using \eqref{G1}--\eqref{G3}, we obtain the second statement
\begin{align*}
&\|D_{s+}^{\alpha}S(t-\cdot)G(u(\cdot))[r]\|_{L_2(\mathcal{B};\mathcal{B})}\\
&\leq \frac{1}{\Gamma(1-\alpha)}\left(\frac{\|S(t-r)G(u(r))\|_{L_2(\mathcal{B};\mathcal{B})}}{(r-s)^{\alpha}}+\alpha\int_{s}^{r}\frac{\|S(t-r)G(u(r))-S(t-q)G(u(q))\|_{L_2(\mathcal{B};\mathcal{B})}}{(r-q)^{\alpha+1}}dq\right)\\
&\leq c_SC_{\alpha}|u(r)|(r-s)^{-\alpha}+C_{\alpha}\int_{s}^r\frac{\|S(t-r)-S(t-q)\|_{L(\mathcal{B};\mathcal{B})}\|G(u(r))\|_{L_2(\mathcal{B};\mathcal{B})}}{(r-q)^{1+\alpha}}dq\\
&\quad+C_{\alpha}\int_{s}^r\frac{\|S(t-q)\|_{L(\mathcal{B}_{-\beta};\mathcal{B})}\|G(u(r))-G(u(q))\|_{L_2(\mathcal{B};\mathcal{B}_{-\beta})}}{(r-q)^{1+\alpha}}dq\\
&\leq c_SC_{\alpha}L_G(r-s)^{-\alpha}e^{\rho r}\|u\|_{\beta,\rho,-\beta}+c_SC_{\alpha}L_G\|u\|_{\beta,\rho,-\beta}(t-r)^{-\beta}\int_{s}^r(r-q)^{\beta-\alpha-1}e^{\rho r}dq\\
&\quad+c_SC_{\alpha}L_G\|u\|_{\beta,\rho,-\beta}\int_{s}^re^{\rho r}{(t-q)^{-\beta}}(r-q)^{\beta-\alpha-1}dq\\
&\leq c_SC_{\alpha,\beta}L_G\|u\|_{\beta,\rho,-\beta}[(r-s)^{-\alpha}+(t-r)^{-\beta}(r-s)^{\beta-\alpha}]e^{\rho r}.
\end{align*}
\end{proof}

Taking additionally (\ref{G4}) into account, we analogously obtain for the fractional derivative of the convolution of the difference of two terms the following estimates. 
\begin{lemma}
    Let $\beta>\alpha$. The following estimates hold true 
    \begin{align*}
      &  \|D_{s+}^{\alpha}S(t-\cdot)[G(u(\cdot))-G(v(\cdot))][r]\|_{L_2(\mathcal{B};\mathcal{B}_{-\beta})} \\ &\leq c_SC_{\alpha,\beta}L_G(1+\|u\|_{\beta,-\beta}+\|v\|_{\beta,-\beta})\|u-v\|_{\beta,\rho,-\beta}[(r-s)^{-\alpha}+(r-s)^{\beta-\alpha}]e^{\rho r}
    \end{align*}
    and 
\begin{align*}
&\|D_{s+}^{\alpha}S(t-\cdot)[G(u(\cdot))-G(v(\cdot))][r]\|_{L_2(\mathcal{B};\mathcal{B})}\\
&\leq c_SC_{\alpha,\beta}L_G(1+\|u\|_{\beta,-\beta}+\|v\|_{\beta,-\beta})\|u\|_{\beta,\rho,-\beta}
[1+(t-r)^{-\beta}(r-s)^{\beta}](r-s)^{-\alpha}e^{\rho r}.
\end{align*}
\end{lemma}
 \begin{proof}
Using \eqref{G4} we compute 
\begin{align*}
&\|D_{s+}^{\alpha}S(t-\cdot)[G(u(\cdot))-G(v(\cdot))][r]\|_{L_2(\mathcal{B};\mathcal{B}_{-\beta})}\\
&\quad\leq \frac{1}{\Gamma(1-\alpha)}\bigg(\frac{\|S(t-r)[G(u(r))-G(v(r))]\|_{L_2(\mathcal{B};\mathcal{B}_{-\beta})}}{(r-s)^{\alpha}}\\
&\qquad+\alpha\int_{s}^{r}\frac{\|S(t-r)[G(u(r))-G(v(r))]-S(t-q)[G(u(q))-G(v(q))]\|_{L_2(\mathcal{B};\mathcal{B}_{-\beta})}}{(r-q)^{\alpha+1}}dq\bigg)\\
&\quad\leq c_SC_{\alpha}|u(r)-v(r)|(r-s)^{-\alpha}\\
&\quad+C_{\alpha}\int_{s}^r\frac{\|S(t-r)-S(t-q)\|_{L(\mathcal{B};\mathcal{B}_{-\beta})}\|G(u(r))-G(v(r))\|_{L_2(\mathcal{B};\mathcal{B})}}{(r-q)^{1+\alpha}}dq\\
&\qquad+C_{\alpha}\int_{s}^r\frac{\|S(t-q)\|_{L(\mathcal{B}_{-\beta};\mathcal{B}_{-\beta})}\|G(u(r))-G(v(r))-G(u(q))+G(v(q))\|_{L_2(\mathcal{B};\mathcal{B}_{-\beta}))}}{(r-q)^{1+\alpha}}dq\\
&\quad\leq c_SC_{\alpha}L_G(r-s)^{-\alpha}e^{\rho r}\|u-v\|_{\beta,\rho,-\beta}+c_SC_{\alpha}L_G\|u-v\|_{\beta,\rho,-\beta}\int_{s}^r(r-q)^{\beta-\alpha-1}e^{\rho r}dq\\
&\qquad+c_SC_{\alpha}L_G(\|u\|_{\beta,-\beta}+\|v\|_{\beta,-\beta})\|u-v\|_{\beta,\rho,-\beta}\int_{s}^re^{\rho r}(r-q)^{\beta-\alpha-1}dq\\
&\quad\leq c_SC_{\alpha,\beta}L_G(1+\|u\|_{\beta,-\beta}+\|v\|_{\beta,-\beta})\|u-v\|_{\beta,\rho,-\beta}[(r-s)^{-\alpha}+(r-s)^{\beta-\alpha}]e^{\rho r},
\end{align*}
entailing the first statement. 
Furthermore, due to \eqref{G1}--\eqref{G4} we have
\begin{align*}
&\|D_{s+}^{\alpha}S(t-\cdot)[G(u(\cdot))-G(v(\cdot))][r]\|_{L_2(\mathcal{B};\mathcal{B})}\\
&\quad\leq \frac{1}{\Gamma(1-\alpha)}\bigg(\frac{\|S(t-r)[G(u(r))-G(v(r))]\|_{L_2(\mathcal{B};\mathcal{B})}}{(r-s)^{\alpha}}\\
&\qquad+\alpha\int_{s}^{r}\frac{\|S(t-r)[G(u(r))-G(v(r))]-S(t-q)[G(u(q))-G(v(q))]\|_{L_2(\mathcal{B};\mathcal{B})}}{(r-q)^{\alpha+1}}dq\bigg)\\
&\leq c_SC_{\alpha}|u(r)-v(r)|(r-s)^{-\alpha}\\
&\quad+C_{\alpha}\int_{s}^r\frac{\|S(t-r)-S(t-q)\|_{L(\mathcal{B};\mathcal{B})}\|G(u(r))-G(v(r))\|_{L_2(\mathcal{B};\mathcal{B})}}{(r-q)^{1+\alpha}}dq\\
&\quad+C_{\alpha}\int_{s}^r\frac{\|S(t-q)\|_{L(\mathcal{B}_{-\beta};\mathcal{B})}\|G(u(r))-G(v(r))-G(u(q))+G(q(r))\|_{L_2(\mathcal{B};\mathcal{B}_{-\beta})}}{(r-q)^{1+\alpha}}dq\\
&\leq c_SC_{\alpha}L_G(r-s)^{-\alpha}e^{\rho r}\|u-v\|_{\beta,\rho,-\beta}+c_SC_{\alpha}L_G\|u-v\|_{\beta,\rho,-\beta}(t-r)^{-\beta}\int_{s}^r(r-q)^{\beta-\alpha-1}e^{\rho r}dq\\
&\quad+c_SC_{\alpha}L_G(\|u\|_{\beta,-\beta}+\|v\|_{\beta,-\beta})\|u-v\|_{\beta,\rho,-\beta}\int_{s}^re^{\rho r}{(t-q)^{-\beta}}(r-q)^{\beta-\alpha-1}dq\\
&\leq c_SC_{\alpha,\beta}L_G(1+\|u\|_{\beta,-\beta}+\|v\|_{\beta,-\beta})\|u\|_{\beta,\rho,-\beta}
[1+(t-r)^{-\beta}(r-s)^{\beta}](r-s)^{-\alpha}e^{\rho r}.
\end{align*}
\end{proof}
\begin{lemma}
  Let $\alpha'\in(0,1)$ such that $\alpha'-\alpha-1>-1$. Then   
  \begin{align*}
      &\|D_{0+}^{\alpha}[S(t-\cdot)-S(s-\cdot)]G(u(\cdot))[r]\|_{L_2(\mathcal{B};\mathcal{B}_{-\beta})}\\
      &\quad\leq C_{\alpha,\beta}c_{S}L_G\|u\|_{\beta,\rho,-\beta}e^{\rho r}(t-s)^{\beta}{[(s-r)^{-\alpha'}r^{\alpha'-\alpha}+(s-r)^{-\beta}r^{\beta-\alpha}]}
  \end{align*}
  and 
  \begin{align*}
      &\|D_{0+}^{\alpha}[S(t-\cdot)-S(s-\cdot)][G(u(\cdot))-G(v(\cdot))][r]\|_{L_2(\mathcal{B};\mathcal{B}_{-\beta})}\\
      &\leq c_{S}C_{\alpha,\beta}L_G(1+\|u\|_{\beta,-\beta}+\|v\|_{\beta,-\beta})\|u-v\|_{\beta,\rho,-\beta}e^{\rho r}(t-s)^{\beta}{[(s-r)^{-\alpha'}r^{\alpha'-\alpha}+(s-r)^{-\beta}r^{\beta-\alpha}]}.
  \end{align*}
\end{lemma}
\begin{proof} By similar computations we derive
\begin{align*}
&\|D_{0+}^{\alpha}[S(t-\cdot)-S(s-\cdot)]G(u(\cdot))[r]\|_{L_2(\mathcal{B};\mathcal{B}_{-\beta})}\\
&\quad\leq \frac{C_{\alpha}\|[S(t-r)-S(s-r)]G(u(r))\|_{L_2(\mathcal{B};\mathcal{B}_{-\beta})}}{r^{\alpha}}\\
&\qquad+C_{\alpha}\int_{0}^{r}\frac{\|[S(t-r)-S(s-r)]G(u(r))-[S(t-q)-S(s-q)]G(u(q))\|_{L_2(\mathcal{B};\mathcal{B}_{-\beta})}}{(r-q)^{\alpha+1}}dq\\
&\quad\leq \frac{C_{\alpha}\|S(t-s)-\text{id}\|_{L(\mathcal{B};\mathcal{B}_{-\beta})}\|S(s-r)\|_{L(\mathcal{B};\mathcal{B})}\|G(u(r))\|_{L_2(\mathcal{B};\mathcal{B})}}{r^{\alpha}}\\
&\qquad+C_{\alpha}\int_{0}^{r}\frac{\|[S(t-r)-S(s-r)-S(t-q)+S(s-q)]G(u(r))\|_{L_2(\mathcal{B};\mathcal{B}_{-\beta})}}{(r-q)^{\alpha+1}}dq\\
&\qquad+C_{\alpha}\int_{0}^{r}\frac{\|[S(t-q)-S(s-q)][G(u(r))-G(u(q))]\|_{L_2(\mathcal{B};\mathcal{B}_{-\beta})}}{(r-q)^{\alpha+1}}dq\\
&\quad\leq \frac{C_{\alpha}\|S(t-s)-\text{id}\|_{L(\mathcal{B};\mathcal{B}_{-\beta})}\|S(s-r)\|_{L(\mathcal{B};\mathcal{B})}\|G(u(r))\|_{L_2(\mathcal{B};\mathcal{B})}}{r^{\alpha}}\\
&\qquad+C_{\alpha}\int_{0}^{r}\frac{1}{(r-q)^{\alpha+1}}\|S(t-s)-\text{id}\|_{L(\mathcal{B};\mathcal{B}_{-\beta})}\|S(s-r)\|_{L(\mathcal{B}_{-\alpha'};\mathcal{B})}\\
&\qquad\times\|\text{id}-S(r-q)\|_{L(\mathcal{B};\mathcal{B}_{-\alpha'})}\|G(u(r))\|_{L_2(\mathcal{B};\mathcal{B})}dq\\
&\qquad+C_{\alpha}\int_{0}^{r}\frac{\|S(t-s)-\text{id}\|_{L(\mathcal{B};\mathcal{B}_{-\beta})}\|S(s-q)\|_{L(\mathcal{B}_{-\beta};\mathcal{B})}\|G(u(r))-G(u(q))\|_{L_2(\mathcal{B};\mathcal{B}_{-\beta}))}}{(r-q)^{\alpha+1}}dq\\
&\quad\leq C_{\alpha}c_{S}L_G\|u\|_{\beta,\rho,-\beta}e^{\rho r}(t-s)^{\beta}\bigg[r^{-\alpha}+\int_0^r(s-r)^{-\alpha'}(r-q)^{\alpha'-\alpha-1}dq\\
&\qquad+\int_{0}^r(s-q)^{-\beta}(r-q)^{\beta-\alpha-1}dq\bigg]\\
&\quad\leq C_{\alpha,\beta}c_{S}L_G\|u\|_{\beta,\rho,-\beta}e^{\rho r}(t-s)^{\beta}{[(s-r)^{-\alpha'}r^{\alpha'-\alpha}+(s-r)^{-\beta}r^{\beta-\alpha}]}.
\end{align*}
Finally according to \eqref{eq2.10} and assumption $(A_4)$, it holds
 \begin{align*}
&\|D_{0+}^{\alpha}[S(t-\cdot)-S(s-\cdot)][G(u(\cdot))-G(v(\cdot))][r]\|_{L_2(\mathcal{B};\mathcal{B}_{-\beta})}\\
&\leq \frac{C_{\alpha}\|[S(t-r)-S(s-r)][G(u(r))-G(v(r))]\|_{L_2(\mathcal{B};\mathcal{B}_{-\beta})}}{r^{\alpha}}\\
&\quad+C_{\alpha}\int_{0}^{r}\frac{1}{(r-q)^{\alpha+1}}\big\|[S(t-r)-S(s-r)][G(u(r))-G(v(r))]\\
&\quad-[S(t-q)-S(s-q)][G(u(q))-G(v(q))]\big\|_{L_2(\mathcal{B};\mathcal{B}_{-\beta})}dq\\
&\leq \frac{C_{\alpha}\|S(t-s)-\text{id}\|_{L(\mathcal{B};\mathcal{B}_{-\beta})}\|S(s-r)\|_{L(\mathcal{B};\mathcal{B})}\|G(u(r))-G(v(r))\|_{L_2(\mathcal{B};\mathcal{B})}}{r^{\alpha}}\\
&+C_{\alpha}\int_{0}^{r}\frac{\|[S(t-r)-S(s-r)-S(t-q)+S(s-q)][G(u(r))-G(v(r))]\|_{L_2(\mathcal{B};\mathcal{B}_{-\beta}))}}{(r-q)^{\alpha+1}}dq\\
&+C_{\alpha}\int_{0}^{r}\frac{\|[S(t-q)-S(s-q)][G(u(r))-G(v(r))-G(u(q))+G(v(q))]\|_{L_2(\mathcal{B};\mathcal{B}_{-\beta})}}{(r-q)^{\alpha+1}}dq\\
&\leq \frac{C_{\alpha}\|S(t-s)-\text{id}\|_{L(\mathcal{B};\mathcal{B}_{-\beta})}\|S(s-r)\|_{L(\mathcal{B};\mathcal{B})}\|G(u(r))-G(v(r))\|_{L_2(\mathcal{B};\mathcal{B})}}{r^{\alpha}}\\
&\quad+C_{\alpha}\int_{0}^{r}\frac{1}{(r-q)^{\alpha+1}}\|S(t-s)-\text{id}\|_{L(\mathcal{B};\mathcal{B}_{-\beta})}\|S(s-r)\|_{L(\mathcal{B}_{-\alpha'};\mathcal{B})}\\
&\quad\times\|\text{id}-S(r-q)\|_{L(\mathcal{B};\mathcal{B}_{-\alpha'})}\|G(u(r))-G(v(r))\|_{L_2(\mathcal{B};\mathcal{B})}dq\\
&\quad+C_{\alpha}\int_{0}^{r}\frac{1}{(r-q)^{\alpha+1}}\|S(t-s)-\text{id}\|_{L(\mathcal{B};\mathcal{B}_{-\beta})}\|S(s-q)\|_{L(\mathcal{B}_{-\beta};\mathcal{B})}\\
&\quad\times\|G(u(r))-G(v(r))-G(u(q))+G(v(q))\|_{L_2(\mathcal{B};\mathcal{B}_{-\beta})}dq\\
&\leq c_{S}C_{\alpha,\beta}L_G\|u-v\|_{\beta,\rho,-\beta}e^{\rho r}(t-s)^{\beta}\bigg[r^{-\alpha}+\int_0^r(s-r)^{-\alpha'}(r-q)^{\alpha'-\alpha-1}dq\\
&+(\|u\|_{\beta,-\beta}+\|v\|_{\beta,-\beta})\int_{0}^r(s-q)^{-\beta}(r-q)^{\beta-\alpha-1}dq\bigg]\\
&\leq c_{S}C_{\alpha,\beta}L_G(1+\|u\|_{\beta,-\beta}+\|v\|_{\beta,-\beta})\|u-v\|_{\beta,\rho,-\beta}e^{\rho r}{[(s-r)^{-\alpha'}r^{\alpha'-\alpha}+(s-r)^{-\beta}r^{\beta-\alpha}]}.
\end{align*}
\end{proof}
In the following computations we will apply the following useful result.  For any $s,t\in\mathbb{R}$, $a,b>-1$,  due to the definition of Euler's Gamma and Beta-functions, it is straightforward that
\begin{align}\label{eq3.12}
  \int_{s}^{t}(t-r)^{a}(r-s)^{b}dr=\frac{\Gamma(a+1)\Gamma(b+1)}{\Gamma(a+b+2)}(t-s)^{a+b+1}.
\end{align}
We further set
\begin{align}\label{eq3.11}
   K_{1}(\rho):=\sup_{0\leq s< t\leq T}\int_{s}^{t}e^{-\rho(t-r)}(r-s)^{a}(t-r)^{b}dr.
\end{align}
Combining the fact that
\begin{equation*}
  \int_{s}^{t}e^{-\rho(t-r)}(r-s)^{a}(t-r)^{b}dr=(t-s)^{a+b+1}\int_{0}^{1}e^{-\rho(t-s)(1-r)}r^a(1-r)^{b}dr
\end{equation*}
and \cite[Lemma 8]{BGHS2017}, we know that $K_{1}(\rho)\rightarrow0$ as $\rho\rightarrow\infty$ if $a,b>-1$ and $a+b+1>0$.
\begin{lemma}\label{estim of G}
Under the assumptions $(\bf A_1$), $(\bf A_3$) and $(\bf A_4$), there exist constants $c_S,C_{\alpha,\beta}$ such that for any $u,v\in C_{-\beta}^{\beta}([0,T];\mathcal{B})$,
\begin{align*}
&\left\|\int_0^\cdot S(\cdot-r)G(u(r))d\omega(r)\right\|_{\beta,\rho,-\beta}\leq c_SC_{\alpha,\beta}K_{1}(\rho)L_G\|u\|_{\beta,\rho,-\beta}\vertiii{\omega}_{\beta',0,T},
\end{align*}
and 
\begin{align*}
&\left\|\int_0^\cdot S(\cdot-r)[G(u(r))-G(v(r))]d\omega(r)\right\|_{\beta,\rho,-\beta}\\
&\quad\leq C_{\alpha,\beta}c_SK_{1}(\rho)L_G(1+\|u\|_{\beta,-\beta}+\|v\|_{\beta,-\beta})\|u-v\|_{\beta,\rho,-\beta}\vertiii{\omega}_{\beta',0,T}.
\end{align*}
\end{lemma}
\begin{proof}First, it is easy to check that
\begin{align*}
|D_{t-}^{1-\alpha}\omega[r]|\leq \vertiii{\omega}_{\beta',0,T}(t-r)^{\alpha+\beta'-1}.
\end{align*} 
Applying the estimates above for each $D^{\alpha}_{s+}$, we have 
\begin{align*}
&\left|\int_{s}^tS(t-r)[G(u(r))-G(v(r))]d\omega(r)\right|_{-\beta}\\&\quad\leq \int_s^t\|D_{s+}^{\alpha}S(t-\cdot)[G(u(\cdot))-G(v(\cdot))][r]\|_{L_2(\mathcal{B};\mathcal{B}_{-\beta}))}|D_{t-}^{1-\alpha}\omega[r]|dr\\
&\quad\leq c_SC_{\alpha,\beta}L_G(1+\|u\|_{\beta,-\beta}+\|v\|_{\beta,-\beta})\|u\|_{\beta,\rho,-\beta}\vertiii{\omega}_{\beta',0,T}\\
&\qquad\times\int_s^t(t-r)^{\alpha+\beta'-1}(r-s)^{-\alpha}[1+(r-s)^{\beta}]e^{\rho r}dr,
\end{align*}
\begin{align*}
&\left|\int_{s}^tS(t-r)G(u(r))d\omega(r)\right|\\&\quad\leq \int_s^t\|D_{s+}^{\alpha}S(t-\cdot)G(u(\cdot))[r]\|_{L_2(\mathcal{B};\mathcal{B})}|D_{t-}^{1-\alpha}\omega[r]|dr\\
&\quad\leq c_SC_{\alpha,\beta}L_G\|u\|_{\beta,\rho,-\beta}\vertiii{\omega}_{\beta',0,T}\int_s^t(t-r)^{\alpha+\beta'-1}(r-s)^{-\alpha}[1+(t-r)^{-\beta}(r-s)^{\beta}]e^{\rho r}dr,
\end{align*}
and 
\begin{equation}\label{U-V:G}
\begin{aligned}
&\left|\int_{s}^tS(t-r)[G(u(r))-G(v(r))]d\omega(r)\right|\\&\quad\leq \int_s^t\|D_{s+}^{\alpha}S(t-\cdot)[G(u(\cdot))-G(v(\cdot))][r]\|_{L_2(\mathcal{B};\mathcal{B})}|D_{t-}^{1-\alpha}\omega[r]|dr\\
&\quad\leq c_SC_{\alpha,\beta}L_G(1+\|u\|_{\beta,-\beta}+\|v\|_{\beta,-\beta})\|u\|_{\beta,\rho,-\beta}\vertiii{\omega}_{\beta',0,T}\\
&\qquad\times\int_s^t(t-r)^{\alpha+\beta'-1}[(r-s)^{-\alpha}+(t-r)^{-\beta}(r-s)^{\beta-\alpha}
]e^{\rho r}dr.
\end{aligned}
\end{equation}

Taking the $\mathcal{B}_{-\beta}$-norm, we have
\begin{align*}
&\left|\int_{s}^tS(t-r)G(u(r))d\omega(r)\right|_{-\beta}
\\&\quad\leq \int_s^t\|D_{s+}^{\alpha}S(t-\cdot)G(u(\cdot))[r]\|_{L_2(\mathcal{B};\mathcal{B}_{-\beta}))}|D_{t-}^{1-\alpha}\omega[r]|dr
\\ &\quad\leq c_SC_{\alpha,\beta}L_G\|u\|_{\beta,\rho,-\beta}\vertiii{\omega}_{\beta',0,T}\int_s^t(t-r)^{\alpha+\beta'-1}[(r-s)^{-\alpha}+(r-s)^{\beta-\alpha}]e^{\rho r}dr.
\end{align*}
\begin{align*}
&\left|\int_0^s [S(t-r)-S(s-r)]G(u(r))d\omega(r)\right|_{-\beta}\\
&\quad\leq c_SC_{\alpha,\beta}L_G\|u\|_{\beta,\rho,-\beta}\vertiii{\omega}_{\beta',0,T}(t-s)^{\beta}\int_0^se^{\rho r}(s-r)^{\alpha+\beta'-1}{[(s-r)^{-\alpha'}r^{\alpha'-\alpha}+(s-r)^{-\beta}r^{\beta-\alpha}]}dr\\
&\quad \leq c_SC_{\alpha,\beta}L_G\|u\|_{\beta,\rho,-\beta}e^{\rho s}\vertiii{\omega}_{\beta',0,T}(t-s)^{\beta}s^{\beta'},
\end{align*}
and \begin{align*}
&\left|\int_0^s [S(t-r)-S(s-r)][G(u(r))-G(v(r))]d\omega(r)\right|_{-\beta}\\
&\quad\leq c_SC_{\alpha,\beta}L_G(1+\|u\|_{\beta,-\beta}+\|v\|_{\beta,-\beta})\|u-v\|_{\beta,\rho,-\beta}\vertiii{\omega}_{\beta',0,T}(t-s)^{\beta}\\
&\qquad\times\int_0^se^{\rho r}(s-r)^{\alpha+\beta'-1}{[(s-r)^{-\alpha'}r^{\alpha'-\alpha}+(s-r)^{-\beta}r^{\beta-\alpha}]}dr\\
&\quad \leq c_SC_{\alpha,\beta}L_G(1+\|u\|_{\beta,-\beta}+\|v\|_{\beta,-\beta})\|u-v\|_{\beta,\rho,-\beta}e^{\rho s}\vertiii{\omega}_{\beta',0,T}(t-s)^{\beta}s^{\beta'}.
\end{align*}
Hence for $u\in C_{-\beta}^{\beta}([0,T];\mathcal{B})$ we compute regarding~\eqref{norm:c}
\begin{align*}
&\left\|\int_0^\cdot S(\cdot-r)G(u(r))d\omega(r)\right\|_{\beta,\rho,-\beta}\\
\\ &\quad\leq \sup_{t\in [0,T]}e^{-\rho t}\left|\int_0^t S(t-r)G(u(r))d\omega(r)\right|+
 \sup_{s<t\in [0,T]}\frac{e^{-\rho t}\left|\int_s^t S(t-r)G(u(r))d\omega(r)\right|_{-\beta}}{(t-s)^{\beta}}\\
&\qquad+\sup_{s<t\in [0,T]}\frac{e^{-\rho t}\left|\int_0^s [S(t-r)-S(s-r)]G(u(r))d\omega(r)\right|_{-\beta}}{(t-s)^{\beta}}\\
&\quad\leq \sup_{t\in [0,T]}c_SC_{\alpha,\beta}L_G\|u\|_{\beta,\rho,-\beta}\vertiii{\omega}_{\beta',0,T}\\
&\qquad \times\int_s^t(t-r)^{\alpha+\beta'-1}[(r-s)^{-\alpha}+(t-r)^{-\beta}(r-s)^{\beta-\alpha}]e^{\rho -(t-r)}dr\\
&\qquad+\sup_{s<t\in [0,T]}c_SC_{\alpha,\beta}L_G\|u\|_{\beta,\rho,-\beta}\vertiii{\omega}_{\beta',0,T}\\
&\qquad \times\frac{\int_s^t(t-r)^{\alpha+\beta'-1}[(r-s)^{-\alpha}+(r-s)^{\beta-\alpha}]e^{-\rho(t-r) }dr}{(t-s)^{\beta}}\\
&\qquad+\sup_{s<t\in [0,T]}c_SC_{\alpha,\beta}L_G\|u\|_{\beta,\rho,-\beta}\vertiii{\omega}_{\beta',0,T}(t-s)^{\beta}\\
&\qquad\times\frac{\int_0^se^{-\rho(t-r))}(s-r)^{\alpha+\beta'-1}{[(s-r)^{-\alpha'}r^{\alpha'-\alpha}+(s-r)^{-\beta}r^{\beta-\alpha}]}dr}{(t-s)^{\beta}}\\
&\quad\leq c_SC_{\alpha,\beta}K_{1}(\rho)L_G\|u\|_{\beta,\rho,-\beta}\vertiii{\omega}_{\beta',0,T}.
\end{align*}
For $u,v\in C^{\beta}_{-\beta}([0,T];\mathcal{B})$ we  further have that 
\begin{align*}
&\left\|\int_0^\cdot S(\cdot-r)[G(u(r))-G(v(r))]d\omega(r)\right\|_{\beta,\rho,-\beta}\\
&\quad=\sup_{t\in [0,T]}e^{-\rho t}\left|\int_0^t S(t-r)[G(u(r))-G(v(r))]d\omega(r)\right|\\
&\qquad+ \sup_{s<t\in [0,T]}\frac{e^{-\rho t}\left|\int_0^t S(t-r)[G(u(r))-G(v(r))]d\omega(r)-\int_0^s S(s-r)[G(u(r))-G(v(r))]d\omega(r)\right|_{-\beta}}{(t-s)^{\beta}}\\
&\quad \leq \sup_{t\in [0,T]}e^{-\rho t}\left|\int_0^s S(t-r)[G(u(r))-G(v(r))]d\omega(r)\right|\\
&\qquad+ \sup_{s<t\in [0,T]}\frac{e^{-\rho t}\left|\int_0^s [S(t-r)-S(s-r)][G(u(r))-G(v(r))]d\omega(r)d\omega(r)\right|_{-\beta}}{(t-s)^{\beta}}\\
&\qquad+ \sup_{s<t\in [0,T]}\frac{e^{-\rho t}\left|\int_s^t S(t-r)[G(u(r))-G(v(r))]d\omega(r)\right|_{-\beta}}{(t-s)^{\beta}}\\
&\quad\leq C_{\alpha,\beta}c_SK_{1}(\rho)L_G(1+\|u\|_{\beta,-\beta}+\|v\|_{\beta,-\beta})\|u-v\|_{\beta,\rho,-\beta}\vertiii{\omega}_{\beta',0,T}.
\end{align*}
\end{proof}
Putting all these deliberations together we infer. 
\begin{theorem}\label{th-exits-solu}
Assume that $(\bf A_1$)-$(\bf A_4$) are satisfied. Then the SPDE~\eqref{eq2.5} has a unique solution $u\in C_{-\beta}^{\beta}([0,T];\mathcal{B})$ with $u(0)=u_0\in \mathcal{B}$. Moreover, its solution operator generates  a random dynamical system $\varphi:\mathbb{R}^+\times\Omega\times \mathcal{B}\to \mathcal{B}$. 
\end{theorem}
\begin{proof}
We apply the Banach fixed point theorem in order to obtain the existence {and} uniqueness of a solution to~\eqref{eq2.5}. By Lemma \ref{estim of G}, we need to focus on the estimates of $S(\cdot)u_0$ and the term containing the drift $F$. In fact, given 
$u_0\in \mathcal{B}$ and $\beta\geq 0$, 
it is straightforward to check that 
\begin{align}\label{norm0}
\nonumber\|S(\cdot)u_0\|_{\beta,\rho,-\beta}&\leq \sup_{t\in {0,T}}e^{-\rho t}|S(t)u_0|+\sup_{0\le s<t\le T}\frac{e^{-\rho t}\|S(t)-S(s)\|_{L(\mathcal{B},\mathcal{B}_{-\beta})}|u_0|}{(t-s)^{\beta}}\\
&\leq c_S|u_0|.
\end{align}
As for the drift term, regarding again~\eqref{norm:c} and~\eqref{F2}, we directly obtain that
\begin{align*}
&\left\|\int_0^\cdot S(\cdot-r)F(u(r))dr\right\|_{\beta,\rho,-\beta}\\
&\quad=\sup_{t\in [0,T]}e^{-\rho t}\left|\int_0^t S(t-r)F(u(r))dr\right|\\
&\qquad+ \sup_{s<t\in [0,T]}\frac{e^{-\rho t}\left|\int_0^t S(t-r)F(u(r))dr-\int_0^s S(s-r)F(u(r))dr\right|_{-\beta}}{(t-s)^{\beta}}\\
&\quad\leq \sup_{t\in [0,T]}e^{-\rho t}\left|\int_0^t S(t-r)F(u(r))dr\right|
+ \sup_{s<t\in [0,T]}\frac{e^{-\rho t}\left|\int_s^t S(t-r)F(u(r))dr\right|_{-\beta}}{(t-s)^{\beta}}\\
&\qquad+\sup_{s<t\in [0,T]}\frac{e^{-\rho t}\left|\int_0^s [S(t-r)-S(s-r)]F(u(r))dr\right|_{-\beta}}{(t-s)^{\beta}}\\
&\quad\leq c_SL_F\frac{1-e^{-\rho T}}{\rho}\|u\|_{\beta,\rho,-\beta}(1+T^{1-\beta})\\
&\qquad+\sup_{s<t\in [0,T]}\frac{e^{-\rho t}\int_0^s \|S(t-r)-S(s-r)\|_{L(\mathcal{B};\mathcal{B}_{-\beta})}|F(u(r))|dr}{(t-s)^{\beta}}\\
&\quad\leq c_SL_F\frac{1-e^{-\rho T}}{\rho}\|u\|_{\beta,\rho,-\beta}(1+T^{1-\beta}+T)\\
&\quad=c_SL_FK_2(\rho)\|u\|_{\beta,\rho,-\beta}(1+T^{1-\beta}+T),
\end{align*}
and
\begin{align*}
&\left\|\int_0^\cdot S(\cdot-r)[F(u(r))-F(v(r))]dr\right\|_{\beta,\rho,-\beta}\\
&\quad=\sup_{t\in [0,T]}e^{-\rho t}\left|\int_0^t S(t-r)[F(u(r))-F(v(r))]dr\right|\\
&\qquad+ \sup_{s<t\in [0,T]}\frac{e^{-\rho t}\left|\int_0^t S(t-r)[F(u(r))-F(v(r))]dr-\int_0^s S(s-r)[F(u(r))-F(v(r))]dr\right|_{-\beta}}{(t-s)^{\beta}}\\
&\quad\leq \sup_{t\in [0,T]}e^{-\rho t}\left|\int_0^t S(t-r)[F(u(r))-F(v(r))]dr\right|
+ \sup_{s<t\in [0,T]}\frac{e^{-\rho t}\left|\int_s^t S(t-r)[F(u(r))-F(v(r))]dr\right|_{-\beta}}{(t-s)^{\beta}}\\
&\qquad+\sup_{s<t\in [0,T]}\frac{e^{-\rho t}\left|\int_0^s [S(t-r)-S(s-r)][F(u(r))-F(v(r))]dr\right|_{-\beta}}{(t-s)^{\beta}}\\
&\quad\leq c_SL_F\frac{1-e^{-\rho T}}{\rho}\|u-v\|_{\beta,\rho,-\beta}(1+T^{1-\beta})\\
&\qquad+\sup_{s<t\in [0,T]}\frac{e^{-\rho t}\int_0^s \|S(t-r)-S(s-r)\|_{L(\mathcal{B};\mathcal{B}_{-\beta})}|F(u(r))-F(v(r))|dr}{(t-s)^{\beta}}\\
&\leq c_SL_FK_2(\rho)\|u-v\|_{\beta,\rho,-\beta}(1+T^{1-\beta}+T),
\end{align*}
where $K_2(\rho)=\frac{1-e^{-\rho T}}{\rho}$ which converges  to $0$ as $\rho\to \infty$.

For simplicity, we set $T:=1$. Let $\mathcal{T}$ be the mapping from $C^{\beta}_{-\beta}([0,1];\mathcal{B})$ into itself given by 
\begin{equation*}
\mathcal{T}(u)(t):=S(t)u_0+\int_0^t S(t-\tau)F(u(\tau))d\tau+\int_0^t S(t-\tau)G(u(\tau))d\omega(\tau).
\end{equation*}
Taking Lemma \ref{estim of G} into consideration, we have that 
\begin{align*}
\|\mathcal{T}(u)\|_{\beta,\rho,-\beta}&\leq c_S|u_0|+ c_SL_FK_2(\rho)\|u\|_{\beta,\rho,-\beta}+ c_SC_{\alpha,\beta}L_GK_{1}(\rho)\|u\|_{\beta,\rho,-\beta}\vertiii{\omega}_{\beta',0,1},
\end{align*}
 and 
\begin{align*}
\|\mathcal{T}(u-v)\|_{\beta,\rho,-\beta}&\leq c_S|u_0-v_0|+ c_SL_FK_2(\rho)\|u-v\|_{\beta,\rho,-\beta}\\
&\quad+ c_SC_{\alpha,\beta}L_GK_{1}(\rho)(1+\|u\|_{\beta,-\beta}+\|v\|_{\beta,-\beta})\|u-v\|_{\beta,\rho,-\beta}\vertiii{\omega}_{\beta',0,1}.
\end{align*}
Choosing a large enough $\rho_0$ such that $\max\{c_SL_FK_2(\rho_0),c_SC_{\alpha,\beta}L_GK_{1}(\rho_0)\vertiii{\omega}_{\beta',0,1}\}\leq \frac{1}{4}$, we get
\begin{align*}
\|\mathcal{T}(u)\|_{\beta,\rho_0,-\beta}&\leq c_S|u_0|+ \frac{1}{2}\|u\|_{\beta,\rho_0, -\beta}.
\end{align*}
Let 
\begin{align*}
B:=\{u\in C^{\beta}_{-\beta}([0,1];\mathcal{B}),\|u\|_{\beta,\rho_0, -\beta}\leq R\}, ~~~R:=2c_S|u_0|.
\end{align*} Then $\mathcal{T}$ is a self-mapping on $B$. Note that 
\begin{align*}
\sup_{u\in B}\|u\|_{\beta,-\beta}\leq e^{\rho_0 T}\sup_{u\in B}\|u\|_{\beta,\rho_0,-\beta}\leq Re^{\rho_0 T}=:R_1.
\end{align*}
Therefore, considering $u_0=v_0$, it is possible to find a $\rho_1$ such that for $u,v\in B$, 
\begin{align*}
\|\mathcal{T}(u)-\mathcal{T}(v)\|_{\beta,\rho_1,-\beta}&\leq c_SL_FK_2(\rho_1)\|u-v\|_{\beta,\rho_1,-\beta}\\
&\quad+ c_SC_{\alpha,\beta}L_GK_{1}(\rho_1)(1+2R_1)\|u-v\|_{\beta,\rho_1,-\beta}\vertiii{\omega}_{\beta',0,1}\\
&\leq \frac{1}{2}\|u-v\|_{\beta,\rho_1,-\beta}.
\end{align*}
Hence by the Banach fixed point theorem, there is a unique fixed point in $B$ such that $\mathcal{T}(u)=u$ for given $u_0\in \mathcal{B}$. In conclusion, following a similar argument to  \cite[Theorem 16]{Chen2014}, we conclude that the solution operator of~\eqref{eq2.5} generates a random dynamical system.
\end{proof}

\section{Main results}
\label{se3}

The aim of this section is to prove the existence of a $C^k$-smooth  stable manifold for the dynamical system $\varphi$ generated by the unique global solution of \eqref{mild sol}. 
%
Let $\mathcal{H}_{\kappa},\check{\mu}<0<\kappa<\min\{-\check{\mu},\hat{\mu}\}$ be a space with elements $U=\{u_i\}_{i\in\mathbb{Z}^+}$, where $u_{i}\in C^{\beta}_{-\beta}([0,1];\mathcal{B})$, with the norm,
\begin{align*}
  \|U\|_{\mathcal{H}_{\kappa}}:={\sup_{i\in\mathbb{Z}^+}e^{\kappa i}\|u_{i}\|_{\beta,-\beta}}<\infty,~~~u_{i}(1)=u_{i+1}(0).
\end{align*}
The space $(\mathcal{H}_{\kappa},\|\cdot\|_{\mathcal{H}_{\kappa}})$ is a Banach space. Let also $\mathcal{H}_{\kappa}^{\xi}$ be the  subset of $\mathcal{H}_{\kappa}$ such that $u_{0}^-(0)=\xi\in \mathcal{B}^-$.  Since $\mathcal{H}_{\kappa}^{\xi}$ is a closed subset of $\mathcal{H}_{\kappa}$, it is  complete.
\subsection{Derivation of the Lyapunov-Perron  operator}\label{se.L-P}

Similar to~\cite{GarridoAtienza2010,KN22} we introduce now a discrete Lyapunov-Perron operator associated with \eqref{mild sol}. 
\begin{lemma}
Given $\xi\in\mathcal{B}^-, U=\{u_i\}_{i\in\mathbb{Z}^+}\in \mathcal{H}_{\kappa}^{\xi}$, $t\in[0,1]$, let $u(t+m)=u_m(t),t\in[0,1]$. Then $u(\cdot)$ is the solution of \eqref{mild sol} with $u^-(0)=\xi$ if and only if $U$ satisfies the following equality:
  \begin{equation}\label{eq.L-P}
\begin{aligned}
\relax[J(U)]_m(t)&=S^{-}(t+m)\xi\\
&\quad+\sum_{i=1}^{m}S^{-}(t+m-i)\int_{0}^{1} S^{-}(1-\tau)F(u_{i-1}(\tau))d\tau\\
&\quad+\sum_{i=1}^{m}S^{-}(t+m-i)\int_{0}^{1} S^{-}(1-\tau)G(u_{i-1}(\tau))d\theta_{i-1}\omega(\tau)\\
&\quad+\int_{0}^{t} S^{-}(t-\tau)F(u_{m}(\tau))d\tau+\int_{0}^{t} S^{-}(t-\tau)G(u_{m}(\tau))d\theta_{m}\omega(\tau)\\
&\quad-\int_{t}^{1} S^{+}(t-\tau)F(u_m(\tau))d\tau-\int_{t}^{1} S^{+}(t-\tau)G(u_m(\tau))d\theta_{m}\omega(\tau)\\
&\quad-\sum_{i=m+2}^{\infty}S^{+}(t+m-i)\int_{0}^{1} S^{+}(1-\tau)F(u_{i-1}(\tau))d\tau\\
&\quad-\sum_{i=m+2}^{\infty}
S^{+}
(t+m-i)\int_{0}^{1}S^{+}
(1-\tau)G(u_{i-1}(\tau))d\theta_{i-1}\omega(\tau),
\end{aligned}
\end{equation}
with the convention that $\sum_{i=k}^{n}f_i=0 $ for any $n<k$.
\end{lemma}
\begin{proof}
If $u$ is the solution of \eqref{mild sol} with initial value $u(0)$ with $u^-(0)=\xi$, denoting $u(t+m)=u_m(t),t\in[0,1]$ and $U=(u_i)_{i\in\mathbb{Z}^+}$, then by the variation of constants formula,  we obtain that
  \begin{equation}\label{eq5.3}
\begin{aligned}
u^{-}_m(t)&=\varphi^{-}(t+m,\omega,u(0))\\&~=S^{-}(t+m)\xi
+\sum_{i=1}^{m}\int_{i-1}^{i} S^{-}(t+m-\tau)F(u(\tau))d\tau
+\int_{m}^{t+m} S^{-}(t+m-\tau)F(u(\tau))d\tau\\
&\quad+\sum_{i=1}^{m}\int_{i-1}^{i} S^{-}(t+m-\tau)G(u(\tau))d\omega(\tau)
+\int_{m}^{t+m} S^{-}(t+m-\tau)G(u(\tau))d\omega(\tau)\\
&=S^{-}(t+m)\xi
+\sum_{i=1}^{m}S^{-}(t+m-i)\int_{0}^{1} S^{-}(1-\tau)F(u_{i-1}(\tau))d\tau\\
&\quad+\sum_{i=1}^{m}S^{-}(t+m-i)\int_{0}^{1} S^{-}(1-\tau)G(u_{i-1}(\tau))d\theta_{i-1}\omega(\tau)\\
&\quad+\int_{0}^{t} S^{-}(t-\tau)F(u_{m}(\tau))d\tau+\int_{0}^{t} S^{-}(t-\tau)G(u_{m}(\tau))d\theta_{m}\omega(\tau).
\end{aligned}
\end{equation}
Regarding the part in $\mathcal{B}^+$, by the definition of the solution of \eqref{mild sol}, we have
\begin{align*}
u^+(s)=S^+(s)u^+(0)+\int_{0}^{s}S^+(s-r)F(u(r))dr+\int_{0}^{s}S^+(s-r)G(u(r))d\omega(r)
\end{align*}
and for $n\geq s$, by the cocycle property, it holds that
\begin{align*}
u^+(n)=S^+(n-s)u^+(s)+\int_{s}^{n}S^+(n-r)F(u(r))dr
+\int_{s}^{n}S^+(n-r)G(u(r))d\omega(r).
\end{align*}
Since $\mathcal{B}^+$ is finite-dimensional, the semigroup $S^+(\cdot)$ can be extended to a group using that $S^+(-t)S^+(t)=S^+(t)S^+(-t)=\text{id}$.
Hence, 
\begin{align*}
u^+(s)=S^+(s-n)u^+(n)-\int_{s}^{n}S^+(s-r)F(u(r))dr-\int_{s}^{n}S^+(s-r)G(u(r))d\omega(r).
\end{align*}
This further implies that
  \begin{equation}\label{eq5.4}
  \begin{aligned}
&u^{+}(t+m)\\
&=S^{+}(t+m-n)u^+(n)-\int_{t+m}^nS^+(t+m-\tau)F((u(\tau)))d\tau
-\int_{t+m}^nS^+(t+m-\tau)G((u(\tau)))d\omega(\tau)\\
&=S^{+}(t+m-n)u^+(n)-\int_{t+m}^{m+1} S^{+}(t+m-\tau)F(u(\tau))d\tau
-\int_{t+m}^{m+1} S^{+}(t+m-\tau)G(u(\tau))d\omega(\tau)\\
&\quad-\sum_{i=m+2}^{n}\int_{i-1}^{i} S^{+}(t+m-\tau)F(u(\tau))d\tau
-\sum_{i=m+2}^{n}
\int_{i-1}^{i}S^{+}
(t+m-\tau)G(u(\tau))d\omega(\tau)\\
&=S^{+}(t+m-n)u^+_n(0)-\int_{t}^{1} S^{+}(t-\tau)F(u_m(\tau))d\tau-\int_{t}^{1} S^{+}(t-\tau)G(u_m(\tau))d\theta_{m}\omega(\tau)\\
&\quad-\sum_{i=m+2}^{n}S^{+}(t+m-i)\int_{0}^{1} S^{+}(1-\tau)F(u_{i-1}(\tau))d\tau\\
&\quad-\sum_{i=m+2}^{n}
S^{+}
(t+m-i)\int_{0}^{1}S^{+}
(1-\tau)G(u_{i-1}(\tau))d\theta_{i-1}\omega(\tau),
\end{aligned}
\end{equation}
where $n\geq m+1$. Noting that it follows from \eqref{eq.S-posi} that
\begin{equation*}
  \begin{aligned}
\left\|S^{+}(t+m-n)u_n^+(0)\right\|&
    \le  c_S e^{\hat{\mu}(t+m-n)}
    \left\|u_n\right\|_{\beta,-\beta}\\
    &= c_S e^{\hat{\mu}(t+m)}e^{-(\hat{\mu}+\kappa)n}e^{\kappa n}
    \left\|u_n\right\|_{\beta,-\beta}\\
&\leq  c_S e^{\hat{\mu}(t+m)}e^{-(\hat{\mu}+\kappa)n}
    \left\|U\right\|_{\mathcal{H}_{\kappa}}
  \end{aligned}
\end{equation*}
which vanishes as $n\to+\infty$ since $\kappa+\hat{\mu}>0$. So taking the limit $n\to+\infty$ in \eqref{eq5.4} results in
 \begin{equation}\label{eq5.5}
  \begin{aligned}
u^{+}_{m}(t)&=
-\int_{t}^{1} S^{+}(t-\tau)F(u_m(\tau))d\tau-\int_{t}^{1} S^{+}(t-\tau)G(u_m(\tau))d\theta_{m}\omega(\tau)\\
&\quad-\sum_{i=m+2}^{\infty}S^{+}(t+m-i)\int_{0}^{1} S^{+}(1-\tau)F(u_{i-1}(\tau))d\tau\\
&\quad-\sum_{i=m+2}^{\infty}
S^{+}
(t+m-i)\int_{0}^{1}S^{+}
(1-\tau)G(u_{i-1}(\tau))d\theta_{i-1}\omega(\tau).
\end{aligned}
\end{equation}
Then $U$ satisfies \eqref{eq.L-P} by combining \eqref{eq5.3} with \eqref{eq5.5}. Conversely, if $U\in \mathcal{H}_{\kappa}^{\xi}$ satisfies \eqref{eq.L-P}, then denoting by $u(t)=u_i(s)$ with $i+s=t$ for any $t\geq 0$,  it follows from a straightforward calculation that $u(t)$ is the solution of \eqref{mild sol} with $u^-(0)=\xi$.
\end{proof}
As commonly met in the theory of invariant manifolds~\cite{GarridoAtienza2010,KN22}, we have to truncate the nonlinear terms $F$ and $G$. To this aim let $\chi$ be  a smooth function on $C^{\beta}_{-\beta}([0,1];\mathcal{B})$ given by
\begin{equation*}
\chi(u):=\left\{
\begin{aligned}
&u,\|u\|_{\beta,-\beta}\leq\frac{1}{2}\\
&0,\|u\|_{\beta,-\beta}>1,
\end{aligned}\right.
\end{equation*}
with continuous first and second derivative bounded by the constants $D_{\chi},D^{2}_{\chi}$  .
For $R>0$ (which will be determined in Subsection \ref{sub-exis}), let 
\begin{equation*}
\chi_{R}(u):=R\chi\left(\frac{u}{R}\right)=\left\{
\begin{aligned}
&u,\|u\|_{\beta,-\beta}\leq\frac{R}{2}\\
&0,\|u\|_{\beta,-\beta}>R.
\end{aligned}\right.
\end{equation*}
We further define $F_{R}(u):=F\circ\chi_{R}(u)$ and $G_{R}(u):=G\circ\chi_{R}(u)$ for $u\in C_{-\beta}^{\beta}([0,1];\mathcal{B})$.
Denote by 
\begin{align*}
L_{F}(R)&=\max\{D_{\chi}\sup_{u,v\in B_\mathcal{B}(0,R),{\eta\in[0,1]}}\|DF(\eta\chi_R(u)+(1-\eta)\chi_{R}(v))\|_{L(\mathcal{B};\mathcal{B})},\\
&D_{\chi}\sup_{u,v\in B_\mathcal{B}(0,R),{\eta\in[0,1]}}\|DF(\eta\chi_R(u)+(1-\eta)\chi_{R}(v))\|_{L(\mathcal{B};\mathcal{B}_{-\beta})}\},
\end{align*}
and 
\begin{align*}
L_{G}(R)&=\max\{D_{\chi}\sup_{w\in B_\mathcal{B}(0,R)}\|DG(w)\|_{L(\mathcal{B};L_2(\mathcal{B};\mathcal{B}))},
 D_{\chi}\sup_{w\in B_\mathcal{B}(0,R)}\|DG(w)\|_{L(\mathcal{B}_{-\beta};L_2(\mathcal{B};\mathcal{B}_{-\beta}))}\}.
\end{align*}
{Note that due to the boundedness of $DF,DG$ and due to the conditions $DF(0)=0,DG(0)=0$, $L_F(R)$ and $L_G(R)$ converge to 0 as $R\to 0$.~Moreover we have that $L_F(R)=C[D_\chi, C_{DF}] R \text{ and } L_G(R)=C[D_\chi, D^2_\chi,C_{DG}, C_{D^2G}] R$, where $C$ is a constant which depends only on $D_\chi$, $D^2_\chi$ and on the bounds of the derivatives of $F$ and $G$ denoted by $C_{DF}$ respectively $C_{DG}$, $C_{D^2G}$. Here we disregard the embedding constants since these do not change the argument}. Due to assumption $(\bf{A_2})$ and $(\bf{A_3})$, it holds that
\begin{align*}
|F_{R}(u)-F_{R}(v)|&=|F(\chi_{R}(u))-F(\chi_{R}(v))|\\
&=\int_{0}^{1}DF(\eta\chi_R(u)+(1-\eta)\chi_{R}(v))(\chi_{R}(u)-\chi_{R}(v))d\eta\\
&\leq D_{\chi}\sup_{u,v\in B_\mathcal{B}(0,R),{\eta\in[0,1]}}\|DF(\eta\chi_R(u)+(1-\eta)\chi_{R}(v))\||u-v|\\
&\leq L_{F}(R)|u-v|
\end{align*} 
and 
\begin{align*}
|F_{R}(u)-F_{R}(v)|_{-\beta}&=|F(\chi_{R}(u))-F(\chi_{R}(v))|_{-\beta}\\
&=\int_{0}^{1}DF(\eta\chi_R(u)+(1-\eta)\chi_{R}(v))(\chi_{R}(u)-\chi_{R}(v))d\eta\\
&\leq D_{\chi}\sup_{u,v\in B_\mathcal{B}(0,R),{\eta\in[0,1]}}\|DF(\eta\chi_R(u)+(1-\eta)\chi_{R}(v))\|_{L(\mathcal{B};\mathcal{B}_{-\beta})}|u-v|\\
&\leq L_{F}(R)|u-v|.
\end{align*} 
Similarly, for $G$, we have 
\begin{align*}
\|G_{R}(u)-G_{R}(v)\|_{L_{2}(\mathcal{B};\mathcal{B})}&=\|G(\chi_{R}(u))-G(\chi_{R}(u))\|_{L_{2}(\mathcal{B};\mathcal{B})}\\
&=\left\|\int_{0}^{1}DG(\eta\chi_R(u)+(1-\eta)\chi_{R}(v))(\chi_{R}(u)-\chi_{R}(v))d\eta\right\|\\
&\leq D_{\chi}\sup_{u,v\in B_\mathcal{B}(0,R)}\|DG(\eta\chi_R(u)+(1-\eta)\chi_{R}(v))\||u-v|\\
&\leq D_{\chi}\sup_{w\in B_\mathcal{B}(0,R)}\|DG(w)(u-v)\|_{L_2(\mathcal{B};\mathcal{B})}\\
&\leq  L_{G}(R)|u-v|,
\end{align*}
and
\begin{align*}
\|G_{R}(u)-G_{R}(v)\|_{L_{2}(\mathcal{B};\mathcal{B}_{-\beta})}&=\|G(\chi_{R}(u))-G(\chi_{R}(u))\|_{L_{2}(\mathcal{B};\mathcal{B}_{-\beta})}\\
&=\left\|\int_{0}^{1}DG(\eta\chi_R(u)+(1-\eta)\chi_{R}(v))(\chi_{R}(u)-\chi_{R}(v))d\eta\right\|_{L_{2}(\mathcal{B};\mathcal{B}_{-\beta})}\\
&\leq D_{\chi}\sup_{u,v\in B_\mathcal{B}(0,R)}\|DG(\eta\chi_R(u)+(1-\eta)\chi_{R}(v))(u-v)\|_{L_{2}(\mathcal{B};\mathcal{B}_{-\beta})}\\
&\leq D_{\chi}\sup_{w\in B_V(0,R)}\|DG(w)(u-v)\|_{L_2(\mathcal{B};\mathcal{B}_{-\beta})}\\
&\leq  L_{G}(R)|u-v|_{-\beta},
\end{align*}
 Since $\mathcal{B}$ is densely embedded into $\mathcal{B}_{-\beta}$, we have, dropping the embedding constant
\begin{align*}
\|G_{R}(u)-G_{R}(v)\|_{L_{2}(\mathcal{B};\mathcal{B}_{-\beta})}\leq  L_{G}(R)|u-v|.
\end{align*}
Lastly, 
\begin{align*}
&\|DG_{R}(u)h-DG_{R}(v)h\|_{L_2(\mathcal{B};\mathcal{B}_{-\beta}))}\\
&\quad
=\left\|\int_{0}^1D^2G(\eta \chi_R(u)+(1-\eta)\chi_{R}(v))d\eta(\chi_R(u)-\chi_{R}(v))h\right\|_{L_{2}(\mathcal{B};\mathcal{B}_{-\beta})}\\
&\quad\leq D_{\chi}\sup_{w\in B_\mathcal{B}(0,R)}\|D^2G(w)(u-v)h\|_{L_2(\mathcal{B};\mathcal{B}_{-\beta}))}\\
&\leq D_{\chi}L_G|u-v|_{-\beta}|h|_{-\beta}.
\end{align*}
Hence for $u_1,v_1,u_2,v_2\in \cB$, we have 
\begin{align*}
&\|G_{R}(u_1)-G_{R}(v_1)-G_{R}(u_2)+G_{R}(v_2)\|_{L_2(\mathcal{B};\mathcal{B}_{-\beta})}\\
&\leq\bigg\|\int_{0}^1\int_0^1D^2G(h(\xi \chi_{R}(u_1)+(1-\xi)\chi_{R}(v_1))+(1-h)(\xi \chi_{R}(u_2)+(1-\xi)\chi_{R}(v_2)))dh\\
&\qquad\times[\xi \chi_{R}( u_1)+(1-\xi)\chi_{R}(v_1)-\xi \chi_{R}(u_2)-(1-\xi)\chi_{R}(v_2)]d\xi (\chi_{R}(u_1)-\chi_{R}(v_1))\bigg\|_{L_2(\mathcal{B};\mathcal{B}_{-\beta})}\\
&\qquad+\left\|\int_{0}^1DG(\xi \chi_{R}(u_2)+(1-\xi)\chi_{R}(v_2))d\xi (\chi_{R}(u_1)-\chi_{R}(v_1)-\chi_{R}(u_2)+\chi_{R}(v_2))\right\|_{L_2(\mathcal{B};\mathcal{B}_{-\beta})}\\
&\leq L_G(|\chi_{R}(u_1)-\chi_{R}(u_2)|_{-\beta}+|\chi_{R}(v_1)-\chi_{R}(v_2)|_{-\beta})|\chi_{R}(u_1)-\chi_{R}(v_1)|_{-\beta}\\
&\quad+L_G(R)|\chi_{R}(u_1)-\chi_{R}(v_1)-\chi_{R}(u_2)-\chi_{R}(v_2)|_{-\beta}\\
&\leq  L_G(D_\chi)^2(|u_1-u_2|_{-\beta}+|v_1-v_2|_{-\beta})|u_1-v_1|_{-\beta}
+L_G(R)D^2_\chi|u_1-v_1-u_2-v_2|_{-\beta}.
\end{align*}
\begin{remark}\label{r}
Due to the previous estimates, one can show that the modified equation~\eqref{eq2.5} (i.e.~replacing $F$ and $G$ by $F_R$ and $G_R$ which are now path-dependent) has a unique global-in-time solution in $C^\beta_{-\beta}([0,T],\cB)$. This follows via a fixed point method as in Section~\ref{assumptions}, see~\cite[Theorem 4.3]{KN22} for a similar argument.
\end{remark}

For  $K>0$, let $R(\theta_i\omega)$ be the solution of 
\begin{align}\label{R}
c_SL_{F}(R(\theta_i\omega))+c_SC_{\alpha,\beta}L_G(R(\theta_i\omega))
  \vertiii{\theta_{i}\omega}_{\beta',0,1}=K.
\end{align}
{The existence of such a random variable is ensured due to the structure of $L_F(R)$ and $L_G(R)$ specified above.}
The random variable $R(\omega)>0$ is tempered from below since $ \vertiii{\theta_{i}\omega}_{\beta',0,1}$ is tempered from above.  Next, we consider the following Lyapunov-Perron-type operator:
\begin{equation}\label{eq-cut}
\begin{aligned}
\relax[\mathcal{J}_{R}(U)]_m(t)&=S^{-}(t+m)\xi
+\sum_{i=1}^{m}S^{-}(t+m-i)\int_{0}^{1} S^{-}(1-\tau)F_{R(\theta_{i-1}\omega)}(u_{i-1}(\tau))d\tau\\
&\quad+\sum_{i=1}^{m}S^{-}(t+m-i)\int_{0}^{1} S^{-}(1-\tau)G_{R(\theta_{i-1}\omega)}(u_{i-1}(\tau))d\theta_{i-1}\omega(\tau)\\
&\quad+\int_{0}^{t} S^{-}(t-\tau)F_{R(\theta_m\omega)}(u_{m}(\tau))d\tau+\int_{0}^{t} S^{-}(t-\tau)G_{R(\theta_m\omega)}(u_{m}(\tau))d\theta_{m}\omega(\tau)\\
&\quad-\sum_{i=m+2}^{\infty}S^{+}(t+m-i)\int_{0}^{1} S^{+}(1-\tau)F_{R(\theta_{i-1}\omega)}(u_{i-1}(\tau))d\tau\\
&\quad-\sum_{i=m+2}^{\infty}
S^{+}
(t+m-i)\int_{0}^{1}S^{+}
(1-\tau)G_{R(\theta_{i-1}\omega)}(u_{i-1}(\tau))d\theta_{i-1}\omega(\tau)\\
&\quad-\int_{t}^{1} S^{+}(t-\tau)F_{R(\theta_{m}\omega)}(u_m(\tau))d\tau-\int_{t}^{1} S^{+}(t-\tau)G_{R(\theta_{m}\omega)}(u_m(\tau))d\theta_{m}\omega(\tau).
\end{aligned}
\end{equation}
This coincides with the Lyapunov-Perron operator of the original equation when $\|u_{i}\|_{\beta,-\beta}\leq \frac{R(\theta_{i}\omega)}{2},$ for all $i=0,1,\cdots$.
\subsection{Existence of local stable manifolds}
\label{sub-exis}
\begin{theorem}\label{theo-exis}
Assume that $(\bf A_1$)-$(\bf A_4$) hold	and that $K$ satisfies the gap condition:
\begin{align*}
K\left(\frac{-1}{1-e^{-(\check{\mu}+\kappa)}}+\frac{1}{1-e^{-(\hat{\mu}+\kappa)}}\right)\leq \frac{1}{2}.
\end{align*}Then for given $\xi\in \mathcal{B}^-$, there exists a unique element $\Gamma(\xi,\omega)=\{\Gamma(\xi,\omega)(i,\cdot)\}_{i\in \mathbb{Z}^+}\in\mathcal{H}^\xi_{\kappa}$ such that 
\begin{align*}
[\mathcal{J}_{R}(\Gamma(\xi,\omega))](i,\cdot)=\Gamma(\xi,\omega)(i,\cdot),~~~~\Gamma^-(\xi,\omega)(0,0)=\xi.
\end{align*} Moreover,  $\Gamma(\xi,\omega)(i,\cdot)$ is Lipschitz with respect to $\xi\in \mathcal{B}^-$.
\end{theorem}
\begin{proof}
We denote each part of $\mathcal{J}_{R}(U,\xi)$, given as \eqref{eq-cut}, as follows,
\begin{align*}
  [\mathcal{J}_{R}(U,\xi)]_{m}(t)&:=S^-(t+m)\xi+I_{1}(U)+I_{2}(U)+I_{3}(U)+I_{4}(U)-I_{5}(U)-I_{6}(U)\\
&\quad-I_{7}(U)-I_{8}(U),
\end{align*}
and estimate all of them individually under the norm $\|\cdot\|_{\beta,-\beta}$. Based on the previous computations, we just need to take into account the exponential contraction and expansion of $S^-$ and $S^+$ given by \eqref{eq.S-nega} and \eqref{eq.S-posi}. First, it follows from Lemma \ref{estim of G} and Theorem \ref{th-exits-solu} that
\begin{align*}
&\|I_3(U)\|_{\beta,-\beta}\leq c_SL_F(R(\theta_{m}\omega))\|u_m\|_{\beta,-\beta},~~
\|I_4(U)\|_{\beta,-\beta}\leq c_SC_{\alpha,\beta}L_G(R(\theta_{m}\omega))\|u_m\|_{\beta,-\beta}\vertiii{\theta_{m}\omega}_{\beta'},\\
&\|I_7(U)\|_{\beta,-\beta}\leq c_SL_F(R(\theta_{m}\omega))\|u_m\|_{\beta,-\beta},~~
\|I_8(U)\|_{\beta,-\beta}\leq c_SC_{\alpha,\beta}L_G(R(\theta_{m}\omega))\|u_m\|_{\beta,-\beta}\vertiii{\theta_{m}\omega}_{\beta'}.
\end{align*}
Afterwards, we need to treat the terms $S^-(\cdot)\xi$ and   $I_1,I_2,I_5,I_6$.  Following the steps of  \eqref{norm0}, we can directly obtain that
\begin{align*}
  \|S^-(\cdot+m)\xi\|_{\beta,-\beta}
  &\leq c_Se^{\check{\mu} m}|\xi|+\sup_{0\leq s<t<1}
  \frac{\|S^-(s+m)\|_{L(\mathcal{B};\mathcal{B})}\|S^-(t-s)-id\|_{L(\mathcal{B};\mathcal{B}_{-\beta})}|\xi|}{(t-s)^{\beta}}\\
  &\leq c_Se^{\check{\mu}m}|\xi|,
\end{align*}
and 
\begin{align*}
  &\|S^-(\cdot+m)(\xi-\xi_0)\|_{\beta,-\beta}\\
  &\quad\leq c_Se^{\check{\mu} m}|\xi-\xi_0|+\sup_{0\leq s<t<1}
  \frac{\|S^-(s+m)\|_{L(\mathcal{B};\mathcal{B})}\|S^-(t-s)-id\|_{L(\mathcal{B};\mathcal{B}_{-\beta})}|\xi-\xi_0|}{(t-s)^{\beta}}\\
  &\quad\leq c_Se^{\check{\mu}m}|\xi-\xi_0|
\end{align*}
for $\xi,\xi_0\in \mathcal{B}^-$. Since for any $i=1,\cdots,m$, $0\leq t<1$, 
\begin{align*}
\max\left\{\|S^-(t+m-i)\|_{L(\mathcal{B};\mathcal{B})},\frac{1}{(t-s)^{\beta}}\|S^-(t+m-i)-S^{-}(s+m-i)\|_{L(\mathcal{B};\mathcal{B}_{-\beta})}\right\}\leq c_Se^{\check{\mu}(m-i)},
\end{align*}
we have 
\begin{align*}
\|I_1(U)\|_{\beta,-\beta}&= \left\|\sum_{i=1}^mS^-(\cdot+m-i)\int_{0}^1 S^{-}(1-\tau)F_{R(\theta_{i-1}\omega)}(u_{i-1}(\tau))d\tau\right\|_{\beta,-\beta}\\
&\leq c_S\sum_{i=1}^m e^{\check{\mu}(m-i)}\left|\int_{0}^1 S^{-}(1-\tau)F_{R(\theta_{i-1}\omega)}(u_{i-1}(\tau))d\tau\right|\\
&\leq c_S\sum_{i=1}^m L_F(R(\theta_{i-1}\omega))e^{\check{\mu}(m-i)}\|u_{i-1}\|_{\beta,-\beta},
\end{align*}
\begin{align*}
\|I_2(U)\|_{\beta,-\beta}&= \left\|\sum_{i=1}^mS^-(\cdot+m-i)\int_{0}^1 S^{-}(1-\tau)G_{R(\theta_{i-1}\omega)}(u_{i-1}(\tau))d\theta_{i-1}\omega(\tau)\right\|_{\beta,-\beta}\\
&\leq c_S\sum_{i=1}^m e^{\check{\mu}(m-i)}\left|\int_{0}^1 S^{-}(1-\tau)G_{R(\theta_{i-1}\omega)}(u_{i-1}(\tau))d\theta_{i-1}\omega(\tau)\right|\\
&\leq c_SC_{\alpha,\beta}\sum_{i=1}^m L_G(R(\theta_{i-1}\omega))e^{\check{\mu}(m-i)}\|u_{i-1}\|_{\beta,-\beta}\vertiii{\theta_{i-1}\omega}_{\beta'}.
\end{align*}  
Similarly, for any $i\geq m+2,0\leq t<1$, it holds
\begin{align*}
     \max\left\{\|S^+(t+m-i)\|_{L(\mathcal{B};\mathcal{B})},\frac{1}{(t-s)^{\beta}}\|S^+(t+m-i)-S^{+}(s+m-i)\|_{L(\mathcal{B};\mathcal{B}_{-\beta})}\right\}\leq c_Se^{\hat{\mu}(m-i+1)}.
\end{align*}
Then we have 
\begin{align*}
\|I_5(U)\|_{\beta,-\beta}&=\left\|\sum_{i=m+2}^\infty S^+(\cdot+m-i)\int_{0}^1 S^{+}(1-\tau)F_{R(\theta_{i-1}\omega)}(u_{i-1}(\tau))d\tau\right\|_{\beta,-\beta}\\
&\leq c_S\sum_{i=m+2}^\infty L_F(R(\theta_{i-1}\omega)){e^{\hat{\mu}(m-i+1)}}\|u_{i-1}\|_{\beta,-\beta},
\end{align*}
and 
\begin{align*}
\|I_6(U)\|_{\beta,-\beta}&=\left\|\sum_{i=m+2}^\infty S^+(\cdot+m-i)\int_{0}^1 S^{+}(1-\tau)G_{R(\theta_{i-1}\omega)}(u_{i-1}(\tau))d\theta_{i-1}\omega(\tau)\right\|_{\beta,-\beta}\\
&\leq c_SC_{\alpha,\beta}\sum_{i=m+2}^\infty L_G(R(\theta_{i-1}\omega)){e^{\hat{\mu}(m-i+1)}}\|u_{i-1}\|_{\beta,-\beta}\vertiii{\theta_{i-1}\omega}_{\beta'}.
\end{align*}
To sum up, we have
\begin{align*}
  &\|[\mathcal{J}_{R}(U,\xi)]_{m}\|_{\beta,-\beta}\\
&\quad\leq c_Se^{\check{\mu}m}|\xi|
+c_S\sum_{i=0}^m e^{\check{\mu}(m-i)}\|u_{i}\|_{\beta,-\beta}\left[L_F(R(\theta_{i}\omega))
+C_{\alpha,\beta} L_G(R(\theta_{i}\omega))\vertiii{\theta_{i}\omega}_{\beta'}\right]\\
&\qquad+c_S\sum_{i=m}^\infty e^{\hat{\mu}(m-i)}\|u_{i}\|_{\beta,-\beta}
\left[L_F(R(\theta_{i}\omega))
+C_{\alpha,\beta} L_G(R(\theta_{i}\omega))\vertiii{\theta_{i}\omega}_{\beta'}\right].
\end{align*}
Referring to  the result of \eqref{U-V:G}, we can similarly obtain that for  $U,V\in \mathcal{H}_{\kappa}$ with $u^-(0)=\xi,v^-(0)=\xi_0$,
\begin{align*}
 & \|[\mathcal{J}_R(U,\xi)-\mathcal{J}_{R}((V,\xi_0))]_{m}\|_{\beta,-\beta}\\
&\leq c_Se^{\check{\mu}m}|\xi-\xi_0|+c_S\sum_{i=0}^m e^{\check{\mu}(m-i)}\|u_{i}-v_{i}\|_{\beta,-\beta}\left[L_F(R(\theta_{i}\omega))
+C_{\alpha,\beta} L_G(R(\theta_{i}\omega))\vertiii{\theta_{i}\omega}_{\beta'}\right]\\
&\quad+c_S\sum_{i=m}^\infty e^{\hat{\mu}(m-i)}\|u_{i}-v_{i}\|_{\beta,-\beta}
\left[L_F(R(\theta_{i}\omega))
+C_{\alpha,\beta} L_G(R(\theta_{i}\omega))\vertiii{\theta_{i}\omega}_{\beta'}\right].
\end{align*}

Lastly, by \eqref{R}, we can conclude that
\begin{align*}
  e^{\kappa m}\|[\mathcal{J}_{R}(U,\xi)]_{m}\|_{\beta,-\beta}
&\leq c_Se^{(\check{\mu}+\kappa)m}|\xi|
+K\sum_{i=0}^m e^{\check{\mu}(m-i)+\kappa (m-i)}e^{\kappa i}\|u_{i}\|_{\beta,-\beta}\\
&\quad+K\sum_{i=m}^\infty e^{\hat{\mu}(m-i)+\kappa (m-i)}e^{\kappa i}\|u_{i}\|_{\beta,-\beta},
\end{align*}
and further 
\begin{align*}
  e^{\kappa m}\|[\mathcal{J}_R(U,\xi)-\mathcal{J}_{R}((V,\xi_0))]_{m}\|_{\beta,-\beta}
&\leq c_Se^{(\check{\mu}+\kappa)m}|\xi-\xi_0|
+K\sum_{i=0}^m e^{\check{\mu}(m-i)+\kappa (m-i)}e^{\kappa i}\|u_{i}-v_i\|_{\beta,-\beta}\\
&\quad+K\sum_{i=m}^\infty e^{\hat{\mu}(m-i)+\kappa (m-i)}e^{\kappa i}\|u_{i}-v_i\|_{\beta,-\beta}.
\end{align*}
Moreover, since
\begin{align*}
\sup_{m\in \mathbb{Z}^+}\sum_{i=0}^m e^{(\check{\mu}+\kappa)(m-i)}
=\frac{-1}{1-e^{-(\check{\mu}+\kappa)}},~~
\sum_{i=m}^\infty e^{(\hat{\mu}+\kappa) (m-i)}=\frac{1}{1-e^{-(\hat{\mu}+\kappa)}},~~~0<\kappa<-\check{\mu},
\end{align*}
we have
\begin{align*}
 \|\mathcal{J}_{R}(U,\xi)\|_{\mathcal{H}_{\kappa}}
&\leq c_S|\xi|+K\left(\frac{-1}{1-e^{-(\check{\mu}+\kappa)}}+\frac{1}{1-e^{-(\hat{\mu}+\kappa)}}\right)\|U\|_{\mathcal{H}_{\kappa}}\\
&\leq c_S|\xi|+\frac{1}{2}\|U\|_{\mathcal{H}_{\kappa}},
\end{align*}
and 
\begin{align*}
 \|\mathcal{J}_{R}(U,\xi)-\mathcal{J}_{R}(V,\xi)\|_{\mathcal{H}_{\kappa}}
&\leq c_S|\xi-\xi_0|+K\left(\frac{-1}{1-e^{-(\check{\mu}+\kappa)}}+\frac{1}{1-e^{-(\hat{\mu}+\kappa)}}\right)\|U-V\|_{\mathcal{H}_{\kappa}}\\
&\leq c_S|\xi-\xi_0|+\frac{1}{2}\|U-V\|_{\mathcal{H}_{\kappa}}.
\end{align*}
Hence by the Banach fixed point theorem,  there is a fixed point $\Gamma(\xi,\omega)\in\mathcal{H}^\xi_{\kappa}$ of $\mathcal{J}_{R}$ for given $\xi\in \mathcal{B}^-$. Since $\vertiii{\omega}_{\beta',0,1}$ is tempered from above, it is obvious that $R(\omega)$ is tempered from below. 

Now, we consider the regularity of $\Gamma(\cdot,\omega)$.  Let  $\Gamma(\xi,\omega)$ and $\Gamma(\xi_0,\omega)$ be the fixed points for  $\xi,\xi_0\in \mathcal{B}^-$. Then we have 
\begin{align*}
\|\Gamma(\xi,\omega)-\Gamma(\xi_0,\omega)\|_{\mathcal{H}_{\kappa}}&\leq \|\mathcal{J}_{R}(\Gamma(\xi,\omega),\xi)-\mathcal{J}_{R}(\Gamma(\xi_0,\omega),\xi_0)\|_{\mathcal{H}_{\kappa}}\\&\leq
\|\mathcal{J}_{R}(\Gamma(\xi,\omega),\xi)-\mathcal{J}_{R}(\Gamma(\xi,\omega),\xi_0)\|_{\mathcal{H}_{\kappa}}\\
&\quad+
\|\mathcal{J}_{R}(\Gamma(\xi,\omega),\xi_0)-\mathcal{J}_{R}(\Gamma(\xi_0,\omega),\xi_0)\|_{\mathcal{H}_{\kappa}}\\&\leq
\sup_{m\in\mathbb{Z}^+}c_Se^{(\check{\mu}+\kappa)m}\|\xi_0-\xi_0\|+\frac{1}{2}\|\Gamma(\xi,\omega)-\Gamma(\xi_0,\omega)\|_{\mathcal{H}_{\kappa}}\\
&\leq c_S|\xi-\xi_0|+\frac{1}{2}\|\Gamma(\xi,\omega)-\Gamma(\xi_0,\omega)\|_{\mathcal{H}_{\kappa}},
\end{align*} 
Hence 
\begin{align*}
\|\Gamma(\xi,\omega)-\Gamma(\xi_0,\omega)\|_{\mathcal{H}_{\kappa}}&\leq 2c_S|\xi-\xi_0|,
\end{align*}
which implies that $\Gamma(\xi,\omega)$ is Lipschitz continuous with respect to $\xi$ with Lipschitz constant $L_{\Gamma}=2c_S$.
\end{proof}

In order to construct the stable manifold, we need the following auxiliary Lemma, compare~\cite[Lemma 5.1]{GarridoAtienza2010}. 
\begin{lemma}\label{le-m+1,m}
For $\xi\in \cB^{-}$, let $u(1)=\Gamma(\xi,\omega)(1,0)$. Then 
\begin{align}\label{m+1,m}
\Gamma(\xi,\omega)(m+1,\cdot)=\Gamma(u^-(1),\theta_{1}\omega)(m,\cdot),~~~m=0,1,\cdots.
\end{align}
\end{lemma}
\begin{proof}
By definition, we have for $t\in[0,1]$
\begin{equation*}
\begin{aligned}
&\Gamma(\xi,\omega)(m+1,t)\\
&=S^{-}(t+m+1)\xi
+\sum_{i=1}^{m+1}S^{-}(t+m+1-i)\int_{0}^{1} S^{-}(1-\tau)F_{R(\theta_{i-1}\omega)}(\Gamma(\xi,\omega)(i-1,\tau))d\tau\\
&\quad+\sum_{i=1}^{m+1}S^{-}(t+m+1-i)\int_{0}^{1} S^{-}(1-\tau)G_{R(\theta_{i-1}\omega)}(\Gamma(\xi,\omega)(i-1,\tau))d\theta_{i-1}\omega(\tau)\\
&\quad~+\int_{0}^{t} S^{-}(t-\tau)F_{R(\theta_{m+1}\omega)}(\Gamma(\xi,\omega)(m+1,\tau))d\tau\\
&\quad+\int_{0}^{t} S^{-}(t-\tau)G_{R(\theta_{m+1}\omega)}(\Gamma(\xi,\omega)(m+1,\tau))d\theta_{m+1}\omega(\tau)\\
&\quad-\sum_{i=m+3}^{\infty}S^{+}(t+m+1-i)\int_{0}^{1} S^{+}(1-\tau)F_{R(\theta_{i-1}\omega)}(\Gamma(\xi,\omega)(i-1,\tau))d\tau\\
&\quad-\sum_{i=m+3}^{\infty}
S^{+}
(t+m+1-i)\int_{0}^{1}S^{+}
(1-\tau)G_{R(\theta_{i-1}\omega)}(\Gamma(\xi,\omega)(i-1,\tau))d\theta_{i-1}\omega(\tau)\\
&\quad-\int_{t}^{1} S^{+}(t-\tau)F_{R(\theta_{m+1}\omega)}(\Gamma(\xi,\omega)(m+1,\tau))d\tau\\
&\quad-\int_{t}^{1} S^{+}(t-\tau)G_{R(\theta_{m+1}\omega)}(\Gamma(\xi,\omega)(m+1,\tau))d\theta_{m+1}\omega(\tau),
\end{aligned}
\end{equation*}
and 
\begin{equation*}
\begin{aligned}
&\Gamma(u^-(1),\theta_{1}\omega)(m,t)\\
&=S^{-}(t+m)u^-(1)
+\sum_{i=1}^{m}S^{-}(t+m-i)\int_{0}^{1} S^{-}(1-\tau)F_{R(\theta_{i}\omega)}(\Gamma(u^-(1),\theta_{1}\omega)(i-1,\tau))d\tau\\
&\quad+\sum_{i=1}^{m}S^{-}(t+m-i)\int_{0}^{1} S^{-}(1-\tau)G_{R(\theta_{i}\omega)}(\Gamma(u^-(1),\theta_{1}\omega)(i-1,\tau))d\theta_{i}\omega(\tau)\\
&\quad+\int_{0}^{t} S^{-}(t-\tau)F_{R(\theta_{m+1}\omega)}(\Gamma(u^-(1),\theta_{1}\omega)(m,\tau))d\tau\\
&\quad+\int_{0}^{t} S^{-}(t-\tau)G_{R(\theta_{m+1}\omega)}(\Gamma(u^-(1),\theta_{1}\omega)(m,\tau))d\theta_{m+1}\omega(\tau)\\
&\quad-\sum_{i=m+2}^{\infty}S^{+}(t+m-i)\int_{0}^{1} S^{+}(1-\tau)F_{R(\theta_{i}\omega)}(\Gamma(u^-(1),\theta_{1}\omega)(i-1,\tau))d\tau\\
&\quad-\sum_{i=m+2}^{\infty}
S^{+}
(t+m-i)\int_{0}^{1}S^{+}
(1-\tau)G_{R(\theta_{i}\omega)}(\Gamma(u^-(1),\theta_{1}\omega)(i-1,\tau))d\theta_{i}\omega(\tau)\\
&\quad-\int_{t}^{1} S^{+}(t-\tau)F_{R(\theta_{m+1}\omega)}(\Gamma(u^-(1),\theta_{1}\omega)(m,\tau))d\tau\\
&\quad-\int_{t}^{1} S^{+}(t-\tau)G_{R(\theta_{m+1}\omega)}(\Gamma(u^-(1),\theta_{1}\omega)(m,\tau))d{\theta_{m+1}\omega(\tau)}.
\end{aligned}
\end{equation*}
For the part containing $S^-(\cdot)$, we have by the definition of mild solution
\begin{align*}
u^-(1)&=S^-(1)\xi+\int_{0}^1S^-(1-\tau)F_{R(\omega)}(\Gamma(\xi,\omega)(0,\tau))d\tau\\
&\quad+\int_{0}^1S^-(1-\tau)G_{R(\omega)}(\Gamma(\xi,\omega)(0,\tau))d\omega(\tau).
\end{align*}
This further leads to 
\begin{equation}\label{eq-m}
\begin{aligned}
&\Gamma^-(u^-(1),\theta_{1}\omega)(m,t)\\
&=S^-(t+m+1)\xi+S^-(t+m)\int_{0}^1S^-(1-\tau)F_{R(\omega)}(\Gamma(\xi,\omega)(0,\tau))d\tau\\
&\quad+S^-(t+m)\int_{0}^1S^-(1-\tau)G_{R(\omega)}(\Gamma(\xi,\omega)(0,\tau))d\omega(\tau)\tau\\
&\quad +\sum_{i=1}^{m}S^{-}(t+m-i)\int_{0}^{1} S^{-}(1-\tau)F_{R(\theta_{i}\omega)}(\Gamma(u^-(1),\theta_{1}\omega)(i-1,\tau))d\tau\\
&\quad+\sum_{i=1}^{m}S^{-}(t+m-i)\int_{0}^{1} S^{-}(1-\tau)G_{R(\theta_{i}\omega)}(\Gamma(u^-(1),\theta_{1}\omega)(i-1,\tau))d\theta_{i}\omega(\tau)\\
&\quad~+\int_{0}^{t} S^{-}(t-\tau)F_{R(\theta_{m+1}\omega)}(\Gamma(u^-(1),\theta_{1}\omega)(m,\tau))d\tau\\
&\quad+\int_{0}^{t} S^{-}(t-\tau)G_{R(\theta_{m+1}\omega)}(\Gamma(u^-(1),\theta_{1}\omega)(m,\tau))d\theta_{m+1}\omega(\tau),
\end{aligned}
\end{equation}
and 
\begin{equation}\label{eq-m+1}
\begin{aligned}
\Gamma^-(\xi,\omega)(m+1,t)
&=S^{-}(t+m+1)\xi+\sum_{i=0}^{m}S^{-}(t+m-i)\int_{0}^{1} S^{-}(1-\tau)F_{R(\theta_{i}\omega)}(\Gamma(\xi,\omega)(i,\tau))d\tau\\
&\quad+\sum_{i=0}^{m}S^{-}(t+m-i)\int_{0}^{1} S^{-}(1-\tau)G_{R(\theta_{i}\omega)}(\Gamma(\xi,\omega)(i,\tau))d\theta_{i}\omega(\tau)\\
&\quad+\int_{0}^{t} S^{-}(t-\tau)F_{R(\theta_{m+1}\omega)}(\Gamma(\xi,\omega)(m+1,\tau))d\tau\\
&\quad+\int_{0}^{t} S^{-}(t-\tau)G_{R(\theta_{m+1}\omega)}(\Gamma(\xi,\omega)(m+1,\tau))d\theta_{m+1}\omega(\tau).
\end{aligned}
\end{equation}
 Then \eqref{eq-m+1}-\eqref{eq-m} gives that
 \begin{align*}
\|\Gamma^-(\xi,\omega)(\cdot+1,\cdot)-\Gamma^-(u^-(1),\theta_{1}\omega)(\cdot,\cdot)\|_{\mathcal{H}_{\kappa}}\leq 
\frac{1}{2}\|\Gamma(\xi,\omega)(\cdot+1,\cdot)-\Gamma(u^-(1),\theta_{1}\omega)(\cdot,\cdot)\|_{\mathcal{H}_{\kappa}},
\end{align*}
which is only satisfied when $\Gamma^-(\xi,\omega)(m+1,\cdot)=\Gamma^-(u^-(1),\theta_{1}\omega)(m,\cdot)$.
Similar produce can be conducted about the positive part and leads to $\Gamma^+(\xi,\omega)(m+1,\cdot)=\Gamma^+(u^-(1),\theta_{1}\omega)(m,\cdot)$. Hence 
we can  conclude that  $\Gamma(\xi,\omega)(m+1,\cdot)=\Gamma(u^-(1),\theta_{1}\omega)(m,\cdot)$ with $u(1)=\Gamma(\xi,\omega)(1,0)$.
\end{proof}
\begin{lemma}\label{le-about U}
Let $\hat{r}$ be a random variable, tempered from below. Then there exists a random positive variable $\hat{\rho}$ which is tempered from below, such that 
\begin{align*}
L_{\Gamma}e^{-i\kappa}\hat{\rho}(\omega)\leq \hat{r}(\theta_{i}\omega), ~~i\in \mathbb{Z}^+.
\end{align*}
\end{lemma}
\begin{proof}
Let $\hat{\rho}(\omega)=\frac{1}{L_{\Gamma}}\inf_{i\in\mathbb{Z}^+}e^{i\kappa}\hat{r}(\theta_{i}\omega)>0$. Then 
\begin{align*}
L_{\Gamma}e^{-i\kappa}\hat{\rho}(\omega)&=e^{-i\kappa}\inf_{j\in\mathbb{Z}^+}e^{j\kappa}\hat{r}(\theta_{j}\omega)
\leq e^{-i\kappa}e^{i\kappa}\hat{r}(\theta_{i}\omega)=\hat{r}(\theta_{i}\omega),
\end{align*}
and 
\begin{align*}
\hat{\rho}(\theta_{i}\omega)e^{\kappa i}&=\frac{1}{L_{\Gamma}}e^{i\kappa}\inf_{j\in\mathbb{Z}^+}e^{j\kappa}\hat{r}(\theta_{j+i}\omega)=\frac{1}{L_{\Gamma}}
\inf_{j\in\mathbb{Z}^+}e^{(j+i)\kappa}\hat{r}(\theta_{j+i}\omega)\\
&=\frac{1}{L_{\Gamma}}\inf_{k\geq i}e^{k\kappa}\hat{r}(\theta_{k}\omega).
\end{align*}
Since $\hat{r}(\omega)$ is tempered from below, for any $\epsilon>0$, there exists an $N>0$ such that for any $k\geq N$, we have 
\begin{align*}
1\geq \frac{1}{L_{\Gamma}}\hat{r}(\theta_{k}\omega)\geq e^{-\epsilon k}.
\end{align*}
Then for such $\epsilon$ and $i\geq N$, we have 
\begin{align*}
e^{i\kappa}\geq \hat{\rho}(\theta_{i}\omega)e^{\kappa i}\geq e^{i(\kappa-\epsilon)},
\end{align*}
that is 
\begin{align*}
1\geq \hat{\rho}(\theta_{i}\omega)\geq e^{-i\epsilon},
\end{align*}
which implies $ \hat{\rho}$ is tempered from below. This completes the proof.
\end{proof}
\begin{lemma}\label{est about W}
Let $\bar{r}=\inf_{s\in[0,1]}\hat{r}(\theta_{s}\omega)$, and $\bar{\rho}(\omega)=\frac{1}{L_\Gamma}\inf_{i\in\mathbb{Z}^+}e^{\kappa i}\bar{r}(\theta_{i}\omega)$. Then for any $\xi\in B_{\mathcal{B}^-}(0,\bar{\rho}(\omega))$, it holds that
\begin{align}
|\Gamma^-(\xi,\omega)(0,s)|\leq L_\Gamma\hat{\rho}(\theta_{s}\omega).
\end{align}
\end{lemma}
\begin{proof}
Note that \begin{align*}
\hat{\rho}(\theta_{s}\omega)&=\frac{1}{L_\Gamma}\inf_{i\in\mathbb{Z}^+}e^{\kappa i}\hat{r}(\theta_{i+s}\omega)\geq \frac{1}{L_\Gamma}\inf_{i\in\mathbb{Z}^+}e^{\kappa i}\inf_{s\in [0,1]}\hat{r}(\theta_{i+s}\omega)\\
&= \frac{1}{L_\Gamma}\inf_{i\in\mathbb{Z}^+}e^{\kappa i}\bar{r}(\theta_{i}\omega)=\bar{\rho}(\omega).
\end{align*}
Herein, we keep in mind that $\bar{\rho}(\omega)\leq \hat{\rho}(\omega)$. Then 
\begin{align*}
|\Gamma^-(\xi,\omega)(0,s)|\leq \|\Gamma(\xi,\omega)\|_{\mathcal{H}_{\kappa}}\le L_\Gamma|\xi|\leq L_\Gamma\bar{\rho}(\omega)\leq L_\Gamma\hat{\rho}(\theta_{s}\omega).
\end{align*}
\end{proof}
For given $R$ in \eqref{R}, let  $\hat{r}(\omega)=\frac{R(\omega)}{2L_\Gamma}$ and $\hat{\rho}(\omega)=\frac{1}{L_{\Gamma}}\inf_{i\in\mathbb{Z}^+}e^{i\kappa}\hat{r}(\theta_{i}\omega)$. Then by the Lipschitz continuity  of $\Gamma(\cdot,\omega)$, we have 
\begin{align}
&\|\Gamma(\xi,\omega)(0,\cdot)\|_{\beta,-\beta}\le L_{\Gamma}|\xi|\leq L_{\Gamma}\hat{r}(\omega)= \frac{R(\omega)}{2},~~\xi\in \overline{B_{\mathcal{B}^-}(0,\hat{r}(\omega))},\label{M-W-1}\\
&\|\Gamma^-(\xi,\omega)(i,\cdot)\|_{\beta,-\beta}\leq L_{\Gamma}e^{-i\kappa}|\xi|\leq L_{\Gamma}e^{-i\kappa}\hat{\rho}(\omega)\leq \hat{r}(\theta_{i}\omega),~~\xi\in\overline{ B_{\mathcal{B}^-}(0,\hat{\rho}(\omega))}.\label{M-W-2}
\end{align}
By Theorem \ref{theo-exis} we obtain a Lipschitz mapping $m(\omega,\cdot)=\Gamma^+(\cdot,\omega)(0,0)$ from $\mathcal{B}^-$ to $\mathcal{B}^+$. 
\begin{theorem}\label{thm:manifold}
The random set $M$ defined as 
\begin{align*}
M(\omega):=\{\xi+\Gamma^+(\xi,\omega)(0,0):\xi\in B_{\mathcal{B}^-}(0,\hat{r}(\omega))\},~~~\omega\in\Omega
\end{align*}
is a local stable manifold for the random dynamical system associated to the SPDE~\eqref{eq2.5}.
\end{theorem} 
\begin{proof} We first consider the case when $L_{\Gamma}>1$.  For  $\bar{\rho}$ as determined in Lemma \ref{est about W}, we define 
 \begin{align*}
    W(\omega)=\{x\in\mathcal{B}:\xi\in B_{\mathcal{B}^-}(0,L_{\Gamma}^{-1}\bar{\rho}(\omega)),\xi=x^-\}.
\end{align*}
Note that $\bar{\rho}(\omega)\leq \hat{\rho}(\omega)$ and then $|\xi|\leq \hat{\rho}(\omega)$ for $x\in W(\omega)$. Hence for any $x\in M(\omega)\cap W(\omega)$, we have \eqref{M-W-1} and \eqref{M-W-2}. Let $u^-(k)=\Gamma^-(u^-(k-1),\theta_{k-1}\omega)(1,0),k=2,3,\cdots$. Then by Lemma \ref{le-m+1,m}, we directly obtain that 
\begin{align*}
&\Gamma^-(\xi,\omega)(2,0)=\Gamma^-(u^-(1),\theta_{1}\omega)(1,0)=u^-(2),
\end{align*}
and
\begin{align*}
&\Gamma^-(\xi,\omega)(3,0)=\Gamma^-(u^-(1),\theta_{1}\omega)(2,0)=\Gamma^-(u^-(1),\theta_{1}\omega)(1,1),
\end{align*}
where the last equality comes from the property $u_{i}(1)=u_{i+1}(0)$ for $\{u_i\}_{i\in\mathbb{Z}^+}\in \mathcal{H}_{\kappa}$. Applying Lemma \ref{le-m+1,m} to $\Gamma^-(u^-(1),\theta_{1}\omega)(1,1)$, we have 
\begin{align*}
\Gamma^-(u^-(1),\theta_{1}\omega)(1,1)=\Gamma^-(\Gamma^-(u^-(1),\theta_{1}\omega)(1,0),\theta_{2}\omega)(0,1).
\end{align*}
Hence 
\begin{align*}
&\Gamma^-(\xi,\omega)(3,0)=\Gamma^-(u^-(2),\theta_{2}\omega)(0,1)
=\Gamma^-(u^-(2),\theta_{2}\omega)(1,0)=u^-(3).
\end{align*}
By induction, one has $\Gamma^-(\xi,\omega)(m,0)=\Gamma^-(u^-(m-1),\theta_{m-1}\omega)(1,0)=u^-(m)$ for all $m\in\mathbb{Z}^+$. Moreover,  iterating \eqref{m+1,m}  indicates 
\begin{align}\label{iteration}
\Gamma(\xi,\omega)(m,\cdot)=\Gamma(u^-(1),\theta_1\omega)(m-1,\cdot)=\cdots=\Gamma(u^-(m),\theta_m\omega)(0,\cdot).
\end{align}
Since  $u^-(m)=\Gamma^-(\xi,\omega)(m,0)$, by $|\xi|\leq \hat{\rho}(\omega)$, we further have
\begin{align}\label{est-s}
\nonumber|u^-(m)|&=|\Gamma^-(\xi,\omega)(m,0)|\leq \|\Gamma^-(\xi,\omega)(m,\cdot)\|_{\beta,-\beta}\leq L_\Gamma e^{-m\kappa}|\xi|\\
&\leq L_\Gamma e^{-m\kappa}\hat{\rho}(\omega)\leq \hat{r}(\theta_{m}\omega),
\end{align}
 implying that
\begin{align}\label{fixed point and solution}
\nonumber\|\Gamma(\xi,\omega)(m,\cdot)\|_{\beta,-\beta}&=\|\Gamma(u^-(m),\theta_m\omega)(0,\cdot)\|_{\beta,-\beta}\leq L_{\Gamma}|u^-(m)|\\
&\leq L_{\Gamma}\hat{r}(\theta_{m}\omega)=\frac{R(\theta_{m}\omega)}{2}.
\end{align}

As discussed in Subsection \ref{se.L-P} and due to \eqref{fixed point and solution}, for $x\in M(\omega)\cap W(\omega)$, $m=0,1,\cdots,$ $s\in[0,1]$, we conclude that $\Gamma(\xi,\omega)(m,s)$ is the solution of \eqref{mild sol} at $m+s$ with $x=\Gamma(\xi,\omega)(0,0)$ as initial {datum,} 
 i.e.
\begin{align*}
\Gamma(\xi,\omega)(m,s)=\varphi(m+s,\omega,x)=\varphi(m,\theta_{s}\omega,\varphi(s,\omega,x)).
\end{align*}
In particular we have 
\begin{align*}
\varphi(s,\omega,x)=\Gamma(\xi,\omega)(0,s)=:u(s).
\end{align*}
By Lemma \ref{est about W}, we have  $|\Gamma^-(\xi,\omega)(0,s)|\leq \hat{\rho}(\theta_{s}\omega)$ and further
\begin{align}
\varphi(m,\theta_{s}\omega,\Gamma(\xi,\omega)(0,s))
=\Gamma(\Gamma^-(\xi,\omega)(0,s),\theta_{s}\omega)(m,0).
\end{align}
Now, we can conclude that 
\begin{align*}
\Gamma(\xi,\omega)(m,s)&=\varphi(m+s,\omega,x)=\varphi(m,\theta_s\omega,u(s))=
\Gamma(u^-(s),\theta_{s}\omega)(m,0),
\end{align*}
for $x\in M(\omega)\cap W(\omega)$, where the first equality implies 
\begin{align*}
    \lim_{t\to \infty}\varphi(t,\omega,x)=0,~~~~\text{exponentially}.
\end{align*}
Furthermore, together with \eqref{iteration}, the second equality entails
\begin{align*}
\varphi(m+s,\omega,x)&=\Gamma(u^-(m+s),\theta_{m+s}\omega)(0,0)
=u^-(m+s)+\Gamma^+(u^-(m+s),\theta_{m+s}\omega)(0,0).
\end{align*}
In which $u^-(k+s):=\Gamma^-(u^-(k-1+s),\theta_{k-1}\omega)(1,0)$ satisfying $|u^-(k+s)|\leq \hat{r}(\theta_{k+s}\omega)$, $k=1,2\cdots, m$, due to \eqref{est-s}. Hence  $\varphi(t,\omega,M(\omega)\cap W(\omega))\subset M(\theta_{t}\omega)$ for any $t\geq 0$, which naturally implies that 
\begin{align*}
    \lim_{|x|\to0}t_{0}(\omega,x)=\infty,~\text{for } x\in M(\omega).
\end{align*}
Hence by Definition \ref{def3.4}, we conclude that $M$ is a local stable manifold for $\varphi$. The same procedure can be conducted for $L_{\Gamma}\leq 1$ using Lemma~\ref{est about W} and replacing $W$ by 
\begin{align*}
    \tilde{W}(\omega):=\{x\in\mathcal{B}:\xi\in B_{\mathcal{B}^-}(0,\bar{\rho}(\omega)),~\xi=x^-\}.
\end{align*} 
\end{proof}
\subsection{Smoothness of local stable manifolds}\label{smoothness}
 Fixing $\xi_0\in \mathcal{B}^-$, 
 denoting  $U=\{u_i\}_{i\in\mathbb{Z^+}}$ with $u_i(t)=\Gamma(\xi_0,\omega)(i,t)$, we now prove that $U$ belongs to $C^k(\mathcal{B}^-;\mathcal{H}_{\kappa})$ (with respect to $\xi_0$) for $k\geq 1$ under suitable assumptions.
\begin{itemize}
 \item [$(\bf A^k_2$)] The mapping $F\in C^k_b(\mathcal{B};\mathcal{B})$  with $F(0)=0,DF(0)=0,\cdots,D^kF(0)=0$ and satisfies the following Lipschitz condition
\begin{align*}
&\|D^lF(u)-D^lF(v)\|_{L^l(\mathcal{B};\mathcal{B})}\leq L_F|u-v|,
\end{align*}
for some constant $L_F>0,l=0,1,\cdots,k-1$.
  \item [$(\bf A^k_3$)] The mapping $G\in C_b^{k+1} (\mathcal{B};L_2(\mathcal{B};\mathcal{B}))$ and $G(0)=DG(0)=\cdots=D^{k}G(0)=0$ as well as
\begin{align*}
&\|D^{l}G(u)-D^{l}G(v)\|_{L^{l}(\mathcal{B};L_2(\mathcal{B};\mathcal{B}))}\leq L_G|u-v|_{-\beta},\\
&\|D^{l+1}G(u)h-D^{l+1}G(v)h\|_{L^{l}(\mathcal{B};L_2(\mathcal{B};\mathcal{B}_{-\beta}))}\leq L_G|u-v|_{-\beta}|h|_{-\beta},
\end{align*}
for some constant $L_G>0,l=0,1,\cdots,k-1$ and $u,v,h\in \mathcal{B}$.
\end{itemize}
Consequently,  the inductive  gap condition becomes
\begin{align}\label{gap k}
K\left(\frac{-1}{1-e^{-(\check{\mu}+k\kappa+\gamma)}}+\frac{1}{1-e^{-(\hat{\mu}+k\kappa+\gamma)}}\right)\leq \frac{1}{2},~~\text{for all}~0\leq \gamma\leq 2\gamma_0.
\end{align}
\begin{remark}
Here we still use the Lipschitz constants $L_F,L_G$. These should strictly depend  on $l$. However, for simplicity, we drop the dependence on $l$ and choose  the constants large enough.
\end{remark}
The idea of the proof relies on a similar to argument to~\cite{DLS2004, LL2010}. However, due to the multiplicative structure of the noise, the computations are technically more involved.
\subsubsection{$C^1$-smoothness}
 We first consider the case of $k=1$ when the gap condition is
\begin{align}\label{gap2}
K\left(\frac{-1}{1-e^{-(\check{\mu}+k\kappa+\gamma)}}+\frac{1}{1-e^{-(\hat{\mu}+k\kappa+\gamma)}}\right)\leq \frac{1}{2},~~\text{for all}~0\leq \gamma\leq 2\gamma_0,~~\text{for  a small}~ \gamma_0>0,
\end{align}
and $(\mathbf{A_2^k})$ and $(\mathbf{A_3^k})$ coincide with $(\mathbf{A_2})$ respectively $(\mathbf{A_3})$ .
We know from Subsection \ref{sub-exis} that $\mathcal{J}_{R}(\cdot,\xi,\omega)$ is contraction on $\mathcal{H}_{\kappa+\gamma}\subset\mathcal{H}_{\kappa}$ for any $0\leq \gamma\leq 2\gamma_0$, which implies $\Gamma(\xi,\omega)\subset \mathcal{H}_{\kappa+\gamma}$. Fixing $\xi_0\in \mathcal{B}^-$, 
and denoting  $U=\{u_i\}_{i\in\mathbb{Z^+}}$ with $u_i(t)=\Gamma(\xi_0,\omega)(i,t)$, we define another operator $\mathcal{L}$ on $\mathcal{H}_{\kappa+\gamma}$ as follows: for $V=\{v_i\}_{i\in\mathbb{Z^+}}\in \mathcal{H}_{\kappa+\gamma}$,
\begin{equation}\label{eq-first D}
\begin{aligned}
&[\mathcal{L}(V)]_m(t)\\
&:=\sum_{i=1}^{m}S^{-}(t+m-i)\int_{0}^{1} S^{-}(1-\tau)DF_{R(\theta_{i-1}\omega)}(u_{i-1}(\tau))v_{i-1}(\tau)d\tau\\
&\quad+\sum_{i=1}^{m}S^{-}(t+m-i)\int_{0}^{1} S^{-}(1-\tau)DG_{R(\theta_{i-1}\omega)}(u_{i-1}(\tau))v_{i-1}(\tau)d\theta_{i-1}\omega(\tau)\\
&\quad+\int_{0}^{t} S^{-}(t-\tau)DF_{R(\theta_m\omega)}(u_{m}(\tau))v_{m}(\tau)d\tau+\int_{0}^{t} S^{-}(t-\tau)DG_{R(\theta_m\omega)}(u_{m}(\tau))v_{m}(\tau)d\theta_{m}\omega(\tau)\\
&\quad-\sum_{i=m+2}^{\infty}S^{+}(t+m-i)\int_{0}^{1} S^{+}(1-\tau)DF_{R(\theta_{i-1}\omega)}(u_{i-1}(\tau))v_{i-1}(\tau)d\tau\\
&\quad-\sum_{i=m+2}^{\infty}
S^{+}
(t+m-i)\int_{0}^{1}S^{+}
(1-\tau)DG_{R(\theta_{i-1}\omega)}(u_{i-1}(\tau))v_{i-1}(\tau)d\theta_{i-1}\omega(\tau)\\
&\quad-\int_{t}^{1} S^{+}(t-\tau)DF_{R(\theta_{m}\omega)}(u_m(\tau))v_{m}(\tau)d\tau-\int_{t}^{1} S^{+}(t-\tau)DG_{R(\theta_{m}\omega)}(u_m(\tau))v_{m}(\tau)d\theta_{m}\omega(\tau)\\
&=:\mathcal{L}_1(V)+\mathcal{L}_2(V)+\mathcal{L}_3(V)+\mathcal{L}_4(V)-\mathcal{L}_5(V)-\mathcal{L}_6(V)-\mathcal{L}_7(V)-\mathcal{L}_8(V).
\end{aligned}
\end{equation}
Before we go to the further step, we make some elementary considerations  on $DG_{R}(u)v$. Similar estimates were conducted for the truncated coefficients $F_R$ and $G_R$.  First, since $G\in C^2_b(\mathcal{B};L_2(\mathcal{B};\mathcal{B}))$ with $DG(0)=0$, for $|u|\leq R$, we have as seen before that
\begin{equation}
\begin{aligned}
\|DG_R(u)v\|_{L_2(\mathcal{B};\mathcal{B}_{-\beta})}&=\|DG_R(u)v-DG_R(0)v\|_{L_2(\mathcal{B};\mathcal{B}_{-\beta})}\\
&\leq L_G|u||v|\leq L_G(R)|v|.
\end{aligned}
\end{equation}
We know that $L_G(R)\to 0$ as $R\to 0$.
Similarly by \eqref{G4}\eqref{G2}, $DG(0)=0$ and the continuity of $DG$, for $\|u\|_{\beta,-\beta}\leq R$, we have
\begin{equation}
\begin{aligned}
&\|DG_R(u(r))v(r)-DG_R(u(q))v(q)\|_{L_2(\mathcal{B};\mathcal{B}_{-\beta})}\\
&\leq 
\|DG_R(u(r))v(r)-DG_R(u(r))v(q)\|_{L_2(\mathcal{B};\mathcal{B}_{-\beta})}+\|DG_R(u(r))v(q)-DG_R(u(q))v(q)\|_{L_2(\mathcal{B};\mathcal{B}_{-\beta})}\\
&\leq \sup_{\|w\|_{\beta,-\beta}\leq R}\|DG(w)\|_{L(\mathcal{B}_{-\beta};L_2(\mathcal{B};\mathcal{B}_{-\beta}))}|v(r)-v(q)|_{-\beta}+L_G|u(r)-u(q)|_{-\beta}|v(q)|\\
&\leq L_G(R)\|v\|_{\beta,-\beta}(r-q)^{\beta}.
\end{aligned}\end{equation}
Based on the last two results and (\ref{G1}), (\ref{G4}), we infer
\begin{equation}\label{DG-1}
\begin{aligned}
&\|D_{s+}^{\alpha}S(t-\cdot)DG_{R}(u(\cdot))v(\cdot)[r]\|_{L_2(\mathcal{B};\mathcal{B})}\\
&\quad\leq \frac{1}{\Gamma(1-\alpha)}\bigg(\frac{\|S(t-r)DG_R(u(r))v(r)\|_{L_2(\mathcal{B};\mathcal{B})}}{(r-s)^{\alpha}}\\
&\qquad+\alpha\int_{s}^{r}\frac{\|S(t-r)DG_R(u(r))v(r)-S(t-q)DG_R(u(q))v(q)\|_{L_2(\mathcal{B};\mathcal{B})}}{(r-q)^{\alpha+1}}dq\bigg)\\
&\quad\leq c_SC_{\alpha}{L_G(R)(r-s)^{-\alpha}|v(r)|}+C_{\alpha}\int_{s}^r\frac{\|S(t-r)-S(t-q)\|_{L(\mathcal{B};\cB)}\|DG_R(u(r))v(r)\|_{L_2(\mathcal{B};\mathcal{B})}}{(r-q)^{1+\alpha}}dq\\
&\qquad+C_{\alpha}\int_{s}^r\frac{\|S(t-q)\|_{L({\mathcal{B}_{-\beta}};\mathcal{B})}\|DG_R(u(r))v(r)-DG_R(u(q))v(q)\|_{L_2(\mathcal{B};{\mathcal{B}_{-\beta}})}}{(r-q)^{1+\alpha}}dq\\
&\quad\leq c_SC_{\alpha}L_G(R)(r-s)^{-\alpha}\|v\|_{\beta,-\beta}+c_SC_{\alpha}L_G(R)\|v\|_{\beta,-\beta}(t-r)^{-\beta}\int_{s}^r(r-q)^{\beta-\alpha-1}dq\\
&\qquad+c_SC_{\alpha}L_G(R)\|v\|_{\beta,-\beta}\int_{s}^r{(t-q)^{-\beta}}(r-q)^{\beta-\alpha-1}dq\\
&\quad\leq c_SC_{\alpha,\beta}L_G(R)\|v\|_{\beta,-\beta}[(r-s)^{-\alpha}+(t-r)^{-\beta}(r-s)^{\beta-\alpha}],
\end{aligned}
\end{equation}
and
\begin{equation}\label{DG-2}
\begin{aligned}
&\|D_{s+}^{\alpha}S(t-\cdot)DG_{R}(u(\cdot))v(\cdot)[r]\|_{L_2(\mathcal{B};\mathcal{B}_{-\beta})}\\
&\quad\leq \frac{1}{\Gamma(1-\alpha)}\bigg(\frac{\|S(t-r)DG_R(u(r))v(r)\|_{L_2(\mathcal{B};\mathcal{B}_{-\beta})}}{(r-s)^{\alpha}}\\
&\qquad+\alpha\int_{s}^{r}\frac{\|S(t-r)DG_R(u(r))v(r)-S(t-q)DG_R(u(q))v(q)\|_{L_2(\mathcal{B};\mathcal{B}_{-\beta})}}{(r-q)^{\alpha+1}}dq\bigg)\\
&\quad\leq c_SC_{\alpha}{L_G(R)(r\!-\!s)^{-\alpha}|v(r)|}\!+\!C_{\alpha}\int_{s}^r\frac{\|S(t-r)\!-\!S(t-q)\|_{L(\mathcal{B};\mathcal{B}_{-\beta})}\|DG_R(u(r))v(r)\|_{L_2(\mathcal{B};\mathcal{B})}}{(r-q)^{1+\alpha}}dq\\
&\qquad+C_{\alpha}\int_{s}^r\frac{\|S(t-q)\|_{L(\mathcal{B}_{-\beta};\mathcal{B}_{-\beta})}\|DG_R(u(r))v(r)-DG_R(u(q))v(q)\|_{L_2(\mathcal{B};\mathcal{B}_{-\beta})}}{(r-q)^{1+\alpha}}dq\\
&\quad\leq c_SC_{\alpha}L_G(R)(r-s)^{-\alpha}\|v\|_{\beta,-\beta}+c_SC_{\alpha}L_G(R)\|v\|_{\beta,-\beta}\int_{s}^r(r-q)^{\beta-\alpha-1}dq\\
&\qquad+c_SC_{\alpha}L_G(R)\|v\|_{\beta,-\beta}\int_{s}^r(r-q)^{\beta-\alpha-1}dq\\
&\quad\leq c_SC_{\alpha,\beta}L_G(R)\|v\|_{\beta,-\beta}[(r-s)^{-\alpha}+(r-s)^{\beta-\alpha}+(r-s)^{\beta-\alpha}].
\end{aligned}
\end{equation}
 Furthermore, for $\alpha'-\alpha>0$, we infer that
\begin{equation}\label{DG-3}
\begin{aligned}
&\|D_{0+}^{\alpha}[S(t-\cdot)-S(s-\cdot)]DG_R(u(\cdot))v(\cdot)[r]\|_{L_2(\mathcal{B};\mathcal{B}_{-\beta})}\\
&\quad\leq \frac{C_{\alpha}\|[S(t-r)-S(s-r)]DG_R(u(r))v(r)\|_{L_2(\mathcal{B};\mathcal{B}_{-\beta})}}{r^{\alpha}}\\
&\qquad+C_{\alpha}\int_{0}^{r}\frac{1}{(r-q)^{\alpha+1}}\|[S(t-r)-S(s-r)]DG_R(u(r))v(r)\\
&\qquad-[S(t-q)-S(s-q)]DG_R(u(r))v(q)\|_{L_2(\mathcal{B};\mathcal{B}_{-\beta})}dq\\
&\quad\leq \frac{C_{\alpha}\|S(t-s)-\text{id}\|_{L(\mathcal{B};\mathcal{B}_{-\beta})}\|S(s-r)\|_{L(\cB;\cB)}\|DG_R(u(r))v(r)\|_{L_2(\mathcal{B};\mathcal{B})}}{r^{\alpha}}\\
&\qquad+C_{\alpha}\int_{0}^{r}\frac{\|[S(t-r)-S(s-r)-S(t-q)+S(s-q)]DG_R(u(r))v(r)\|_{L_2(\mathcal{B};\mathcal{B}_{-\beta})}}{(r-q)^{\alpha+1}}dq\\
&\qquad+C_{\alpha}\int_{0}^{r}\frac{\|[S(t-q)-S(s-q)][DG_R(u(r))v(r)-DG_R(u(r))v(q)]\|_{L_2(\mathcal{B};\mathcal{B}_{-\beta})}}{(r-q)^{\alpha+1}}dq\\
&\quad\leq \frac{C_{\alpha}\|S(t-s)-\text{id}\|_{L(\mathcal{B};\mathcal{B}_{-\beta})}\|S(s-r)\|_{L(\cB;\cB)}\|DG_R(u(r))v(r)\|_{L_2(\mathcal{B};\mathcal{B})}}{r^{\alpha}}\\
&\qquad+C_{\alpha}\int_{0}^{r}\frac{1}{(r-q)^{\alpha+1}}\|S(t-s)-\text{id}\|_{L(\mathcal{B};\mathcal{B}_{-\beta})}\|S(s-r)\|_{L(\mathcal{B}_{-\alpha'};\mathcal{B})}\\
&\qquad\times\|\text{id}-S(r-q)\|_{L(\cB;\cB_{-\alpha'})}\|DG_R(u(r))v(r)\|_{L_2(\mathcal{B};\mathcal{B})}dq\\
&\qquad+C_{\alpha}\int_{0}^{r}\frac{1}{(r-q)^{\alpha+1}}\|S(t-s)-\text{id}\|_{L(\mathcal{B};\mathcal{B}_{-\beta})}\|S(s-q)\|_{L(\mathcal{B}_{-\beta};\mathcal{B})}\\
&\qquad\times\|DG_R(u(r))v(r)-DG_R(u(r))v(q)\|_{L_2(\mathcal{B};\mathcal{B}_{-\beta})}dq\\
&\quad\leq c_{S}C_{\alpha,\beta}L_G(R)\|v\|_{\beta,-\beta}(t-s)^{\beta}\bigg[r^{-\alpha}+\int_0^r(s-r)^{-\alpha'}(r-q)^{\alpha'-\alpha-1}dq\\
&\qquad+\int_{0}^r(s-q)^{-\beta}(r-q)^{\beta-\alpha-1}\bigg]\\
&\quad\leq c_{S}C_{\alpha,\beta}L_G(R)\|u\|_{\beta,-\beta}(t-s)^{\beta}{[(s-r)^{-\alpha'}r^{\alpha'-\alpha}+(s-r)^{-\beta}r^{\beta-\alpha}]}.
\end{aligned}
\end{equation}
\begin{lemma}
The operator $\mathcal{L}:\mathcal{H}_{\kappa+\gamma}\to \mathcal{H}_{\kappa+\gamma}$ is  bounded for any $\gamma\in[0,2\gamma_0]$.
\end{lemma}
\begin{proof}
 The fact that $\mathcal{L}$ is well-defined on $\mathcal{H}_{\kappa+\gamma}$ can be obtained by a similar approach as for $\mathcal{J}$ using ($\mathbf{A_2}$), ($\mathbf{A_3}$) and (\ref{gap2}). Therefore we omit the details. We mainly focus on its  boundedness.
Note that  $F\in C^1_b(\mathcal{B};\mathcal{B})$ with $DF(0)=0$, thus
\begin{align*}
|DF_{R(\theta_{i-1}\omega)}(u_{i-1})v_{i-1}|\leq L_F(R(\theta_{i-1}\omega))|v_{i-1}|.
\end{align*}
Then by the similar arguments as for $I_1$, $I_5$, $I_3$ and $I_7$ in  Theorem \ref{theo-exis}, we have 
\begin{align*}
\|\mathcal{L}_1(V)\|_{\beta,-\beta}&\leq \sum_{i=1}^m\left\|S^{-}(\cdot+m-i)\int_{0}^{1} S^{-}(1-\tau)DF_{R(\theta_{i-1}\omega)}(u_{i-1}(\tau))v_{i-1}(\tau)d\tau\right\|_{\beta,-\beta}\\
&\leq c_{S}\sum_{i=1}^{m}L_F(R(\theta_{i-1}\omega))e^{\check{\mu}(m-i)}\|v_{i-1}\|_{\beta,-\beta},\\
\|\mathcal{L}_3(V)\|_{\beta,-\beta}&\leq \left\|\int_{0}^{\cdot} S^{-}(\cdot-\tau)DF_{R(\theta_{m}\omega)}(u_{m}(\tau))v_{m}(\tau)d\tau\right\|_{\beta,-\beta}
\leq c_{S}L_F(R(\theta_{m}\omega))\|v_{m}\|_{\beta,-\beta},\\
\|\mathcal{L}_5(V)\|_{\beta,-\beta}&\leq \sum_{i=m+2}^\infty\left\|S^{+}(\cdot+m-i)\int_{0}^{1} S^{+}(1-\tau)DF_{R(\theta_{i-1}\omega)}(u_{i-1}(\tau))v_{i-1}(\tau)d\tau\right\|_{\beta,-\beta}\\
&\leq c_{S}\sum_{i=m+2}^{\infty}L_F(R(\theta_{i-1}\omega))e^{\hat{\mu}(m-i)}\|v_{i-1}\|_{\beta,-\beta},\\
\|\mathcal{L}_7(V)\|_{\beta,-\beta}&\leq \left\|\int_{\cdot}^0 S^{+}(\cdot-\tau)DF_{R(\theta_{m}\omega)}(u_{m}(\tau))v_{m}(\tau)d\tau\right\|_{\beta,-\beta}
\leq c_{S}L_F(R(\theta_{m}\omega))\|v_{m}\|_{\beta,-\beta}.
\end{align*}

On the other hand, \eqref{DG-1},\eqref{DG-2},\eqref{DG-3} entail that
\begin{align*}
&\left|\int_{s}^tS(t-r)DG_R(u(r))v(r)d\omega(r)\right|\\
&\quad\leq c_SC_{\alpha,\beta}L_G(R)\|v\|_{\beta,-\beta}\vertiii{\omega}_{\beta',0,1}\int_s^t(t-r)^{\alpha+\beta'-1}(r-s)^{-\alpha}dr\\
&\qquad+c_SC_{\alpha,\beta}L_G(R)\|v\|_{\beta,-\beta}\vertiii{\omega}_{\beta',0,1}\int_s^t(t-r)^{\alpha+\beta'-1}{(t-r)^{-\beta}(r-s)^{\beta-\alpha}dr}\\
&\quad\leq c_SC_{\alpha,\beta}L_G(R)\|v\|_{\beta,-\beta}\vertiii{\omega}_{\beta',0,1}{(t-s)^{\beta'},}
\end{align*}
and 
\begin{align*}
&\left|\int_{s}^tS(t-r)DG_R(u(r))v(r)d\omega(r)\right|_{-\beta}\\
&\quad\leq c_SC_{\alpha,\beta}L_G(R)\|v\|_{\beta,-\beta}\vertiii{\omega}_{\beta',0,1}\int_s^t(t-r)^{\alpha+\beta'-1}[(r-s)^{-\alpha}+(r-s)^{\beta-\alpha}]dr\\
&\quad\leq c_SC_{\alpha,\beta}L_G(R)\|v\|_{\beta,-\beta}\vertiii{\omega}_{\beta',0,1}\left[(t-s)^{\beta'}+(t-s)^{\beta+\beta'}\right],
\end{align*}
and 
\begin{align*}
&\left|\int_0^s [S(t-r)-S(s-r)]DG_R(u(r))v(r)d\omega(r)\right|_{-\beta}\\
&\quad\leq C_{\alpha,\beta}c_{S}L_G(R)\|u\|_{\beta,-\beta}(t-s)^{\beta}\vertiii{\omega}_{\beta',0,1}\int_0^s(s-r)^{\alpha+\beta'-1}{[(s-r)^{-\alpha'}r^{\alpha'-\alpha}+(s-r)^{-\beta}r^{\beta-\alpha}]}dr\\
&\quad \leq C_{\alpha,\beta}c_{S}L_G(R)\|u\|_{\beta,-\beta}(t-s)^{\beta}\vertiii{\omega}_{\beta',0,1}(t-s)^{\beta}s^{\beta'}.
\end{align*}
 Hence, by the definition of the norm on $C^\beta_{-\beta}$, we have
\begin{align*}
&\left\|\int_0^\cdot S(\cdot-r)DG_R(u(r))v(r)d\omega(r)\right\|_{\beta,-\beta}\\
&\quad=\sup_{t\in [0,1]}\left|\int_0^t S(t-r)DG_R(u(r))v(r)d\omega(r)\right|\\
&\qquad+ \sup_{s<t\in [0,1]}\frac{\left|\int_0^t S(t-r)DG_R(u(r))v(r)d\omega(r)-\int_0^s S(s-r)DG_R(u(r))v(r)d\omega(r)\right|_{-\beta}}{(t-s)^{\beta}}\\
&\quad\leq \sup_{t\in [0,1]}\left|\int_0^t S(t-r)DG_R(u(r))v(r)d\omega(r)\right|+
 \sup_{s<t\in [0,1]}\frac{\left|\int_s^t S(t-r)DG_R(u(r))v(r)d\omega(r)\right|_{-\beta}}{(t-s)^{\beta}}\\
&\qquad+\sup_{s<t\in [0,1]}\frac{\left|\int_0^s [S(t-r)-S(s-r)]DG_R(u(r))v(r)d\omega(r)\right|_{-\beta}}{(t-s)^{\beta}}\\
&\quad\leq c_SC_{\alpha,\beta}L_G(R)\|v\|_{\beta,-\beta}\vertiii{\omega}_{\beta',0,1}.
\end{align*}
Then we further obtain 
\begin{align*}
\|\mathcal{L}_2(V)\|_{\beta,-\beta}&\leq \sum_{i=1}^m\left\|S^{-}(\cdot+m-i)\int_{0}^{1} S^{-}(1-\tau)DG_{R(\theta_{i-1}\omega)}(u_{i-1}(\tau))v_{i-1}(\tau)d\theta_{i-1}\omega\tau\right\|_{\beta,-\beta}\\
&\leq c_{S}C_{\alpha,\beta}\sum_{i=1}^{m}L_G(R(\theta_{i-1}\omega))\vertiii{\theta_{i-1}\omega}_{\beta',0,1}e^{\check{\mu}(m-i)}\|v_{i-1}\|_{\beta,-\beta},\\
\|\mathcal{L}_4(V)\|_{\beta,-\beta}&\leq \left\|\int_{0}^{\cdot} S^{-}(\cdot-\tau)DG_{R(\theta_{m}\omega)}(u_{m}(\tau))v_{m}(\tau)d\theta_{m}\omega(\tau)\right\|_{\beta,-\beta}\\
&\leq c_{S}C_{\alpha,\beta}L_G(R(\theta_{m}\omega))\vertiii{\theta_{m}\omega}_{\beta',0,1}\|v_{m}\|_{\beta,-\beta},\\
\|\mathcal{L}_6(V)\|_{\beta,-\beta}&\leq \sum_{i=m+2}^\infty\left\|S^{+}(\cdot+m-i)\int_{0}^{1} S^{+}(1-\tau)DG_{R(\theta_{i-1}\omega)}(u_{i-1}(\tau))v_{i-1}(\tau)d\theta_{i-1}\omega(\tau)\right\|_{\beta,-\beta}\\
&\leq c_{S}C_{\alpha,\beta}\sum_{i=m+2}^{\infty}L_G(R(\theta_{i-1}\omega))e^{\hat{\mu}(m-i)}\|v_{i-1}\|_{\beta,-\beta}\vertiii{\theta_{i-1}\omega}_{\beta',0,1},\\
\|\mathcal{L}_8(V)\|_{\beta,-\beta}&\leq \left\|\int_{\cdot}^0 S^{+}(\cdot-\tau)DG_{R(\theta_{m}\omega)}(u_{m}(\tau))v_{m}(\tau)d\theta_{m}\omega(\tau)\right\|_{\beta,-\beta}\\
&\leq c_{S}C_{\alpha,\beta}L_G(R(\theta_{m}\omega))\vertiii{\theta_{m}\omega}_{\beta',0,1}\|v_{m}\|_{\beta,-\beta}.
\end{align*}
In conclusion, we have
\begin{align*}
  \|[\mathcal{L}(V)]_{m}\|_{\beta,-\beta}
&\leq c_S\sum_{i=0}^m e^{\check{\mu}(m-i)}\|v_{i}\|_{\beta,-\beta}\left[L_F(R(\theta_{i}\omega))
+C_{S,\alpha,\beta} L_G(R(\theta_{i}\omega))\vertiii{\theta_{i}\omega}_{\beta'}\right]\\
&\quad+c_S\sum_{i=m}^\infty e^{\hat{\mu}(m-i)}\|v_{i}\|_{\beta,-\beta}
\left[L_F(R(\theta_{i}\omega))
+C_{\alpha,\beta} L_G(R(\theta_{i}\omega))\vertiii{\theta_{i}\omega}_{\beta'}\right],
\end{align*}
At last, combining the gap condition \eqref{gap2}, we conclude that
\begin{equation}\label{eq-esti-L}
\begin{aligned}
\|[\mathcal{L}(V)]\|_{\mathcal{H}_{\kappa+\gamma}}&\leq K\sup_{m\in\mathbb{Z}^+}\left[ \sum_{i=0}^m e^{(\check{\mu}+\kappa+\gamma)(m-i)}+\sum_{i=m}^\infty e^{(\hat{\mu}+\kappa+\gamma)(m-i)}\right]\|V\|_{\mathcal{H}_{\kappa+\gamma}}\\
&\leq \frac{1}{2} \|V\|_{\mathcal{H}_{\kappa+\gamma}}.
\end{aligned}\end{equation}
 Hence $\mathcal{L}$ is a bounded linear operator and also a uniform contraction on $\mathcal{H}_{\kappa+\gamma}$. 
\end{proof}
\begin{remark}
Naturally the operator $id-\mathcal{L}$ is {injective} on $\mathcal{H}_{\kappa+\gamma}$. It remains to show that this is also {surjective}, meaning that for any $V\in \mathcal{H}_{\kappa+\gamma}$, one can find a $V'\in\mathcal{H}_{\kappa+\gamma}$ such that 
\begin{align*}
(\mathcal{L}-id)V'=V.
\end{align*} To this aim, we set $\mathcal{L}_{V}x:=\mathcal{L}x-V$ for $x\in \mathcal{H}_{\kappa+\gamma}$. This is  a contraction and has therefore a fixed point $V'$ such that 
\begin{align*}
\mathcal{L}V'-V=V',
\end{align*}
that is 
\begin{align*}
\mathcal{L}V'-V'=V.
\end{align*}
In conclusion $id-\mathcal{L}$ is {surjective}. Therefore it has an inverse which is a bounded linear operator on $\mathcal{H}_{\kappa+\gamma}$.
\end{remark}
With $\mathcal{L}$ at hand, we define another operator $\mathcal{I}$ on $\mathcal{H}_{\kappa+\gamma}$ with $\gamma\in[0,\gamma_0]$, which is given by 
\begin{equation}\label{defin-I}
\begin{aligned}
&[\mathcal{I}(V-U)]_{m}(t)=[V-U]_{m}(t)-[\mathcal{L}(V-U)]_{m}(t)-S^-(t+m)(\xi-\xi_0).
\end{aligned}
\end{equation}
We claim that $\|\mathcal{I}(V-U)\|_{\mathcal{H}_{\kappa+\gamma}}=o(|\xi-\xi_{0}|)$ as $\xi\to\xi_{0}$. Then it yields
\begin{align*}
  v_{m}(t)-u_{m}(t)-[\mathcal{L}(V-U)]_{m}(t)
  =S(t+m)(\xi-\xi_{0})+o(|\xi-\xi_{0}|)
\end{align*}
as $\xi\to\xi_{0}$ for any $t\in [0,1],m\in\mathbb{Z}^+$, which implies 
\begin{align*}
  v_{m}(t)-u_{m}(t)&=(\text{id}-\mathcal{L})^{-1}S(t+m)(\xi-\xi_{0})+o(|\xi_{0}-\xi|).
\end{align*}
Hence $u(\cdot,\xi,\omega)$ is differentiable with respect to $\xi\in\cB^-$ and its derivative satisfies $D_{\xi}u(\cdot,\xi,\omega)\in L(\mathcal{B}^-;\mathcal{H}_{\kappa+\gamma})$.

\begin{lemma}
The operator defined in \eqref{defin-I} satisfies  $\|\mathcal{I}(V-U)\|_{\mathcal{H}_{\kappa+\gamma}}=o(|\xi-\xi_{0}|)$ as $\xi\to\xi_{0}$.
\end{lemma}
\begin{proof}
For an $m_1>0$ to be determined, we decompose $\mathcal{I}$ into a sum as follows:
\begin{align*}
[\mathcal{I}(V-U)]_{m}(t):=\mathcal{I}^-_1(t)+\mathcal{I}^-_2(t)+\mathcal{I}^+_1(t)+\mathcal{I}^+_2(t),
\end{align*}
where
\begin{equation*}
 {\mathcal{I}^-_1(t)=
	      \footnotesize \begin{cases}
\sum_{i=1}^{m}S^{-}(t+m-i)\int_{0}^{1} S^{-}(1-\tau)[F_{R(\theta_{i-1}\omega)}(v_{i-1}(\tau))-F_{R(\theta_{i-1}\omega)}(u_{i-1}(\tau))\\
\quad- DF_{R(\theta_{i-1}\omega)}(u_{i-1}(\tau))(v_{i-1}(\tau)-u_{i-1}(\tau))]d\tau\\
\quad+\sum_{i=1}^{m}S^{-}(t+m-i)\int_{0}^{1} S^{-}(1-\tau)[G_{R(\theta_{i-1}\omega)}(G_{i-1}(\tau))-G_{R(\theta_{i-1}\omega)}(u_{i-1}(\tau))\\
\quad- DG_{R(\theta_{i-1}\omega)}(u_{i-1}(\tau))(v_{i-1}(\tau)-u_{i-1}(\tau))]d\theta_{i-1}\omega(\tau)\\
\quad+\int_{0}^{t} S^{-}(t-\tau)[F_{R(\theta_{m}\omega)}(v_{m}(\tau))-F_{R(\theta_{m}\omega)}(u_{m}(\tau))\\
\quad- DF_{R(\theta_{m}\omega)}(u_{m}(\tau))(v_{m}(\tau)-u_{m}(\tau))]d\tau\\
\quad+\int_{0}^{t} S^{-}(t-\tau)[G_{R(\theta_{m}\omega)}(v_{m}(\tau))-G_{R(\theta_{m}\omega)}(u_{m}(\tau))\\
\quad- DG_{R(\theta_{m}\omega)}(u_{m}(\tau))(v_{m}(\tau)-u_{m}(\tau))]d\theta_{m}\omega(\tau),
 & 1\leq m<m_1, \\
	      0, & m\geq m_1,
	       \end{cases}}
\end{equation*}
\begin{equation*}
 \mathcal{I}^-_2(t)=
	      \footnotesize   \begin{cases}
{0,}& 1\leq m<m_1 \\
\sum_{i=1}^{m}S^{-}(t+m-i)\int_{0}^{1} S^{-}(1-\tau)[F_{R(\theta_{i-1}\omega)}(v_{i-1}(\tau))-F_{R(\theta_{i-1}\omega)}(u_{i-1}(\tau))\\
\quad- DF_{R(\theta_{i-1}\omega)}(u_{i-1}(\tau))(v_{i-1}(\tau)-u_{i-1}(\tau))]d\tau\\
\quad+\sum_{i=1}^{m}S^{-}(t+m-i)\int_{0}^{1} S^{-}(1-\tau)[G_{R(\theta_{i-1}\omega)}(v_{i-1}(\tau))-G_{R(\theta_{i-1}\omega)}(u_{i-1}(\tau))\\
\quad- DG_{R(\theta_{i-1}\omega)}(u_{i-1}(\tau))(v_{i-1}(\tau)-u_{i-1}(\tau))]d\theta_{i-1}\omega(\tau)\\
\quad+\int_{0}^{t} S^{-}(t-\tau)[F_{R(\theta_{m}\omega)}(v_{m}(\tau))-F_{R(\theta_{m}\omega)}(u_{m}(\tau))\\
\quad- DF_{R(\theta_{m}\omega)}(u_{m}(\tau))(v_{m}(\tau)-u_{m}(\tau))]d\tau\\
\quad+\int_{0}^{t} S^{-}(t-\tau)[G_{R(\theta_{m}\omega)}(v_{m}(\tau))-G_{R(\theta_{m}\omega)}(u_{m}(\tau))\\
\quad- DG_{R(\theta_{m}\omega)}(u_{m}(\tau))(v_{m}(\tau)-u_{m}(\tau))]d\theta_{m}\omega(\tau),
 & m\geq m_1,
	       \end{cases}
\end{equation*}
\begin{equation*}
 \mathcal{I}^+_1(t)=
	      \footnotesize   \begin{cases}
-\sum_{i=m+2}^{m_1}S^{+}(t+m-i)\int_{0}^{1} S^{+}(1-\tau)[F_{R(\theta_{i-1}\omega)}(v_{i-1}(\tau))-F_{R(\theta_{i-1}\omega)}(u_{i-1}(\tau))\\
\quad- DF_{R(\theta_{i-1}\omega)}(u_{i-1}(\tau))(v_{i-1}(\tau)-u_{i-1}(\tau))]d\tau\\
\quad-\sum_{i=m+2}^{m_1}
S^{+}
(t+m-i)\int_{0}^{1}S^{+}
(1-\tau)[G_{R(\theta_{i-1}\omega)}(v_{i-1}(\tau))-G_{R(\theta_{i-1}\omega)}(u_{i-1}(\tau))\\
\quad- DG_{R(\theta_{i-1}\omega)}(u_{m}(\tau))(v_{i-1}(\tau)-u_{i-1}(\tau))]d\theta_{i-1}\omega(\tau)\\
\quad-\int_{t}^{1} S^{+}(t-\tau)[F_{R(\theta_{m}\omega)}(v_{m}(\tau))-F_{R(\theta_{m}\omega)}(u_{m}(\tau))\\
\quad-DF_{R(\theta_{m}\omega)}(u_{m}(\tau))(v_{m}(\tau)-u_{m}(\tau))]d\tau\\
\quad-\int_{t}^{1} S^{+}(t-\tau)[G_{R(\theta_{m}\omega)}(v_{m}(\tau))
-G_{R(\theta_{m}\omega)}(u_{m}(\tau))\\
\quad-DG_{R(\theta_{m}\omega)}(u_{m}(\tau))(v_{m}(\tau)-u_{m}(\tau))]d\theta_{m}\omega(\tau),
 & 1\leq m<m_1 \\
	         0, & m\geq m_1,
	       \end{cases}
\end{equation*}
\begin{equation*}
 \mathcal{I}^+_2(t)=
	       \footnotesize  \begin{cases}
-\sum_{i=m_1+2}^{\infty}S^{+}(t+m-i)\int_{0}^{1} S^{+}(1-\tau)[F_{R(\theta_{i-1}\omega)}(v_{i-1}(\tau))-F_{R(\theta_{i-1}\omega)}(u_{i-1}(\tau))\\
\quad- DF_{R(\theta_{i-1}\omega)}(u_{i-1}(\tau))(v_{i-1}(\tau)-u_{i-1}(\tau))]d\tau\\
\quad-\sum_{i=m_1+2}^{\infty}
S^{+}
(t+m-i)\int_{0}^{1}S^{+}
(1-\tau)[G_{R(\theta_{i-1}\omega)}(v_{i-1}(\tau))-G_{R(\theta_{i-1}\omega)}(u_{i-1}(\tau))\\
\quad- DG_{R(\theta_{i-1}\omega)}(u_{i-1}(\tau))(v_{i-1}(\tau)-u_{i-1}(\tau))]d\theta_{i-1}\omega(\tau)\\
\quad-\int_{0}^{1} S^{+}(t+m-m_1-1)[F_{R(\theta_{m_1}\omega)}(v_{m_1}(\tau))-F_{R(\theta_{m_1}\omega)}(u_{m_1}(\tau))\\
\quad-DF_{R(\theta_{m_1}\omega)}(u_{m_1}(\tau))(v_{m_1}(\tau)-u_{m_1}(\tau))]d\tau\\
\quad-\int_{0}^{1} S^{+}(t+m-m_1-1)[G_{R(\theta_{m_1}\omega)}(v_{m_1}(\tau))-G_{R(\theta_{m_1}\omega)}(u_{m_1}(\tau))\\
\quad-DG_{R(\theta_{m_1}\omega)}(u_{m_1}(\tau))(v_{m_1}(\tau)-u_{m_1}(\tau))]d\theta_{m_1}\omega(\tau),
& 1\leq m<m_1, \\
-\sum_{i=m+2}^{\infty}S^{+}(t+m-i)\int_{0}^{1} S^{+}(1-\tau)[F_{R(\theta_{i-1}\omega)}(v_{i-1}(\tau))-F_{R(\theta_{i-1}\omega)}(u_{i-1}(\tau))\\
\quad- DF_{R(\theta_{i-1}\omega)}(u_{i-1}(\tau))(v_{i-1}(\tau)-u_{i-1}(\tau))]d\tau\\
\quad-\sum_{i=m+2}^{\infty}
S^{+}
(t+m-i)\int_{0}^{1}S^{+}
(1-\tau)[G_{R(\theta_{i-1}\omega)}(v_{i-1}(\tau))-G_{R(\theta_{i-1}\omega)}(u_{i-1}(\tau))\\
\quad- DG_{R(\theta_{i-1}\omega)}(u_{m}(\tau))(v_{i-1}(\tau)-u_{i-1}(\tau))]d\theta_{i-1}\omega(\tau)\\
\quad-\int_{t}^{1} S^{+}(t-\tau)[F_{R(\theta_{m}\omega)}(v_{m}(\tau))-F_{R(\theta_{m}\omega)}(u_{m}(\tau))\\
\quad-DF_{R(\theta_{m}\omega)}(u_{m}(\tau))(v_{m}(\tau)-u_{m}(\tau))]d\tau\\
\quad-\int_{t}^{1} S^{+}(t-\tau)[G_{R(\theta_{m}\omega)}(v_{m}(\tau))-G_{R(\theta_{m}\omega)}(u_{m}(\tau))\\
\quad-DG_{R(\theta_{m}\omega)}(u_{m}(\tau))(v_{m}(\tau)-u_{m}(\tau))]d\theta_{m}\omega(\tau),
 & m\geq m_1.
	       \end{cases}
\end{equation*}
Taking $\gamma\leq \gamma_0$, we estimate all of the above terms. For $m\geq  m_1$, by assumption $(\mathbf{A}_2^k),(\mathbf{A}_3^k)$ and the definitions of  $L_F(R)$ and $L_G(R)$, we have the following inequalities:
\begin{subequations}
\begin{align}
\label{eq-a}&|F_{R(\theta_{i-1}\omega)}(v_{i-1})-F_{R(\theta_{i-1}\omega)}(u_{i-1})
- DF_{R(\theta_{i-1}\omega)}(u_{i-1})(v_{i-1}-u_{i-1})|\\ 
&\nonumber\quad\leq L_F(R(\theta_{i-1}\omega))|v_{i-1}-u_{i-1}|,\\
\label{eq-b}&|F_{R(\theta_{i-1}\omega)}(v_{i-1})-F_{R(\theta_{i-1}\omega)}(u_{i-1})
- DF_{R(\theta_{i-1}\omega)}(u_{i-1})(v_{i-1}-u_{i-1})|_{-\beta}\\  
&\nonumber\quad\leq L_F(R(\theta_{i-1}\omega))|v_{i-1}-u_{i-1}|,\\
 \label{eq-c}&\|G_{R(\theta_{i-1}\omega)}(v_{i-1})-G_{R(\theta_{i-1}\omega)}(u_{i-1})
- DG_{R(\theta_{i-1}\omega)}(u_{i-1})(v_{i-1}-u_{i-1})\|_{L_2(\mathcal{B};\mathcal{B})}\\ 
&\nonumber\quad\leq L_G(R(\theta_{i-1}\omega))|v_{i-1}-u_{i-1}|,\\
\label{eq-d}&\|G_{R(\theta_{i-1}\omega)}(v_{i-1})-G_{R(\theta_{i-1}\omega)}(u_{i-1})
- DG_{R(\theta_{i-1}\omega)}(u_{i-1})(v_{i-1}-u_{i-1})\|_{L_2(\mathcal{B};\mathcal{B}_{-\beta})}\\  
&\nonumber\quad\leq L_G(R(\theta_{i-1}\omega))|v_{i-1}-u_{i-1}|,\\
\label{eq-e}&\|G_{R(\theta_{i-1}\omega)}(v_{i-1}(r))-G_{R(\theta_{i-1}\omega)}(u_{i-1}(r))
- DG_{R(\theta_{i-1}\omega)}(u_{i-1}(r))[v_{i-1}(r)-u_{i-1}(r)]\\   
&\nonumber\quad-G_{R(\theta_{i-1}\omega)}(v_{i-1}(q))+G_{R(\theta_{i-1}\omega)}(u_{i-1}(q))
\!+\! DG_{R(\theta_{i-1}\omega)}(u_{i-1}(q))[v_{i-1}(q)-u_{i-1}(q)]\|_{L_2(\mathcal{B};\mathcal{B}_{-\beta})}\\
&\nonumber\quad\leq L_G(R)(r-q)^{\beta}\|v_{i-1}-u_{i-1}\|_{\beta,-\beta}.
\end{align}
\end{subequations}
Replacing $F(u),G(u)$ as in Lemma \ref{estim of G} and Theorem \ref{th-exits-solu} by $F_{R(\theta_{i-1}\omega)}(v_{i-1}(\tau))-F_{R(\theta_{i-1}\omega)}(u_{i-1}(\tau))
- DF_{R(\theta_{i-1}\omega)}(u_{i-1}(\tau))(v_{i-1}(\tau)-u_{i-1}(\tau))$ and $G_{R(\theta_{i-1}\omega)}(v_{i-1}(\tau))-G_{R(\theta_{i-1}\omega)}(u_{i-1}(\tau))
- DG_{R(\theta_{i-1}\omega)}(u_{i-1}(\tau))(v_{i-1}(\tau)-u_{i-1}(\tau))$ respectively, we obtain the following inequalities from \eqref{eq-a}, \eqref{eq-b}, \eqref{eq-c}, \eqref{eq-d} and \eqref{eq-e}:
\begin{align*}
&\bigg|\int_{0}^{1} S(1-\tau)[F_{R(\theta_{i-1}\omega)}(v_{i-1}(\tau))
-F_{R(\theta_{i-1}\omega)}(u_{i-1}(\tau))\\
&\quad~~~~~~~~~~~~~~~~~~~~~~~~~~~~~~~~~~~~~~~- DF_{R(\theta_{i-1}\omega)}(u_{i-1}(\tau))(v_{i-1}(\tau)-u_{i-1}(\tau))]d\tau\bigg|\\
&\quad\leq L_F(R(\theta_{i-1}\omega))\|v_{i-1}-u_{i-1}\|_{\beta,-\beta},
\end{align*}
\begin{align*}
&\bigg\|\int_{0}^{\cdot} S(\cdot-\tau)[F_{R(\theta_{i-1}\omega)}(v_{i-1}(\tau))-F_{R(\theta_{i-1}\omega)}(u_{i-1}(\tau))\\
&\quad~~~~~~~~~~~~~~~~~~~~~~~~~~~~~~~~~~~~~~~- DF_{R(\theta_{i-1}\omega)}(u_{i-1}(\tau))(v_{i-1}(\tau)-u_{i-1}(\tau))]d\tau\bigg\|_{\beta,-\beta}\\
&\quad\leq L_F(R(\theta_{i-1}\omega))\|v_{i-1}-u_{i-1}\|_{\beta,-\beta},
\end{align*}
\begin{align*}
&\bigg\|\int_{0}^{1} S(1-\tau)[G_{R(\theta_{i-1}\omega)}(v_{i-1}(\tau))-G_{R(\theta_{i-1}\omega)}(u_{i-1}(\tau))\\
&\quad~~~~~~~~~~~~~~~~~~~~~~~~~~~~~~~~~~~~~~~- DG_{R(\theta_{i-1}\omega)}(u_{i-1}(\tau))(v_{i-1}(\tau)-u_{i-1}(\tau))]d\tau\bigg\|_{L_2(\mathcal{B};\mathcal{B})}\\
&\quad\leq L_G(R(\theta_{i-1}\omega))\|v_{i-1}-u_{i-1}\|_{\beta,-\beta},
\end{align*}
and 
\begin{align*}
&\bigg\|\int_{0}^{\cdot} S(\cdot-\tau)[G_{R(\theta_{i-1}\omega)}(v_{i-1}(\tau))-G_{R(\theta_{i-1}\omega)}(u_{i-1}(\tau))\\
&\quad~~~~~~~~~~~~~~~~~~~~~~~~~~~~~~~~~~~~~~~- DG_{R(\theta_{i-1}\omega)}(u_{i-1}(\tau))(v_{i-1}(\tau)-u_{i-1}(\tau))]d\tau\bigg\|_{\beta,-\beta}\\
&\quad\leq L_G(R(\theta_{i-1}\omega))\|v_{i-1}-u_{i-1}\|_{\beta,-\beta}.
\end{align*}
Hence, for any $\epsilon>0$ and  large enough $m_1$, $\mathcal{I}_2^-$ and $\mathcal{I}_{2}^+$ can be estimated by 
\begin{align*}
\sup_{m\geq m_1}e^{m(\kappa+\gamma)}\|\mathcal{I}_2^-\|_{\beta,-\beta}&\leq K\|U-V\|_{\mathcal{H}_{\kappa+2\gamma_0}}\sup_{m\geq m_1}\sum_{i=0}^{m-1} e^{\check{\mu}(m-i-1)+(\kappa+\gamma)m-(\kappa+2\gamma_0)i}\\
&\leq 2c_SK e^{(\gamma-2\gamma_0)m_1}\sup_{m\geq m_1}\sum_{i=0}^{m-1} e^{\check{\mu}(m-i-1)+(\kappa+2\gamma_0)(m-i)}|\xi-\xi_0|\\
&\leq \frac{\epsilon}{4}|\xi-\xi_0|,
\end{align*}
and 
\begin{align*}
\sup_{m\geq m_1}e^{m(\kappa+\gamma)}\|\mathcal{I}_2^+\|_{\beta,-\beta}&\leq K\|U-V\|_{\mathcal{H}_{\kappa+2\gamma_0}}\sup_{m\geq m_1}\sum_{i=m}^\infty e^{\hat{\mu}(m-i)+(\kappa+\gamma)m-(\kappa+2\gamma_0)i}\\
&\leq 2c_SKe^{(\gamma-2\gamma_0)m_1}\sup_{m\geq m_1}\sum_{i=m}^\infty e^{\hat{\mu}(m-i)+(\kappa+2\gamma_0)(m-i)}|\xi-\xi_0|\\
&\leq \frac{\epsilon}{4}|\xi-\xi_0|.
\end{align*}
For the case of $1\leq m\leq m_1$, one can directly apply the result of the case of $m>m_1$ for $\mathcal{I}_2^+$, that is for large enough $m_1$ it holds that
\begin{align*}
\sup_{1\leq m< m_1}e^{m(\kappa+\gamma)}\|\mathcal{I}_2^+\|_{\beta,-\beta}
&\leq2c_SKe^{(\gamma-2\gamma_0)m_1}|\xi-\xi_0|\sup_{m\leq m_1}\sum_{i=m_1}^{\infty}e^{(\check{\mu}+\kappa+\gamma)(m-i)}\\
&\leq \frac{\epsilon}{4}|\xi-\xi_0|.
\end{align*}
However, the method applied above cannot be used for $\mathcal{I}_1^-$ {and $\mathcal{I}_1^+ $} 
when $1\leq m\leq m_1$. Fortunately, we can use the continuity of $F,G$ and $\Gamma$ to obtain a sufficiently small quantity. More precisely, for any $\epsilon>0$, we denote by 
\begin{align*}
\mathcal{E}_{1}&:=\max\bigg\{\frac{\epsilon}{{4}c^2_S\sup_{m\leq m_1}\sum_{i=m}^{m_1-1}e^{\hat{\mu}(m-i)+(\kappa+\gamma)m-(\kappa+2\gamma_0)i}(1+C_{\alpha,\beta}\vertiii{\theta_{i}\omega}_{\beta',0,1})},\\
&
\frac{\epsilon}{{4}c^2_S\sup_{m\leq m_1}{\sum_{i=1}^{m}e^{\check{\mu}(m-i)+(\kappa+\gamma)m-(\kappa+2\gamma_0)i}
 (1+C_{\alpha,\beta}\vertiii{\theta_{i}\omega}_{\beta',0,1})}}
\bigg\}.\end{align*}
By the continuity of $DF,DG$ and  $\Gamma(\omega,\xi)$ with respect to $\xi$, for given $\epsilon>0$, there exists a constant $\epsilon_1>0$ such that when $|\xi-\xi_{0}|\leq \epsilon_1$, the following inequalities are valid:
\begin{equation*}
\begin{aligned}\sup_{i\leq m\leq m_1,\tau\in [0,1]}\max\bigg\{
&\|DF_{R(\theta_{i-1}\omega)}(\tau v_{i-1}+(1-\tau)u_{i-1})-DF_{R(\theta_{i-1}\omega)}(u_{i-1})\|_{L(\mathcal{B};\mathcal{B})}
\\
&\|DF_{R(\theta_{i-1}\omega)}(\tau v_{i-1}+(1-\tau)u_{i-1})-DF_{R(\theta_{i-1}\omega)}(u_{i-1})\|_{L(\mathcal{B};\mathcal{B}_{-\beta})}
\\
&\left\|DG_{R(\theta_{i-1}\omega)}(\tau v_{i-1}+(1-\tau )u_{i-1})-DG_{R(\theta_{i-1}\omega)}(u_{i-1})\right\|_{L(\mathcal{B};L_2(\mathcal{B};\mathcal{B}_{-\beta}))}
\\
&\left\|DG_{R(\theta_{i-1}\omega)}(\tau v_{i-1}+(1-\tau)u_{i-1})-DG_{R(\theta_{i-1}\omega)}(u_{i-1})\right\|_{L(\mathcal{B};L_2(\mathcal{B};\mathcal{B}))}
\bigg\}
\leq \mathcal{E}_{1}.
\end{aligned}
\end{equation*}
Hence we further have 
\begin{equation*}\begin{aligned}
\max\bigg\{
&|F_{R(\theta_{i-1}\omega)}(v_{i-1})-F_{R(\theta_{i-1}\omega)}(u_{i-1})
- DF_{R(\theta_{i-1}\omega)}(u_{i-1})(v_{i-1}-u_{i-1})|,\\
&|F_{R(\theta_{i-1}\omega)}(v_{i-1})-F_{R(\theta_{i-1}\omega)}(u_{i-1})
- DF_{R(\theta_{i-1}\omega)}(u_{i-1})(v_{i-1}-u_{i-1})|_{-\beta},\\
&\|G_{R(\theta_{i-1}\omega)}(v_{i-1})-G_{R(\theta_{i-1}\omega)}(u_{i-1})
- DG_{R(\theta_{i-1}\omega)}(u_{i-1})(v_{i-1}-u_{i-1})\|_{L_2(\mathcal{B};\mathcal{B}_{-\beta})},\\
&\|G_{R(\theta_{i-1}\omega)}(v_{i-1})-G_{R(\theta_{i-1}\omega)}(u_{i-1})
- DG_{R(\theta_{i-1}\omega)}(u_{i-1})(v_{i-1}-u_{i-1})\|_{L_2(\mathcal{B};\mathcal{B})}
\bigg\}\\
&\leq \mathcal{E}_{1}|v_{i-1}-u_{i-1}|,
\end{aligned}
\end{equation*}
and additionally,
\begin{align*}
&\|G_{R(\theta_{i-1}\omega)}(v_{i-1}(r))-G_{R(\theta_{i-1}\omega)}(u_{i-1}(r))
- DG_{R(\theta_{i-1}\omega)}(u_{i-1}(r))[v_{i-1}(r)-u_{i-1}(r)]\\
&\quad-G_{R(\theta_{i-1}\omega)}(v_{i-1}(q))+G_{R(\theta_{i-1}\omega)}(u_{i-1}(q))
+ DG_{R(\theta_{i-1}\omega)}(u_{i-1}(q))[v_{i-1}(q)-u_{i-1}(q)]\|_{L_2(\mathcal{B};\mathcal{B}_{-\beta})}\\
&\leq \mathcal{E}_{1}(r-q)^{\beta}\|v_{i-1}-u_{i-1}\|_{\beta,-\beta}.
\end{align*}
Hence we further obtain 
\begin{align*}
&\bigg|\int_{0}^{1} S(1-\tau)[F_{R(\theta_{i-1}\omega)}(v_{i-1}(\tau))-F_{R(\theta_{i-1}\omega)}(u_{i-1}(\tau))\\
&~~~~~~~~~~~~~~~~~~~~~~~~~~~~~~~~~~~~~~~~~~- DF_{R(\theta_{i-1}\omega)}(u_{i-1}(\tau))(v_{i-1}(\tau)-u_{i-1}(\tau))]d\tau\bigg|\\
&\quad\leq \mathcal{E}_{1}c_S\|v_{i-1}-u_{i-1}\|_{\beta,-\beta},\\
&\bigg\|\int_{0}^{\cdot} S(\cdot-\tau)[F_{R(\theta_{i-1}\omega)}(v_{i-1}(\tau))-F_{R(\theta_{i-1}\omega)}(u_{i-1}(\tau))
\\
&~~~~~~~~~~~~~~~~~~~~~~~~~~~~~~~~~~~~~~~~~~- DF_{R(\theta_{i-1}\omega)}(u_{i-1}(\tau))(v_{i-1}(\tau)-u_{i-1}(\tau))]d\tau\bigg\|_{\beta,-\beta}\\
&\quad\leq \mathcal{E}_{1}c_S\|v_{i-1}-u_{i-1}\|_{\beta,-\beta},\\
&\bigg\|\int_{0}^{1} S(1-\tau)[G_{R(\theta_{i-1}\omega)}(v_{i-1}(\tau))-G_{R(\theta_{i-1}\omega)}(u_{i-1}(\tau))
\\
&~~~~~~~~~~~~~~~~~~~~~~~~~~~~~~~~~~~~~~~~~~- DG_{R(\theta_{i-1}\omega)}(u_{i-1}(\tau))(v_{i-1}(\tau)-u_{i-1}(\tau))]d\theta_{i-1}\omega(\tau)\bigg\|_{L_2(\mathcal{B};\mathcal{B})}\\
&\quad\leq \mathcal{E}_{1}c_SC_{\alpha,\beta}\vertiii{\theta_{i-1}\omega}_{\beta,0,1}\|v_{i-1}-u_{i-1}\|_{\beta,-\beta},\\
&\bigg\|\int_{0}^{\cdot} S(\cdot-\tau)[G_{R(\theta_{i-1}\omega)}(v_{i-1}(\tau))-G_{R(\theta_{i-1}\omega)}(u_{i-1}(\tau))
\\
&~~~~~~~~~~~~~~~~~~~~~~~~~~~~~~~~~~~~~~~~~~-DG_{R(\theta_{i-1}\omega)}(u_{i-1}(\tau))(v_{i-1}(\tau)-u_{i-1}(\tau))]d\theta_{i-1}\omega(\tau)\bigg\|_{\beta,-\beta}\\
&\quad\leq \mathcal{E}_{1}c_SC_{\alpha,\beta}\vertiii{\theta_{i-1}\omega}_{\beta,0,1}\|v_{i-1}-u_{i-1}\|_{\beta,-\beta}.
\end{align*}
Finally, {$\mathcal{I}_1^-$ 
and $\mathcal{I}_1^+$} can be estimated by
\begin{subequations}\begin{align}
&\sup_{1\leq m< m_1}e^{m(\kappa+\gamma)}\|\mathcal{I}_1^-\|_{\beta,-\beta}\\
&\nonumber\quad\leq 
{\mathcal{E}_{1}c_S\|V-U\|_{\mathcal{H}_{\kappa+2\gamma_0}}\sup_{m\leq m_1}\sum_{i=0}^{m}e^{\check{\mu}(m-i)+(\kappa+\gamma)m-(\kappa+2\gamma_0)i}
 (1+C_{\alpha,\beta}\vertiii{\theta_{i}\omega}_{\beta',0,1})}\\
&\sup_{1\leq m< m_1}e^{m(\kappa+\gamma)}\|\mathcal{I}_1^+\|_{\beta,-\beta}\\
&\nonumber\quad\leq \mathcal{E}_{1}c_S\|V-U\|_{\mathcal{H}_{\kappa+2\gamma_0}}\sup_{m\leq m_1}\sum_{i=m}^{m_1-1}e^{\hat{\mu}(m-i)+(\kappa+\gamma)m-(\kappa+2\gamma_0)i}(1+C_{\alpha,\beta}\vertiii{\theta_{i}\omega}_{\beta',0,1}).
\end{align}
\end{subequations}
All of the above quantities are bounded by ${\frac{\epsilon}{2}}|\xi-\xi_0|$. 
In conclusion, there exist $m_1>0$ large enough and $\epsilon_1>0$ such that for $|\xi-\xi_0|\leq\epsilon_1$, we have 
\begin{align*}
\|\mathcal{I}(V-U)\|_{\mathcal{H}_{\kappa+\gamma}}\leq \epsilon|\xi-\xi_0|,
\end{align*}
which proves the claim. 
\end{proof}

According to the definition of the solution to \eqref{eq-cut}, we have $D_{\xi}U=\{D_{\xi}u_i\}_{i\in\mathbb{Z}^+}$ with
\begin{equation}
\begin{aligned}
D_{\xi}u_m(t)
&=S^{-}(t)+\sum_{i=1}^{m}S^{-}(t+m-i)\int_{0}^{1} S^{-}(1-\tau)DF_{R(\theta_{i-1}\omega)}(u_{i-1}(\tau))D_{\xi}u_{i-1}(\tau)d\tau\\
&\quad+\sum_{i=1}^{m}S^{-}(t+m-i)\int_{0}^{1} S^{-}(1-\tau)DG_{R(\theta_{i-1}\omega)}(u_{i-1}(\tau))D_{\xi}u_{i-1}(\tau)d\theta_{i-1}\omega(\tau)\\
&\quad+\int_{0}^{t} S^{-}(t-\tau)DF_{R(\theta_m\omega)}(u_{m}(\tau))D_{\xi}u_{m}(\tau)d\tau\\
&\quad+\int_{0}^{t} S^{-}(t-\tau)DG_{R(\theta_m\omega)}(u_{m}(\tau))D_{\xi}u_{m}(\tau)d\theta_{m}\omega(\tau)\\
&\quad-\sum_{i=m+2}^{\infty}S^{+}(t+m-i)\int_{0}^{1} S^{+}(1-\tau)DF_{R(\theta_{i-1}\omega)}(u_{i-1}(\tau))D_{\xi}u_{i-1}(\tau)d\tau\\
&\quad-\sum_{i=m+2}^{\infty}
S^{+}
(t+m-i)\int_{0}^{1}S^{+}
(1-\tau)DG_{R(\theta_{i-1}\omega)}(u_{i-1}(\tau))D_{\xi}u_{i-1}(\tau)d\theta_{i-1}\omega(\tau)\\
&\quad-\int_{t}^{1} S^{+}(t-\tau)DF_{R(\theta_{m}\omega)}(u_m(\tau))D_{\xi}u_{m}(\tau)d\tau \\
&\quad-\int_{t}^{1} S^{+}(t-\tau)DG_{R(\theta_{m}\omega)}(u_m(\tau))D_{\xi}u_{m}(\tau)d\theta_{m}\omega(\tau).
\end{aligned}
\end{equation}
Combining this with the estimate on $\mathcal{L}$, one can directly obtain that 
\begin{align*}
\sup_{i\in\mathbb{Z}^+}e^{(\kappa+\gamma)i}\|D_{\xi}u_{i}\|_{\beta,-\beta}&\leq c_S+\frac{1}{2}\sup_{i\in\mathbb{Z}^+}e^{(\kappa+\gamma)i}\|D_{\xi}u_{i}\|_{\beta,-\beta}
\leq 2c_S,
\end{align*}
for any $\gamma\in[0,\gamma_0]$. Hence $D_\xi u$ is a bounded linear operator on $\cB^-$. Lastly we have to show the continuity with respect to $\xi$. By the definition of $\mathcal{L}$, we can express the difference of $D_\xi u$ and $D_\xi v$ as
\begin{align*}
D_{\xi}v_m(t)-D_{\xi}u_m(t)=[\mathcal{L}(D_{\xi}V-D_{\xi}U)]_{m}(t)+[\mathcal{T}]_{m}(t).
\end{align*}
By \eqref{eq-esti-L}, $D_{\xi}u_m(t)$ is continuous with respect to $\xi$ if  $\|\mathcal{T}\|_{\kappa+\gamma}\to 0$ as $\xi\to\xi_0$.
To this aim,  for an $m_1'>0$ to be determined later, we rewrite $\mathcal{T}$ as follows:
\begin{align*}
[\mathcal{T}]_{m}(t)=\mathcal{T}^-_1(t)+\mathcal{T}^-_2(t)+\mathcal{T}^+_1(t)+\mathcal{T}^+_2(t),
\end{align*}
in which 
\begin{equation*}
 {\mathcal{T}^-_1(t)=
	      \footnotesize   \begin{cases}
\sum_{i=1}^{m}S^{-}(t+m-i)\int_{0}^{1} S^{-}(1-\tau)[DF_{R(\theta_{i-1}\omega)}(v_{i-1}(\tau))\\
\quad-DF_{R(\theta_{i-1}\omega)}(u_{i-1}(\tau))]D_{\xi}v_{i-1}(\tau)d\tau\\
\quad+\sum_{i=1}^{m}S^{-}(t+m-i)\int_{0}^{1} S^{-}(1-\tau)[DG_{R(\theta_{i-1}\omega)}(v_{i-1}(\tau))\\
\quad-DG_{R(\theta_{i-1}\omega)}(u_{i-1}(\tau))]D_{\xi}v_{i-1}(\tau)d\theta_{i-1}\omega(\tau)\\
\quad+\int_{0}^{t} S^{-}(t-\tau)[DF_{R(\theta_{m}\omega)}(v_{m}(\tau))-DF_{R(\theta_{m}\omega)}(u_{m}(\tau))]D_{\xi}v_{m}(\tau)d\tau\\
\quad+\int_{0}^{t} S^{-}(t-\tau)[DG_{R(\theta_{m}\omega)}(v_{m}(\tau))-DG_{R(\theta_{m}\omega)}(u_{m}(\tau))]D_{\xi}v_{m}(\tau)d\theta_{m}\omega(\tau),
 & 1\leq m<m_1', \\
	         0, & m\geq m_1',
	       \end{cases}}
\end{equation*}
\begin{equation*}
 \mathcal{T}^-_2(t)=
	      \footnotesize   \begin{cases}
    {0,}
& 1\leq m<m_1', \\
\sum_{i=1}^{m}S^{-}(t+m-i)\int_{0}^{1} S^{-}(1-\tau)[DF_{R(\theta_{i-1}\omega)}(v_{i-1}(\tau))\\
\quad-DF_{R(\theta_{i-1}\omega)}(u_{i-1}(\tau))]D_{\xi}v_{i-1}(\tau)d\tau\\
\quad+\sum_{i=1}^{m}S^{-}(t+m-i)\int_{0}^{1} S^{-}(1-\tau)[DG_{R(\theta_{i-1}\omega)}(v_{i-1}(\tau))\\
\quad-DG_{R(\theta_{i-1}\omega)}(u_{i-1}(\tau))]D_{\xi}v_{i-1}(\tau)d\theta_{i-1}\omega(\tau)\\
\quad+\int_{0}^{t} S^{-}(t-\tau)[DF_{R(\theta_{m}\omega)}(v_{m}(\tau))-DF_{R(\theta_{m}\omega)}(u_{m}(\tau))]D_{\xi}v_{m}(\tau)d\tau\\
\quad+\int_{0}^{t} S^{-}(t-\tau)[DG_{R(\theta_{m}\omega)}(v_{m}(\tau))-DG_{R(\theta_{m}\omega)}(u_{m}(\tau))]D_{\xi}v_{m}(\tau)d\theta_{m}\omega(\tau),
 & m\geq m_1',
	       \end{cases}
\end{equation*}
\begin{equation*}
 \mathcal{T}^+_1(t)=
	      \footnotesize   \begin{cases}
-\sum_{i=m+2}^{m_1'}S^{+}(t+m-i)\int_{0}^{1} S^{+}(1-\tau)[DF_{R(\theta_{i-1}\omega)}(v_{i-1}(\tau))\\
\quad-DF_{R(\theta_{i-1}\omega)}(u_{i-1}(\tau))]D_{\xi}v_{i-1}(\tau)d\tau\\
\quad-\sum_{i=m+2}^{m_1'}
S^{+}
(t+m-i)\int_{0}^{1}S^{+}
(1-\tau)[DG_{R(\theta_{i-1}\omega)}(v_{i-1}(\tau))\\
\quad-DG_{R(\theta_{i-1}\omega)}(u_{i-1}(\tau))]D_{\xi}v_{i-1}(\tau)d\theta_{i-1}\omega(\tau)\\
\quad-\int_{t}^{1} S^{+}(t-\tau)[DF_{R(\theta_{m}\omega)}(v_{m}(\tau))-DF_{R(\theta_{m}\omega)}(u_{m}(\tau))]D_{\xi}v_{m}(\tau)d\tau\\
\quad -\int_{t}^{1} S^{+}(t-\tau)[DG_{R(\theta_{m}\omega)}(v_{m}(\tau))-DG_{R(\theta_{m}\omega)}(u_{m}(\tau))]D_{\xi}v_{m}(\tau)d\theta_{m}\omega(\tau),
 & 1\leq m<m_1', \\
	         0, & m\geq m_1',
	       \end{cases}
\end{equation*}
\begin{equation*}
 \mathcal{T}^+_2(t)=
	      \footnotesize   \begin{cases}
-\sum_{i=m_1'+2}^{\infty}S^{+}(t+m-i)\int_{0}^{1} S^{+}(1-\tau)[DF_{R(\theta_{i-1}\omega)}(v_{i-1}(\tau))\\
\quad-DF_{R(\theta_{i-1}\omega)}(u_{i-1}(\tau))]D_{\xi}v_{i-1}(\tau)d\tau\\
\quad-\sum_{i=m_1'+2}^{\infty}
S^{+}
(t+m-i)\int_{0}^{1}S^{+}
(1-\tau)[DG_{R(\theta_{i-1}\omega)}(v_{i-1}(\tau))\\
\quad-DG_{R(\theta_{i-1}\omega)}(u_{i-1}(\tau))]D_{\xi}v_{i-1}(\tau)d\theta_{i-1}\omega(\tau)\\
\quad-\int_{0}^{1} S^{+}(t+m-m_1'-\tau)[DF_{R(\theta_{m_1'}\omega)}(u_{m_1'}(\tau))-DF_{R(\theta_{m_1'}\omega)}(v_{m_1'}(\tau))]D_{\xi}v_{m_1'}(\tau)d\tau\\
\quad -\int_{0}^{1} S^{+}(t+m-m_1'-\tau)[DG_{R(\theta_{m_1'}\omega)}(v_{m_1'}(\tau))\\
\quad-DG_{R(\theta_{m_1'}\omega)}(u_{m_1'}(\tau))]D_{\xi}v_{m_1'}d\theta_{m_1'}\omega(\tau),
& 1\leq m<m_1' ,\\
-\sum_{i=m+2}^{\infty}S^{+}(t+m-i)\int_{0}^{1} S^{+}(1-\tau)[DF_{R(\theta_{i-1}\omega)}(v_{i-1}(\tau))\\
\quad-DF_{R(\theta_{i-1}\omega)}(u_{i-1}(\tau))]D_{\xi}v_{i-1}(\tau)d\tau\\
\quad-\sum_{i=m+2}^{\infty}
S^{+}
(t+m-i)\int_{0}^{1}S^{+}
(1-\tau)[DG_{R(\theta_{i-1}\omega)}(v_{i-1}(\tau))\\
\quad-DG_{R(\theta_{i-1}\omega)}(u_{i-1}(\tau))]D_{\xi}v_{i-1}(\tau)d\theta_{i-1}\omega(\tau)\\
\quad-\int_{t}^{1} S^{+}(t-\tau)[DF_{R(\theta_{m}\omega)}(v_{m}(\tau))-DF_{R(\theta_{m}\omega)}(u_{m}(\tau))]D_{\xi}v_{m}(\tau)d\tau\\
\quad -\int_{t}^{1} S^{+}(t-\tau)[DG_{R(\theta_{m}\omega)}(v_{m}(\tau))-DG_{R(\theta_{m}\omega)}(u_{m}(\tau))]D_{\xi}v_{m}(\tau)d\theta_{m}\omega(\tau),
 & m\geq m_1'.
	       \end{cases}
\end{equation*}
\begin{lemma}
Assume that $(\mathbf{A_2^k})$, $(\mathbf{A_2^k})$ and \eqref{gap k} are satisfied. Then for $\gamma\in[0,\gamma_0)$ it holds
\begin{align*}
\lim_{\xi\to\xi_0}\|\mathcal{T}\|_{\mathcal{H}_{\kappa+\gamma}}=0.
\end{align*}
\end{lemma}
\begin{proof}
We now take $\gamma< \gamma_0$, for $m\geq  m_1'$. Due to $(\mathbf{A_2^k}),(\mathbf{A_3^k})$ and by the definition of $L_F(R),L_G(R)$, we have 
\begin{subequations}
\begin{align}
&|[DF_{R(\theta_{i-1}\omega)}(v_{i-1})-DF_{R(\theta_{i-1}\omega)}]D_{\xi}u_{i-1}|
\leq L_F(R(\theta_{i-1}\omega))|D_{\xi}u_{i-1}|,\\
&|[DF_{R(\theta_{i-1}\omega)}(v_{i-1})-DF_{R(\theta_{i-1}\omega)}(u_{i-1})]D_{\xi}u_{i-1}|_{-\beta}
\leq L_F(R(\theta_{i-1}\omega))|D_{\xi}u_{i-1}|,\\
&\|[DG_{R(\theta_{i-1}\omega)}(v_{i-1})-DG_{R(\theta_{i-1}\omega)}(u_{i-1})]D_{\xi}u_{i-1}\|_{L_2(\mathcal{B};\mathcal{B}_{-\beta})}
\leq L_G(R(\theta_{i-1}\omega))|D_{\xi}u_{i-1}|,\\
&\|[DG_{R(\theta_{i-1}\omega)}(v_{i-1})-DG_{R(\theta_{i-1}\omega)}(u_{i-1})]D_{\xi}u_{i-1}\|_{L_2(\mathcal{B};\mathcal{B})}
\leq L_G(R(\theta_{i-1}\omega))|D_{\xi}u_{i-1}|,\\
&\|[DG_{R(\theta_{i-1}\omega)}(v_{i-1}(r))-DG_{R(\theta_{i-1}\omega)}(u_{i-1}(r))]D_{\xi}u_{i-1}(r)\\
\nonumber&\quad-[DG_{R(\theta_{i-1}\omega)}(v_{i-1}(q))-DG_{R(\theta_{i-1}\omega)}(u_{i-1}(q))]D_{\xi}u_{i-1}(q)\|_{L_2(\mathcal{B};\mathcal{B}_{-\beta})}\\
\nonumber&\leq L_G(R(\theta_{i-1}\omega))(r-q)^{\beta}\|D_{\xi}u_{i-1}\|_{\beta,-\beta}.
\end{align}
\end{subequations}
Hence we further obtain that 
\begin{subequations}
\begin{align}
\label{F-a}&\left|\int_{0}^{1} S(1-\tau)[DF_{R(\theta_{i-1}\omega)}(v_{i-1}(\tau))-DF_{R(\theta_{i-1}\omega)}(u_{i-1}(\tau))]D_{\xi}u_{i-1}(\tau)d\tau\right|\\
\nonumber&\quad\leq L_F(R(\theta_{i-1}\omega))\|D_{\xi}u_{i-1}(\tau)\|_{\beta,-\beta},\\
\label{F-b}&\left\|\int_{0}^{\cdot} S(\cdot-\tau)[DF_{R(\theta_{i-1}\omega)}(v_{i-1}(\tau))-DF_{R(\theta_{i-1}\omega)}(u_{i-1}(\tau))]D_{\xi}u_{i-1}(\tau)d\tau\right\|_{\beta,-\beta}\\
\nonumber&\quad\leq L_F(R(\theta_{i-1}\omega))\|D_{\xi}u_{i-1}\|_{\beta,-\beta},\\
\label{G-c}&\left\|\int_{0}^{1} S(1-\tau)[DG_{R(\theta_{i-1}\omega)}(v_{i-1}(\tau))-DG_{R(\theta_{i-1}\omega)}(u_{i-1}(\tau))]D_{\xi}u_{i-1}(\tau)d\theta_{i-1}\omega(\tau)\right\|_{L_2(\mathcal{B};\mathcal{B})}\\
\nonumber&\quad\leq C_{\alpha,\beta}L_G(R(\theta_{i-1}\omega))\vertiii{\theta_{i-1}\omega}_{\beta',0,1}\|D_{\xi}u_{i-1}\|_{\beta,-\beta},\\
\label{G-d}&\left\|\int_{0}^{\cdot} S(\cdot-\tau)[DG_{R(\theta_{i-1}\omega)}(v_{i-1}(\tau))-DG_{R(\theta_{i-1}\omega)}(u_{i-1}(\tau))]D_{\xi}u_{i-1}(\tau)d\theta_{i-1}\omega(\tau)\right\|_{\beta,-\beta}\\
\nonumber&\quad\leq C_{\alpha,\beta}L_G(R(\theta_{i-1}\omega))\vertiii{\theta_{i-1}\omega}_{\beta',0,1}\|D_{\xi}u_{i-1}\|_{\beta,-\beta}.
\end{align}
\end{subequations}
By similar computation as for $\mathcal{J}_{R}$ and using \eqref{F-a},\eqref{F-b},\eqref{G-c} and \eqref{G-d}, we derive
\begin{align*}
\|\mathcal{T}_2^-\|_{\mathcal{H}_{\kappa+\gamma}}&\leq K\|D_{\xi}U\|_{\mathcal{H}_{\kappa+\gamma_0}}\sup_{m\geq m_1'}\sum_{i=0}^m e^{\check{\mu}(m-i)+(\kappa+\gamma)m-(\kappa+\gamma_0)i},
\end{align*}
\begin{align*}
\|\mathcal{T}_2^+\|_{\mathcal{H}_{\kappa+\gamma}}&\leq K\|D_{\xi}U\|_{\mathcal{H}_{\kappa+\gamma_0}}\sup_{m\geq m_1'}\sum_{i=m}^\infty e^{\hat{\mu}(m-i)+(\kappa+\gamma)m-(\kappa+\gamma_0)i},
\end{align*}
and 
\begin{align*}
\|\mathcal{T}_2^+\|_{\mathcal{H}_{\kappa+\gamma}}&\leq K\sup_{m\leq m_1'}\sum_{i=m_1'}^{\infty}e^{{\color{blue}\hat{\mu}(m-i)}+(\kappa+\gamma)m-(\kappa+\gamma_0)i}\|D_{\xi}U\|_{\mathcal{H}_{\kappa+\gamma_0}}.
\end{align*}
Then for any $\epsilon>0$, there exists a large enough $m_1'$ such that both of them are bounded by $\frac{\epsilon}{4}\|D_{\xi}U\|_{\mathcal{H}_{\kappa+\gamma_0}}$.

For other terms when $1\leq m\leq m_1'$, we use the continuity of $DF,DG$ and $\Gamma(\omega,\xi)$ with respect to $\xi$. For any $\epsilon>0$ we denote by 
\begin{align*}
\mathcal{E}_{1}'&:=\max\bigg\{\frac{\epsilon}{4c_S\sup_{m\leq m_1'}{\color{blue}\sum_{i=0}^{m}e^{\check{\mu}(m-i)+(\kappa+\gamma)m-(\kappa+\gamma_0)i}}(1+C_{\alpha,\beta}\vertiii{\theta_{i}\omega}_{\beta',0,1})},\\
&
\frac{\epsilon}{4c_S\sup_{m\leq m_1'}\sum_{i=m}^{m_1'-1}e^{\hat{\mu}(m-i)+(\kappa+\gamma)m-(\kappa+\gamma_0)i}(1+C_{\alpha,\beta}\vertiii{\theta_{i}\omega}_{\beta',0,1})}
\bigg\}.\end{align*}
Due to the continuity of the derivatives of $F$ and $G$, there exists a constant $\epsilon_1'>0$ such that when $|\xi-\xi_{0}|\leq \epsilon_1'$, 
it holds 
\begin{align*}
\sup_{i\leq m\leq m_1'}\max\bigg\{
&\|DF_{R(\theta_{i-1}\omega)}( v_{i-1})-DF_{R(\theta_{i-1}\omega)}(u_{i-1})\|_{L(\mathcal{B};\mathcal{B})},\\
&\|DF_{R(\theta_{i-1}\omega)}( v_{i-1})-DF_{R(\theta_{i-1}\omega)}(u_{i-1})\|_{L(\mathcal{B};\mathcal{B}_{-\beta})}\\
&\left\|DG_{R(\theta_{i-1}\omega)}(v_{i-1})-DG_{R(\theta_{i-1}\omega)}(u_{i-1})\right\|_{L(\mathcal{B};L_2(\mathcal{B};\mathcal{B}_{-\beta}))}\\
&\left\|DG_{R(\theta_{i-1}\omega)}(v_{i-1})-DG_{R(\theta_{i-1}\omega)}(u_{i-1})\right\|_{L(\mathcal{B};L_2(\mathcal{B};\mathcal{B}))}
\bigg\}
\leq \mathcal{E}_{1}',
\end{align*}
and
\begin{align*}
&\|[DG_{R(\theta_{i-1}\omega)}(v_{i-1}(r))-DG_{R(\theta_{i-1}\omega)}(u_{i-1}(r))]D_{\xi}u_{i-1}(r)\\
&\quad-[DG_{R(\theta_{i-1}\omega)}(v_{i-1}(q))-DG_{R(\theta_{i-1}\omega)}(u_{i-1}(q))]D_{\xi}u_{i-1}(q)\|_{L_2(\mathcal{B};\mathcal{B}_{-\beta})}\\
&\leq \mathcal{E}_{1}'(r-q)^{\beta}\|D_{\xi}u_{i-1}\|_{\beta,-\beta}.
\end{align*}
Hence we  obtain 
\begin{align*}
&\left|\int_{0}^{1} S(1-\tau)[DF_{R(\theta_{i-1}\omega)}(v_{i-1}(\tau))-DF_{R(\theta_{i-1}\omega)}(u_{i-1}(\tau))]D_{\xi}u_{i-1}(\tau)d\tau\right|,
\end{align*}
and 
\begin{align*}
&\left\|\int_{0}^{\cdot} S(\cdot-\tau)[DF_{R(\theta_{i-1}\omega)}(v_{i-1}(\tau))-DF_{R(\theta_{i-1}\omega)}(u_{i-1}(\tau))]D_{\xi}u_{i-1}(\tau)d\tau\right\|_{\beta,-\beta}
\end{align*}
are bounded by $\mathcal{E}_{1}'\|D_{\xi}u_{i-1}\|_{\beta,-\beta}$, and 
\begin{align*}
&\left\|\int_{0}^{1} S(1-\tau)[DG_{R(\theta_{i-1}\omega)}(v_{i-1}(\tau))-DG_{R(\theta_{i-1}\omega)}(u_{i-1}(\tau))]D_{\xi}u_{i-1}(\tau)d\theta_{i-1}\omega(\tau)\right\|_{L_2(\mathcal{B};\mathcal{B})},
\end{align*}
and 
\begin{align*}
&\left\|\int_{0}^{\cdot} S(\cdot-\tau)[DG_{R(\theta_{i-1}\omega)}(v_{i-1}(\tau))-DG_{R(\theta_{i-1}\omega)}(u_{i-1}(\tau))]D_{\xi}u_{i-1}(\tau)d\theta_{i-1}\omega(\tau)\right\|_{\beta,-\beta}
\end{align*}
are bounded by $\mathcal{E}_{1}'\vertiii{\theta_{i-1}\omega}_{\beta',0,1}\|D_{\xi}u_{i-1}\|_{\beta,-\beta}$. Hence 
\begin{align*}
	\|\mathcal{T}_1^-\|_{\mathcal{H}_{\kappa+\gamma}}&\leq \mathcal{E}_{1}'c_S\|D_{\xi}U\|_{\mathcal{H}_{\kappa+\gamma_0}}\sup_{m\leq m_1'}{\sum_{i=0}^{m}e^{\check{\mu}(m-i)+(\kappa+\gamma)m-(\kappa+\gamma_0)i}}(1+C_{\alpha,\beta}\vertiii{\theta_{i}\omega}_{\beta',0,1}),\\
\|\mathcal{T}_1^+\|_{\mathcal{H}_{\kappa+\gamma}}&\leq \mathcal{E}_{1}'c_S\|D_{\xi}U\|_{\mathcal{H}_{\kappa+\gamma_0}}\sup_{m\leq m_1'}\sum_{i=m}^{m_1'-1}e^{\hat{\mu}(m-i)+(\kappa+\gamma)m-(\kappa+\gamma_0)(i-1)}(1+C_{\alpha,\beta}\vertiii{\theta_{i}\omega}_{\beta',0,1})
\end{align*}
and all of them are bounded by ${\frac{\epsilon}{2}}\|D_{\xi}U\|_{\mathcal{H}_{\kappa+\gamma_0}}$.
In conclusion, there exists a constant $m_1'>0$ large enough and $\epsilon_1'$ (as above) such that for $|\xi-\xi_0|\leq\epsilon_1'$, we have 
\begin{align*}
\|\mathcal{T}\|_{\mathcal{H}_{\kappa+\gamma}}\leq \epsilon \|D_{\xi}U\|_{\mathcal{H}_{\kappa+\gamma_0}}.
\end{align*}
For $\gamma\in[0,\gamma_0)$ this implies that $\mathcal{T}\to 0$ as $\xi\to\xi_0$.
\end{proof}
\subsubsection{$C^k$-smoothness}
Now we prove that $U$ is in  $C^k(\mathcal{B}^-;\mathcal{H}_{\kappa})$ for $k\geq 2$ by induction. Suppose that $U$ is of $C^j$-class as an operator from $\mathcal{B}^-$ to $\mathcal{H}_{j\kappa+2\gamma*}$ for all $1\leq j\leq k-1$, $\gamma*\in\left[0,\frac{\gamma_0}{2}\right)$. 
By a direct computation, we know that  $D^{k-1}_\xi u_{m}$ satisfies 
\begin{equation}
\begin{aligned}
D^{k-1}_{\xi}u_m(t)
&=\sum_{i=1}^{m}S^{-}(t+m-i)\int_{0}^{1} S^{-}(1-\tau)DF_{R(\theta_{i-1}\omega)}(u_{i-1}(\tau))D^{k-1}_{\xi}u_{i-1}(\tau)d\tau\\
&\quad+\sum_{i=1}^{m}S^{-}(t+m-i)\int_{0}^{1} S^{-}(1-\tau)DG_{R(\theta_{i-1}\omega)}(u_{i-1}(\tau))D^{k-1}_{\xi}u_{i-1}(\tau)d\theta_{i-1}\omega(\tau)\\
&\quad+\int_{0}^{t} S^{-}(t-\tau)DF_{R(\theta_m\omega)}(u_{m}(\tau))D^{k-1}_{\xi}u_{m}(\tau)d\tau\\
&\quad+\int_{0}^{t} S^{-}(t-\tau)DG_{R(\theta_m\omega)}(u_{m}(\tau))D^{k-1}_{\xi}u_{m}(\tau)d\theta_{m}\omega(\tau)\\
&\quad-\sum_{i=m+2}^{\infty}S^{+}(t+m-i)\int_{0}^{1} S^{+}(1-\tau)DF_{R(\theta_{i-1}\omega)}(u_{i-1}(\tau))D^{k-1}_{\xi}u_{i-1}(\tau)d\tau\\
&\quad-\sum_{i=m+2}^{\infty}
S^{+}
(t+m-i)\int_{0}^{1}S^{+}
(1-\tau)DG_{R(\theta_{i-1}\omega)}(u_{i-1}(\tau))D^{k-1}_{\xi}u_{i-1}(\tau)d\theta_{i-1}\omega(\tau)\\
&\quad-\int_{t}^{1} S^{+}(t-\tau)DF_{R(\theta_{m}\omega)}(u_m(\tau))D^{k-1}_{\xi}u_{m}(\tau)d\tau\\
&\quad-\int_{t}^{1} S^{+}(t-\tau)DG_{R(\theta_{m}\omega)}(u_m(\tau))D^{k-1}_{\xi}u_{m}(\tau)d\theta_{m}\omega(\tau)+[\mathcal{S}^{k-1}(\omega,\xi_0)]_m(t),
\end{aligned}
\end{equation}
where \begin{align*}
[\mathcal{S}^{k-1}(\omega,\xi_0)]_m(t)&=\sum_{i=1}^{m}S^{-}(t+m-i)\int_{0}^{1} S^{-}(1-\tau)\mathcal{F}^{k-1}_{R(\theta_{i-1}\omega)}(u_{i-1})(\tau)d\tau\\
&\quad+\sum_{i=1}^{m}S^{-}(t+m-i)\int_{0}^{1} S^{-}(1-\tau)\mathcal{G}^{k-1}_{R(\theta_{i-1}\omega)}(u_{i-1})(\tau)d\theta_{i-1}\omega(\tau)\\
&\quad+\int_{0}^{t} S^{-}(t-\tau)\mathcal{F}^{k-1}_{R(\theta_{m}\omega)}(u_{m})(\tau)d\tau+\int_{0}^{t} S^{-}(t-\tau)\mathcal{G}^{k-1}_{R(\theta_{m}\omega)}(u_{m})(\tau)d\theta_{m}\omega(\tau)\\
&\quad-\sum_{i=m+2}^{\infty}S^{+}(t+m-i)\int_{0}^{1} S^{+}(1-\tau)\mathcal{F}^{k-1}_{R(\theta_{i-1}\omega)}(u_{i-1})(\tau)d\tau\\
&\quad-\sum_{i=m+2}^{\infty}
S^{+}
(t+m-i)\int_{0}^{1}S^{+}
(1-\tau)\mathcal{G}^{k-1}_{R(\theta_{i-1}\omega)}(u_{i-1})(\tau)d\theta_{i-1}\omega(\tau)\\
&\quad-\int_{t}^{1} S^{+}(t-\tau)\mathcal{F}^{k-1}_{R(\theta_{m}\omega)}(u_{m})(\tau)d\tau-\int_{t}^{1} S^{+}(t-\tau)\mathcal{G}^{k-1}_{R(\theta_{m}\omega)}(u_{m})(\tau)d\theta_{m}\omega(\tau),
\end{align*}
with
\begin{align*}
\mathcal{F}^{k-1}_{R(\theta_{i-1}\omega)}(u_{i-1})(\tau)&=\sum_{l=0}^{k-3}
\begin{pmatrix}
k-2\\
l
 \end{pmatrix}D_{\xi}^{k-2-l}(DF_{R(\theta_{i-1}\omega)}(u_{i-1}(\tau)))D^{l+1}_{\xi}u_{i-1}(\tau),\\
\mathcal{G}^{k-1}_{R(\theta_{i-1}\omega)}(u_{i-1})(\tau)&=\sum_{l=0}^{k-3}\begin{pmatrix}
k-2\\
l
 \end{pmatrix}D_{\xi}^{k-2-l}(DG_{R(\theta_{i-1}\omega)}(u_{i-1}(\tau)))D^{l+1}_{\xi}u_{i-1}(\tau).
\end{align*}
Then
\begin{align*}
D_{\xi}\mathcal{F}^{k-1}_{R(\theta_{i-1}\omega)}(u_{i-1})(\tau)&=\sum_{l=0}^{k-3}\begin{pmatrix}
k-2\\
l
 \end{pmatrix}\bigg[D_{\xi}^{k-1-l}(DF_{R(\theta_{i-1}\omega)}(u_{i-1}(\tau)))D^{l+1}_{\xi}u_{i-1}(\tau)\\
&\quad+D_{\xi}^{k-2-l}(DF_{R(\theta_{i-1}\omega)}(u_{i-1}(\tau)))D^{l+2}_{\xi}u_{i}(\tau)\bigg],
\end{align*}
and 
\begin{align*}
D_{\xi}\mathcal{G}^{k-1}_{R(\theta_{i-1}\omega)}(u_{i-1})(\tau)&=\sum_{l=0}^{k-3}\begin{pmatrix}
k-2\\
l
 \end{pmatrix}\bigg[D_{\xi}^{k-1-l}(DG_{R(\theta_{i-1}\omega)}(u_{i-1}(\tau)))D^{l+1}_{\xi}u_{i-1}(\tau)\\
&\quad+D_{\xi}^{k-2-l}(DG_{R(\theta_{i-1}\omega)}(u_{i-1}(\tau)))D^{l+2}_{\xi}u_{i-1}(\tau)\bigg].
\end{align*}
Consequently, for $\mathcal{F}^{k-1}_{R}$, we have 
\begin{align*}
&\mathcal{F}^{k-1}_{R(\theta_{i-1}\omega)}(v_{i-1})-\mathcal{F}^{k-1}_{R(\theta_{i-1}\omega)}(u_{i-1})-D_{\xi}(\mathcal{F}^{k-1}_{R(\theta_{i-1}\omega)}(u_{i-1}))(\xi-\xi_0)\\
&=\sum_{l=0}^{k-3}\begin{pmatrix}
k-2\\
l
 \end{pmatrix}\bigg\{D_{\xi}^{k-2-l}\left(DF_{R(\theta_{i-1}\omega)}(v_{i-1})-DF_{R(\theta_{i-1}\omega)}(u_{i-1})\right)D^{l+1}_{\xi}v_{i-1}\\
&\quad+D_{\xi}^{k-2-l}(DF_{R(\theta_{i-1}\omega)}(u_{i-1}))\left[D^{l+1}_{\xi}v_{i-1}-D^{l+1}_{\xi}u_{i-1}\right]\\
&\quad-\left[D_{\xi}^{k-1-l}(DF_{R(\theta_{i-1}\omega)}(u_{i-1}))D^{l+1}_{\xi}u_{i-1}+D_{\xi}^{k-2-l}(DF_{R(\theta_{i-1}\omega)}(u_{i-1}))D^{l+2}_{\xi}u_{i-1}\right](\xi-\xi_0)\bigg\}\\
&=\sum_{l=0}^{k-3}\begin{pmatrix}
k-2\\
l
 \end{pmatrix}\bigg\{D_{\xi}^{k-2-l}\left(D^2F_{R(\theta_{i-1}\omega)}(w_{i-1})D_{\xi}w'_{i-1}(\xi-\xi_0)\right)D^{l+1}_{\xi}v_{i-1}\\
&\quad+D_{\xi}^{k-2-l}(DF_{R(\theta_{i-1}\omega)}(u_{i-1}))\left[D^{l+2}_{\xi}w''_{i-1}(\xi-\xi_0)\right]\\
&\quad-\left[D_{\xi}^{k-1-l}(DF_{R(\theta_{i-1}\omega)}(u_{i-1}))D^{l+1}_{\xi}u_{i-1}+D_{\xi}^{k-2-l}(DF_{R(\theta_{i-1}\omega)}(u_{i-1}))D^{l+2}_{\xi}u_{i-1}\right](\xi-\xi_0)\bigg\}\\
&=\sum_{l=0}^{k-3}\begin{pmatrix}
k-2\\
l
 \end{pmatrix}\bigg\{D_{\xi}^{k-2-l}\left(D^2F_{R(\theta_{i-1}\omega)}(w_{i-1})D_{\xi}w'_{i-1}-D^2F_{R(\theta_{i-1}\omega)}(u_{i-1})D_{\xi}u_{i-1}\right)\\
&\quad\times(\xi-\xi_0)D^{l+1}_{\xi}v_{i-1}\\
&\quad+D_{\xi}^{k-2-l}(DF_{R(\theta_{i-1}\omega)}(u_{i-1}))\left[D^{l+2}_{\xi}w''_{i-1}-D^{l+2}_{\xi}u_{i-1}\right](\xi-\xi_0)\\
&\quad-D_{\xi}^{k-1-l}(DF_{R(\theta_{i-1}\omega)}(u_{i-1}))\left[D^{l+1}_{\xi}v_{i-1}-D^{l+1}_{\xi}u_{i-1}\right](\xi-\xi_0)\bigg\},
\end{align*}
where $w_{i-1}=rv_{i-1}+(1-r)u_{i-1}$, $w'_{i-1}=r'v_{i-1}+(1-r')u_{i-1}$, and $w''_{i-1}=r''v_{i-1}+(1-r'')u_{i-1}$ for some $r,r',r''\in[0,1]$. The case of $G$ is similar, so we omit the details. We do not intend to write down the expressions for $D_{\xi}^{k-2-l}$ but analyse its construction which is enough to estimate it. For example, applying the chain rule to $D_{\xi}^{k-2-l}(DF_{R(\theta_{i-1}\omega)}(u_{i-1}(\tau)))$  (respectively, $D_{\xi}^{k-2-l}(DG_{R(\theta_{i-1}\omega)}(u_{i-1}(\tau)))$)
we find that each term in $\mathcal{F}^{k-1}_{R(\theta_{i-1}\omega)}(u_{i-1})$   ({respectively, }$\mathcal{G}^{k-1}_{R(\theta_{i-1}\omega)}(u_{i-1})$) contains factors $D^{l_1}F_{R(\theta_{i-1}\omega)}(u_{i-1})$  ({respectively, }$D^{l_1}G_{R(\theta_{i-1}\omega)}(u_{i-1})$) for some $l_1=1,2,\cdots,k-2$, 
and at least two derivatives $D^{l_2}_{\xi}u_{i-1}(\tau)$ and $D^{l_3}_{\xi}u_{i-1}(\tau)$ for some $l_2,l_3\in \{1,2,\cdots,k-2\}$ such that the sum of the order number of all derivatives 
is $k-1$. 
Hence
\begin{align}\label{DF-k}
\|D_{\xi}\mathcal{F}^{k-1}_{R(\theta_{i-1}\omega)}(u_{i-1})(\tau)\|_{L^{k}(\mathcal{B}^-;\mathcal{B})}\leq L_{F}(R(\theta_{i-1}\omega))C_{k,\gamma*}e^{-(k\kappa+2\gamma*)(i-1)},
\end{align} 
and
\begin{align}\label{DG-k}
\|D_{\xi}\mathcal{G}^{k-1}_{R(\theta_{i-1}\omega)}(u_{i-1})(\tau)\|_{L^{k}(\mathcal{B}^-;L_2(\mathcal{B};\mathcal{B}))}\leq L_{G}(R(\theta_{i-1}\omega))C_{k,\gamma*}e^{-(k\kappa+2\gamma*)(i-1)}.
\end{align}
 Hence  by assumption,  we have 
\begin{align}\label{F-k}
\nonumber&\|\mathcal{F}^{k-1}_{R(\theta_{i-1}\omega)}(v_{i-1})(\tau)-\mathcal{F}^{k-1}_{R(\theta_{i-1}\omega)}(u_{i-1})(\tau)-D_{\xi}\mathcal{F}^{k-1}_{R(\theta_{i-1}\omega)}(v_{i-1})(\tau)(\xi-\xi_0)\|_{L^{k-1}(\mathcal{B}^-;\mathcal{B})}\\
\nonumber&\leq C_{k,\gamma*}L_{F}(R(\theta_{i-1}\omega))|\xi-\xi_0|\bigg(e^{-(k\kappa+2\gamma*)(i-1)}|v_{i-1}(\tau)-u_{i-1}(\tau)|\\
&\quad\vee e^{-[(k-1)\kappa+2\gamma*](i-1)}|D_{\xi}v_{i-1}(\tau)-D_{\xi}u_{i-1}(\tau)|\vee\\
\nonumber&\quad\cdots\vee e^{-(\kappa+2\gamma*)(i-1)}|D^{k-1}_{\xi}v_{i-1}(\tau)-D^{k-1}_{\xi}u_{i-1}(\tau)|\bigg),
\end{align}
and
\begin{align}\label{G-k-1}
\nonumber&\|\mathcal{G}^{k-1}_{R(\theta_{i-1}\omega)}(v_{i-1})(\tau)-\mathcal{G}^{k-1}_{R(\theta_{i-1}\omega)}(u_{i-1}(\tau))-D_{\xi}\mathcal{G}^{k-1}_{R(\theta_{i-1}\omega)}(v_{i-1})(\tau)\|_{L^{k-1}(\mathcal{B}^-;L_2(\mathcal{B};\mathcal{B}_{-\beta}))}\\
\nonumber&\quad\leq C_{k,\gamma*}L_{G}(R(\theta_{i-1}\omega))|\xi-\xi_0|\bigg(e^{-[k\kappa+2\gamma*](i-1)}|v_{i-1}(\tau)-u_{i-1}(\tau)|_{-\beta}\\
&\quad\vee e^{-[(k-1)\kappa+2\gamma*](i-1)}|D_{\xi}v_{i-1}(\tau)-D_{\xi}u_{i-1}(\tau)|_{-\beta}\vee\\
\nonumber&\quad\cdots\vee e^{-(\kappa+2\gamma*)(i-1)}|D^{k-1}_{\xi}v_{i-1}(\tau)-D^{k-1}_{\xi}u_{i-1}(\tau)|_{-\beta}\bigg),
\end{align}
and
\begin{align}\label{G-k-2}
\nonumber&\|\mathcal{G}^{k-1}_{R(\theta_{i-1}\omega)}(v_{i-1})(r)-\mathcal{G}^{k-1}_{R(\theta_{i-1}\omega)}(u_{i-1}(r))-D_{\xi}\mathcal{G}^{k-1}_{R(\theta_{i-1}\omega)}(v_{i-1})(r)\\\nonumber&\quad-\mathcal{G}^{k-1}_{R(\theta_{i-1}\omega)}(v_{i-1})(q)
+\mathcal{G}^{k-1}_{R(\theta_{i-1}\omega)}(u_{i-1}(q))+D_{\xi}\mathcal{G}^{k-1}_{R(\theta_{i-1}\omega)}(v_{i-1})(q)\|_{L^{k-1}(\mathcal{B}^-;L_2(\mathcal{B};\mathcal{B}_{-\beta}))}\\
\nonumber&\quad\leq C_{k,\gamma*}L_{G}(R(\theta_{i-1}\omega))(r-q)^{\beta}|\xi-\xi_0|\bigg(e^{-[k\kappa+2\gamma*](i-1)}\|v_{i-1}-u_{i-1}\|_{\beta,-\beta}\\
&\quad\vee e^{-[(k-1)\kappa+2\gamma*](i-1)}\|D_{\xi}v_{i-1}-D_{\xi}u_{i-1}\|_{\beta,-\beta}\vee\\
\nonumber&\quad\cdots\vee e^{-(\kappa+2\gamma*)(i-1)}\|D^{k-1}_{\xi}v_{i-1}-D^{k-1}_{\xi}u_{i-1}\|_{\beta,-\beta}\bigg),
\end{align}
where $C_{k,\gamma*}$ only depends on $\|D_{\xi}^j U\|_{j\kappa+\gamma}$ for $i=0,1,2,\cdots,k-1$ and $\gamma\in[0,2\gamma*]$.

Recalling the definition of $\mathcal{L}$, we denote by 
\begin{equation}
\begin{aligned}
&[\mathcal{I}(D^{k-1}_{\xi}V-D^{k-1}_{\xi}U)]_{m}(t)\\
&:=[D^{k-1}_{\xi}V-D^{k-1}_{\xi}U]_{m}(t)-[\mathcal{L}(D^{k-1}_{\xi}V-D^{k-1}_{\xi}U)]_{m}(t)\\
&=\sum_{i=1}^{m}S^{-}(t+m-i)\int_{0}^{1} S^{-}(1-\tau)\bigg[DF_{R(\theta_{i-1}\omega)}(v_{i-1}(\tau))\\
&~~~~~~~~~~~~~~~~~~~~~~~~~~~~~~~~~~~~~~~~~~~~~~~-DF_{R(\theta_{i-1}\omega)}(u_{i-1}(\tau))\bigg]D^{k-1}_{\xi}v_{i-1}(\tau)d\tau\\
&\quad+\sum_{i=1}^{m}S^{-}(t+m-i)\int_{0}^{1} S^{-}(1-\tau)\bigg[DG_{R(\theta_{i-1}\omega)}(v_{i-1}(\tau))\\
&~~~~~~~~~~~~~~~~~~~~~~~~~~~~~~~~~~~~~~~~~~~~~~~-DG_{R(\theta_{i-1}\omega)}(u_{i-1}(\tau))\bigg]D^{k-1}_{\xi}v_{i-1}(\tau)\omega(\tau)\\
&\quad+\int_{0}^{t} S^{-}(t-\tau)[DF_{R(\theta_{m}\omega)}(v_{m}(\tau))-DF_{R(\theta_{m}\omega)}(u_{m}(\tau))]D^{k-1}_{\xi}v_{m}(\tau)d\tau\\
&\quad+\int_{0}^{t} S^{-}(t-\tau)[DG_{R(\theta_{m}\omega)}(v_{m}(\tau))-DG_{R(\theta_{m}\omega)}(u_{m}(\tau))]D^{k-1}_{\xi}v_{m}(\tau)d\theta_{m}\omega(\tau)\\
&\quad-\sum_{i=m+2}^{\infty}S^{+}(t+m-i)\int_{0}^{1} S^{+}(1-\tau)\bigg[DF_{R(\theta_{i-1}\omega)}(v_{i-1}(\tau))\\
&~~~~~~~~~~~~~~~~~~~~~~~~~~~~~~~~~~~~~~~~~~~~~~~~~-DF_{R(\theta_{i-1}\omega)}(u_{i-1}(\tau))\bigg]D^{k-1}_{\xi}v_{i-1}(\tau)d\tau\\
&\quad-\sum_{i=m+2}^{\infty}
S^{+}
(t+m-i)\int_{0}^{1}S^{+}
(1-\tau)\bigg[DG_{R(\theta_{i-1}\omega)}(v_{i-1}(\tau))\\
&~~~~~~~~~~~~~~~~~~~~~~~~~~~~~~~~~~~~~~~~~~~~~~~~~-DG_{R(\theta_{i-1}\omega)}(u_{i-1}(\tau))\bigg]D^{k-1}_{\xi}v_{i-1}(\tau)d\theta_{i-1}\omega(\tau)\\
&\quad-\int_{t}^{1} S^{+}(t-\tau)[DF_{R(\theta_{m}\omega)}(v_{i-1}(\tau))-DF_{R(\theta_{m}\omega)}(u_{m}(\tau))]D^{k-1}_{\xi}v_{m}(\tau)d\tau\\
&\quad -\int_{t}^{1} S^{+}(t-\tau)[DG_{R(\theta_{m}\omega)}(v_{m}(\tau))-DG_{R(\theta_{m}\omega)}(u_{m}(\tau))]D^{k-1}_{\xi}v_{m}(\tau)d\theta_{m}\omega(\tau)\\
&\quad+[\mathcal{S}^{k-1}(\omega,\xi)]_m(t)-[\mathcal{S}^{k-1}(\omega,\xi_0)]_m(t).
\end{aligned}
\end{equation}
To obtain the continuous differentiability of $D_{\xi}^{k-1}u$, we decompose $\mathcal{I}$ as follows:
\begin{align*}
&[\mathcal{I}(D^{k-1}_{\xi}V-D^{k-1}_{\xi}U)]_{m}(t)\\
&\quad=[\mathcal{S}^{k-1}(\omega,\xi)]_m(t)-[\mathcal{S}^{k-1}(\omega,\xi_0)]_m(t)-D_{\xi}[\mathcal{S}^{k-1}(\omega,\xi_0)]_m(t)(\xi-\xi_0)\\
&\qquad+D_{\xi}[\mathcal{S}^{k-1}(\omega,\xi_0)]_m(t)(\xi-\xi_0)+\mathcal{I}^{1}(t)+\mathcal{I}^{2}(t)+\mathcal{I}^{3}(t)+\mathcal{I}^{4}(t)(\xi-\xi_0),
\end{align*}
and  rewrite it as 
\begin{equation}\label{difference-k-1}
\begin{aligned}
&D^{k-1}_{\xi}v_{m}(t)-D^{k-1}_{\xi}u_{m}(t)-(id-\mathcal{L})^{-1}(D_{\xi}[\mathcal{S}^{k-1}(\omega,\xi_0)]_{m}(t){(\xi-\xi_0)}+ \mathcal{I}^4(\xi-\xi_0))\\
&\quad=(id-\mathcal{L})^{-1}\{[\mathcal{S}^{k-1}(\omega,\xi)]_m(t)-[\mathcal{S}^{k-1}(\omega,\xi_0)]_m(t)-D_{\xi}[\mathcal{S}^{k-1}(\omega,\xi_0)]_m(t)(\xi-\xi_0)\\
&\qquad+\mathcal{I}^{1}(t)+\mathcal{I}^{2}(t)+\mathcal{I}^{3}(t)\},
\end{aligned}
\end{equation}
where
\begin{align*}
\mathcal{I}^1(t)
&=\sum_{i=1}^{m}S^{-}(t+m-i)\int_{0}^{1} S^{-}(1-\tau)[DF_{R(\theta_{i-1}\omega)}(v_{i-1}(\tau))-DF_{R(\theta_{i-1}\omega)}(u_{i-1}(\tau))\\
&\quad-D^2F_{R(\theta_{i-1}\omega)}(u_{i-1}(\tau))(v_{i-1}(\tau)-u_{i-1}(\tau))]D^{k-1}_{\xi}u_{i-1}(\tau)d\tau\\
&\quad+\sum_{i=1}^{m}S^{-}(t+m-i)\int_{0}^{1} S^{-}(1-\tau)[DG_{R(\theta_{i-1}\omega)}(v_{i-1}(\tau))-DG_{R(\theta_{i-1}\omega)}(u_{i-1}(\tau))\\
&\quad-D^2G_{R(\theta_{i-1}\omega)}(u_{i-1}(\tau))(v_{i-1}(\tau)-u_{i-1}(\tau))]D^{k-1}_{\xi}u_{i-1}(\tau)\omega(\tau)\\
&\quad+\int_{0}^{t} S^{-}(t-\tau)[DF_{R(\theta_{m}\omega)}(v_{m}(\tau))-DF_{R(\theta_{m}\omega)}(u_{m}(\tau))\\
&\quad-D^2F_{R(\theta_{m}\omega)}(u_{m}(\tau))(v_{m}(\tau)-u_{m}(\tau))]D^{k-1}_{\xi}u_{m}(\tau)d\tau\\
&\quad+\int_{0}^{t} S^{-}(t-\tau)[DG_{R(\theta_{m}\omega)}(v_{m}(\tau))-DG_{R(\theta_{m}\omega)}(u_{m}(\tau))\\
&\quad-D^2G_{R(\theta_{m}\omega)}(u_{m}(\tau))(v_{m}(\tau)-u_{m}(\tau))]D^{k-1}_{\xi}u_{m}(\tau)d\theta_{m}\omega(\tau)\\
&\quad-\sum_{i=m+2}^{\infty}S^{+}(t+m-i)\int_{0}^{1} S^{+}(1-\tau)[DF_{R(\theta_{i-1}\omega)}(v_{i-1}(\tau))-DF_{R(\theta_{i-1}\omega)}(u_{i-1}(\tau))\\
&\quad-D^2F_{R(\theta_{i-1}\omega)}(u_{i-1}(\tau))(v_{i-1}(\tau)-u_{i-1}(\tau))]D^{k-1}_{\xi}u_{i-1}(\tau)d\tau\\
&\quad-\sum_{i=m+2}^{\infty}
S^{+}
(t+m-i)\int_{0}^{1}S^{+}
(1-\tau)[DG_{R(\theta_{i-1}\omega)}(v_{i-1}(\tau))-DG_{R(\theta_{i-1}\omega)}(u_{i-1}(\tau))\\
&\quad-D^2G_{R(\theta_{i-1}\omega)}(u_{i-1}(\tau))(v_{i-1}(\tau)-u_{i-1}(\tau))]D^{k-1}_{\xi}u_{i-1}(\tau)d\theta_{i-1}\omega(\tau)\\
&\quad-\int_{t}^{1} S^{+}(t-\tau)[DF_{R(\theta_{m}\omega)}(v_{i-1}(\tau))-DF_{R(\theta_{m}\omega)}(u_{m}(\tau))\\
&\quad-D^2F_{R(\theta_{m}\omega)}(u_{m}(\tau))(v_{m}(\tau)-u_{m}(\tau))]D^{k-1}_{\xi}u_{m}(\tau)d\tau\\
&\quad -\int_{t}^{1} S^{+}(t-\tau)[DG_{R(\theta_{m}\omega)}(v_{m}(\tau))-DG_{R(\theta_{m}\omega)}(u_{m}(\tau))\\
&\quad-D^2G_{R(\theta_{m}\omega)}(u_{m}(\tau))(v_{m}(\tau)-u_{m}(\tau))]D^{k-1}_{\xi}u_{m}(\tau)d\theta_{m}\omega(\tau),
\end{align*}
\begin{align*}
\mathcal{I}^2(t)
&=\sum_{i=1}^{m}S^{-}(t+m-i)\int_{0}^{1} S^{-}(1-\tau)[D^2F_{R(\theta_{i-1}\omega)}(u_{i-1}(\tau))(v_{i-1}(\tau)-u_{i-1}(\tau))\\
&\quad -D^2F_{R(\theta_{i-1}\omega)}(u_{i-1}(\tau))D_{\xi}u_{i-1}(\tau)(\xi-\xi_0)]D^{k-1}_{\xi}u_{i-1}(\tau)d\tau\\
&\quad+\sum_{i=1}^{m}S^{-}(t+m-i)\int_{0}^{1} S^{-}(1-\tau)[D^2G_{R(\theta_{i-1}\omega)}(u_{i-1}(\tau))(v_{i-1}(\tau)-u_{i-1}(\tau))\\
&\quad-D^2G_{R(\theta_{i-1}\omega)}(u_{i-1}(\tau))D_{\xi}u_{i-1}(\tau)(\xi-\xi_0)]D^{k-1}_{\xi}u_{i-1}(\tau)\omega(\tau)\\
&\quad+\int_{0}^{t} S^{-}(t-\tau)[D^2F_{R(\theta_{m}\omega)}(u_{m}(\tau))(v_{m}(\tau)-u_{m}(\tau))\\
&\quad-D^2F_{R(\theta_{m}\omega)}(u_{m}(\tau))D_{\xi}u_{m}(\tau)(\xi-\xi_0)]D^{k-1}_{\xi}u_{m}(\tau)d\tau\\
&\quad+\int_{0}^{t} S^{-}(t-\tau)[D^2G_{R(\theta_{m}\omega)}(u_{m}(\tau))(v_{m}(\tau)-u_{m}(\tau))\\
&\quad-D^2G_{R(\theta_{m}\omega)}(u_{m}(\tau))D_{\xi}u_{m}(\tau)(\xi-\xi_0)]D^{k-1}_{\xi}u_{m}(\tau)d\theta_{m}\omega(\tau)\\
&\quad-\sum_{i=m+2}^{\infty}S^{+}(t+m-i)\int_{0}^{1} S^{+}(1-\tau)[D^2F_{R(\theta_{i-1}\omega)}(u_{i-1}(\tau))(v_{i-1}(\tau)-u_{i-1}(\tau))\\
&\quad-D^2F_{R(\theta_{i-1}\omega)}(u_{i-1}(\tau))D_{\xi}u_{i-1}(\tau)(\xi-\xi_0)]D^{k-1}_{\xi}u_{i-1}(\tau)d\tau\\
&\quad-\sum_{i=m+2}^{\infty}
S^{+}
(t+m-i)\int_{0}^{1}S^{+}
(1-\tau)[D^2G_{R(\theta_{i-1}\omega)}(u_{i-1}(\tau))(v_{i-1}(\tau)-u_{i-1}(\tau))\\
&\quad-D^2G_{R(\theta_{i-1}\omega)}(u_{i-1}(\tau))D_{\xi}u_{i-1}(\tau)(\xi-\xi_0)]D^{k-1}_{\xi}u_{i-1}(\tau)d\theta_{i-1}\omega(\tau)\\
&\quad-\int_{t}^{1} S^{+}(t-\tau)[D^2F_{R(\theta_{m}\omega)}(u_{m}(\tau))(v_{m}(\tau)-u_{m}(\tau))\\
&\quad-D^2F_{R(\theta_{m}\omega)}(u_{m}(\tau))D_{\xi}u_{m}(\tau)(\xi-\xi_0)]D^{k-1}_{\xi}u_{m}(\tau)d\tau\\
&\quad -\int_{t}^{1} S^{+}(t-\tau)[D^2G_{R(\theta_{m}\omega)}(u_{m}(\tau))(v_{m}(\tau)-u_{m}(\tau))\\
&\quad-D^2G_{R(\theta_{m}\omega)}(u_{m}(\tau))D_{\xi}u_{m}(\tau)(\xi-\xi_0)]D^{k-1}_{\xi}u_{m}(\tau)d\theta_{m}\omega(\tau),
\end{align*}
\begin{align*}
&\mathcal{I}^3(t)\\
&~=\sum_{i=1}^{m}S^{-}(t+m-i)\int_{0}^{1} S^{-}(1-\tau)[DF_{R(\theta_{i-1}\omega)}(v_{i-1}(\tau))-DF_{R(\theta_{i-1}\omega)}(u_{i-1}(\tau))]\\
&\quad\times[D^{k-1}_{\xi}v_{i-1}(\tau)-D^{k-1}_{\xi}u_{i-1}(\tau)]d\tau\\
&\quad+\sum_{i=1}^{m}S^{-}(t+m-i)\int_{0}^{1} S^{-}(1-\tau)[DG_{R(\theta_{i-1}\omega)}(v_{i-1}(\tau))-DG_{R(\theta_{i-1}\omega)}(u_{i-1}(\tau))]\\
&\quad\times[D^{k-1}_{\xi}v_{i-1}(\tau)-D^{k-1}_{\xi}u_{i-1}(\tau)]d\theta_{i-1}\omega(\tau)\\
&\quad+\int_{0}^{t} S^{-}(t-\tau)[DF_{R(\theta_{m}\omega)}(v_{m}(\tau))-DF_{R(\theta_{m}\omega)}(u_{m}(\tau))][D^{k-1}_{\xi}v_{m}(\tau)-D^{k-1}_{\xi}u_{m}(\tau)]d\tau\\
&\quad+\int_{0}^{t} S^{-}(t-\tau)[DG_{R(\theta_{m}\omega)}(v_{m}(\tau))-DG_{R(\theta_{m}\omega)}(u_{m}(\tau))][D^{k-1}_{\xi}v_{m}(\tau)-D^{k-1}_{\xi}u_{m}(\tau)]d\theta_{m}\omega(\tau)\\
&\quad-\sum_{i=m+2}^{\infty}S^{+}(t+m-i)\int_{0}^{1} S^{+}(1-\tau)[DF_{R(\theta_{i-1}\omega)}(v_{i-1}(\tau))-DF_{R(\theta_{i-1}\omega)}(u_{i-1}(\tau))]\\
&\quad\times[D^{k-1}_{\xi}v_{i-1}(\tau)-D^{k-1}_{\xi}u_{i-1}(\tau)]d\tau\\
&\quad-\sum_{i=m+2}^{\infty}
S^{+}
(t+m-i)\int_{0}^{1}S^{+}
(1-\tau)[DG_{R(\theta_{i-1}\omega)}(v_{i-1}(\tau))-DG_{R(\theta_{i-1}\omega)}(u_{i-1}(\tau))]\\
&\quad\times[D^{k-1}_{\xi}v_{i-1}(\tau)-D^{k-1}_{\xi}u_{i-1}(\tau)]d\theta_{i-1}\omega(\tau)\\
&\quad-\int_{t}^{1} S^{+}(t-\tau)[DF_{R(\theta_{m}\omega)}(v_{m}(\tau))-DF_{R(\theta_{m}\omega)}(u_{m}(\tau))][D^{k-1}_{\xi}v_{m}(\tau)-D^{k-1}_{\xi}u_{m}(\tau)]d\tau\\
&\quad -\int_{t}^{1} S^{+}(t-\tau)[DG_{R(\theta_{m}\omega)}(v_{m}(\tau))-DG_{R(\theta_{m}\omega)}(u_{m}(\tau))][D^{k-1}_{\xi}v_{m}(\tau)-D^{k-1}_{\xi}u_{m}(\tau)]d\theta_{m}\omega(\tau),
\end{align*}
and 
\begin{align*}
&\mathcal{I}^4(t)(\xi-\xi_0)\\
&~=\sum_{i=1}^{m}S^{-}(t+m-i)\int_{0}^{1} S^{-}(1-\tau)D^2F_{R(\theta_{i-1}\omega)}(u_{i-1}(\tau))D_{\xi}u_{i-1}(\tau)(\xi-\xi_0)D^{k-1}_{\xi}u_{i-1}(\tau)d\tau\\
&\quad+\sum_{i=1}^{m}S^{-}(t+m-i)\int_{0}^{1} S^{-}(1-\tau)D^2G_{R(\theta_{i-1}\omega)}(u_{i-1}(\tau))D_{\xi}u_{i-1}(\tau)\\
&\quad\times(\xi-\xi_0)D^{k-1}_{\xi}u_{i-1}(\tau)d\theta_{i-1}\omega(\tau)\\
&\quad+\int_{0}^{t} S^{-}(t-\tau)D^2F_{R(\theta_{m}\omega)}(u_{m}(\tau))D_{\xi}u_{m}(\tau)(\xi-\xi_0)D^{k-1}_{\xi}u_{m}(\tau)d\tau\\
&\quad+\int_{0}^{t} S^{-}(t-\tau)D^2G_{R(\theta_{m}\omega)}(u_{m}(\tau))D_{\xi}u_{m}(\tau)(\xi-\xi_0)D^{k-1}_{\xi}u_{m}(\tau)d\theta_{m}\omega(\tau)\\
&\quad-\sum_{i=m+2}^{\infty}S^{+}(t+m-i)\int_{0}^{1} S^{+}(1-\tau)D^2F_{R(\theta_{i-1}\omega)}(u_{i-1}(\tau))D_{\xi}u_{i-1}(\tau)(\xi-\xi_0)D^{k-1}_{\xi}u_{i-1}(\tau)d\tau\\
&\quad-\sum_{i=m+2}^{\infty}
S^{+}
(t+m-i)\int_{0}^{1}S^{+}
(1-\tau)D^2G_{R(\theta_{i-1}\omega)}(u_{i-1}(\tau))D_{\xi}u_{i-1}(\tau)\\
&\quad\times(\xi-\xi_0)D^{k-1}_{\xi}u_{i-1}(\tau)d\theta_{i-1}\omega(\tau)\\
&\quad-\int_{t}^{1} S^{+}(t-\tau)D^2F_{R(\theta_{m}\omega)}(u_{m}(\tau))D_{\xi}u_{m}(\tau)(\xi-\xi_0)D^{k-1}_{\xi}u_{m}(\tau)d\tau\\
&\quad -\int_{t}^{1} S^{+}(t-\tau)D^2G_{R(\theta_{m}\omega)}(u_{m}(\tau))D_{\xi}u_{m}(\tau)(\xi-\xi_0)D^{k-1}_{\xi}u_{m}(\tau)d\theta_{m}\omega(\tau).
\end{align*}
\begin{lemma}
Under the assumptions $(\mathbf{A_2^k}),(\mathbf{A_3^k})$ and \eqref{gap k}, the right hand side of (\ref{difference-k-1}) is of order $o(|\xi-\xi_0|)$ as $\xi\to \xi_0$.
\end{lemma}
\begin{proof}
We first consider $$[\mathcal{S}^{k-1}(\omega,\xi)]_m(t)-[\mathcal{S}^{k-1}(\omega,\xi_0)]_m(t)-D_{\xi}[\mathcal{S}^{k-1}(\omega,\xi_0)]_m(t)(\xi-\xi_0).$$ For simplicity, denoting it by $\Delta [\mathcal{S}(\xi,\xi_0)]_m(t)$, we have
\begin{align*}
&\Delta [\mathcal{S}(\xi,\xi_0)]_m(t)\\
&~=\sum_{i=1}^{m}S^{-}(t+m-i)\int_{0}^{1} S^{-}(1-\tau)[\mathcal{F}^{k-1}_{R(\theta_{i-1}\omega)}(v_{i-1})(\tau)-\mathcal{F}^{k-1}_{R(\theta_{i-1}\omega)}(u_{i-1})(\tau)\\
&\quad-D_{\xi}(\mathcal{F}^{k-1}_{R(\theta_{i-1}\omega)}(u_{i-1}))(\tau)(\xi-\xi_0)]d\tau\\
&\quad+\sum_{i=1}^{m}S^{-}(t+m-i)\int_{0}^{1} S^{-}(1-\tau)[\mathcal{G}^{k-1}_{R(\theta_{i-1}\omega)}(v_{i-1})(\tau)-\mathcal{G}^{k-1}_{R(\theta_{i-1}\omega)}(u_{i-1})(\tau)\\
&\quad-D_{\xi}(\mathcal{G}^{k-1}_{R(\theta_{i-1}\omega)}(u_{i-1}))(\tau)(\xi-\xi_0)]d\theta_{i-1}\omega(\tau)\\
&\quad+\int_{0}^{t} S^{-}(t-\tau)[\mathcal{F}^{k-1}_{R(\theta_{m}\omega)}(v_{m})(\tau)-\mathcal{F}^{k-1}_{R(\theta_{m}\omega)}(u_{m})(\tau)-D_{\xi}(\mathcal{F}^{k-1}_{R(\theta_{m}\omega)}(u_{m}))(\tau)(\xi-\xi_0)]d\tau\\
&\quad+\int_{0}^{t} S^{-}(t-\tau)[\mathcal{G}^{k-1}_{R(\theta_{m}\omega)}(v_{m})(\tau)-\mathcal{G}^{k-1}_{R(\theta_{m}\omega)}(u_{m})(\tau)-D_{\xi}(\mathcal{G}^{k-1}_{R(\theta_{m}\omega)}(u_{m}))(\tau)(\xi-\xi_0)]d\theta_{m}\omega(\tau)\\
&\quad-\sum_{i=m+2}^{\infty}S^{+}(t+m-i)\int_{0}^{1} S^{+}(1-\tau)[\mathcal{F}^{k-1}_{R(\theta_{i-1}\omega)}(v_{i-1})(\tau)-\mathcal{F}^{k-1}_{R(\theta_{i-1}\omega)}(u_{i-1})(\tau)\\
&\quad-D_{\xi}(\mathcal{F}^{k-1}_{R(\theta_{i-1}\omega)}(u_{i-1}))(\tau)(\xi-\xi_0)]d\tau\\
&\quad-\sum_{i=m+2}^{\infty}
S^{+}
(t+m-i)\int_{0}^{1}S^{+}
(1-\tau)[\mathcal{G}^{k-1}_{R(\theta_{i-1}\omega)}(v_{i-1})(\tau)-\mathcal{G}^{k-1}_{R(\theta_{i-1}\omega)}(u_{i-1})(\tau)\\
&\quad-D_{\xi}(\mathcal{G}^{k-1}_{R(\theta_{i-1}\omega)}(u_{i-1}))(\tau)(\xi-\xi_0)]d\theta_{i-1}\omega(\tau)\\
&\quad-\int_{t}^{1} S^{+}(t-\tau)[\mathcal{F}^{k-1}_{R(\theta_{m}\omega)}(v_{m})(\tau)-\mathcal{F}^{k-1}_{R(\theta_{m}\omega)}(u_{m})(\tau)-D_{\xi}(\mathcal{F}^{k-1}_{R(\theta_{m}\omega)}(u_{m}))(\tau)(\xi-\xi_0)]d\tau\\
&\quad-\int_{t}^{1} S^{+}(t-\tau)[\mathcal{G}^{k-1}_{R(\theta_{m}\omega)}(v_{m})(\tau)-\mathcal{G}^{k-1}_{R(\theta_{m}\omega)}(u_{m})(\tau)-D_{\xi}(\mathcal{G}^{k-1}_{R(\theta_{m}\omega)}(u_{m}))(\tau)(\xi-\xi_0)]d\theta_{m}\omega(\tau).
\end{align*}
Based on the estimates of \eqref{F-k},\eqref{G-k-1}, \eqref{G-k-2} and the continuous dependence on $\xi$ of $D_{\xi}^lU$, for $l=0,1,\cdots,k-1$, we know that for any $\epsilon>0$,  there exists an $\epsilon_k>0$ such that when $|\xi-\xi_0|\leq \epsilon_k$,   we have 
\begin{align*}
\max_{0\leq l \leq k-1}\sup_{i\in\mathbb{Z}^+}\|D_{\xi}^lv_{i}-D_{\xi}^lu_{i}\|_{\beta,-\beta}\leq \mathcal{E}_{k},
\end{align*}
where \begin{align*}
\frac{\epsilon}{\mathcal{E}_{k}}=KC_{k,\gamma*}\sup_{m\in\mathbb{Z}^+}\bigg[\sum_{i=0}^m e^{\check{\mu}(m-i)+(k\kappa+\gamma)m-(k\kappa+2\gamma*)i}+\sum_{i=m}^\infty e^{\hat{\mu}(m-i)+(k\kappa+\gamma)m-(k\kappa+2\gamma*)i}\bigg].
\end{align*}
Then 

\begin{align*}
\sup_{m\in \mathbb{Z}^+}e^{m(k\kappa+\gamma)}\|\Delta\mathcal{S}(\xi,\xi_0)\|_{\beta,-\beta}&\leq KC_{k,\gamma*}\mathcal{E}_{k}|\xi-\xi_0|\sup_{m\in\mathbb{Z}^+}\bigg[\sum_{i=0}^m e^{\check{\mu}(m-i)+(k\kappa+\gamma)m-(k\kappa+2\gamma*)i}\\&\quad+\sum_{i=m}^\infty e^{\hat{\mu}(m-i)+(k\kappa+\gamma)m-(k\kappa+2\gamma*)i}\bigg]\\
&\leq \epsilon|\xi-\xi_0|.
\end{align*}
Applying  similar procedures to $\mathcal{I}^1,\mathcal{I}^2,\mathcal{I}^3$, we additionally obtain that 
\begin{align*}
 \mathcal{I}^1= o(|\xi-\xi_0|),
~~\mathcal{I}^2= o(|\xi-\xi_0|),
~~\mathcal{I}^3= o(|\xi-\xi_0|),
\end{align*}
as $\xi\to\xi_0$, which completes the proof.
\end{proof}
As a result, we have the following equality:
\begin{align*}
&D^{k-1}_{\xi}v_{m}(t)-D^{k-1}_{\xi}u_{m}(t)-(id-\mathcal{L})^{-1}(D_{\xi}[\mathcal{S}^{k-1}(\omega,\xi_0)]_{m}(t)(\xi-\xi_0)+ \mathcal{I}^4(\xi-\xi_0))=o(|\xi-\xi_0|).
\end{align*}
Besides, one can also obtain 
\begin{align*}
&\| D_{\xi}[\mathcal{S}^{k-1}(\omega,\xi_0)]_{m}(\xi-\xi_0)\|_{\mathcal{H}_{k\kappa+\gamma}}\leq C|\xi-\xi_0|,~~
\| \mathcal{I}^4(\xi-\xi_0)\|_{\mathcal{H}_{k\kappa+\gamma}}\leq C|\xi-\xi_0|
\end{align*}
for some constant $C$ independent of $V$. Hence $D_{\xi}[\mathcal{S}^{k-1}(\omega,\xi_0)]_{m}(t)$ and $\mathcal{I}^4$ are bounded  linear  operators on $\mathcal{B}^-$,
which implies the existence of   $D^{k}_{\xi}u$  in $\mathcal{H}_{k\kappa+\gamma}$ for $\gamma\in[0,2\gamma*]$. By a direct computation, we have
\begin{equation}
\begin{aligned}
&D^{k}_{\xi}u_m(t)\\
&=\sum_{i=1}^{m}S^{-}(t+m-i)\int_{0}^{1} S^{-}(1-\tau)DF_{R(\theta_{i-1}\omega)}(u_{i-1}(\tau))D^{k}_{\xi}u_{i-1}(\tau)d\tau\\
&\quad+\sum_{i=1}^{m}S^{-}(t+m-i)\int_{0}^{1} S^{-}(1-\tau)DG_{R(\theta_{i-1}\omega)}(u_{i-1}(\tau))D^{k}_{\xi}u_{i-1}(\tau)d\theta_{i-1}\omega(\tau)\\
&\quad+\int_{0}^{t} S^{-}(t-\tau)DF_{R(\theta_m\omega)}(u_{m}(\tau))D^{k}_{\xi}u_{m}(\tau)d\tau\\
&\quad+\int_{0}^{t} S^{-}(t-\tau)DG_{R(\theta_m\omega)}(u_{m}(\tau))D^{k}_{\xi}u_{m}(\tau)d\theta_{m}\omega(\tau)\\
&\quad-\sum_{i=m+2}^{\infty}S^{+}(t+m-i)\int_{0}^{1} S^{+}(1-\tau)DF_{R(\theta_{i-1}\omega)}(u_{i-1}(\tau))D^{k}_{\xi}u_{i-1}(\tau)d\tau\\
&\quad-\sum_{i=m+2}^{\infty}
S^{+}
(t+m-i)\int_{0}^{1}S^{+}
(1-\tau)DG_{R(\theta_{i-1}\omega)}(u_{i-1}(\tau))D^{k}_{\xi}u_{i-1}(\tau)d\theta_{i-1}\omega(\tau)\\
&\quad-\int_{t}^{1} S^{+}(t-\tau)DF_{R(\theta_{m}\omega)}(u_m(\tau))D^{k}_{\xi}u_{m}(\tau)d\tau\\
&\quad-\int_{t}^{1} S^{+}(t-\tau)DG_{R(\theta_{m}\omega)}(u_m(\tau))D^{k}_{\xi}u_{m}(\tau)d\theta_{m}\omega(\tau)+[\mathcal{S}^k(\omega,\xi_0)]_{m}(t),
\end{aligned}
\end{equation}
By $(\mathbf{A_2^k})$ and $(\mathbf{A_3^k})$ and \eqref{gap k}, it could be checked that there exists a constant $M>0$ such that 
\begin{align*}
\|\mathcal{S}^k(\omega,\xi_0)\|_{\mathcal{H}_{k\kappa+\kappa}}\leq M.
\end{align*}
Hence regarding the estimate of $\mathcal{L}$, we have 
\begin{align*}
\|D_{\xi}^k U\|_{\mathcal{H}_{k\kappa+\kappa}}&\leq M+\frac{1}{2}\|D_{\xi}^{k} U\|_{\mathcal{H}_{k\kappa+\kappa}}
\leq 2M
\end{align*}
for any $\gamma\in[0,2\gamma*]$. For the continuity of  $D_{\xi}^k U$  with respect with $\xi$, recalling the definition of $\mathcal{L}$, we express the difference of $D_\xi^k u$ and $D_\xi^k v$  as
\begin{align*}
D_\xi^k v_{m}(t)-D_\xi^k u_{m}(t)&=[\mathcal{L}(D_\xi^k V(t)-D_\xi^k U)]_{m}(t)+[\mathcal{T}^{k}]_{m}(t)
+[\mathcal{S}^k(\omega,\xi)]_{m}(t)-[\mathcal{S}^k(\omega,\xi_0)]_{m}(t),
\end{align*}
where
\begin{equation}
\begin{aligned}
&[\mathcal{T}^k]_{m}(t)\\
&=\sum_{i=1}^{m}S^{-}(t+m-i)\int_{0}^{1} S^{-}(1-\tau)[DF_{R(\theta_{i-1}\omega)}(v_{i-1}(\tau))\\
&\quad-DF_{R(\theta_{i-1}\omega)}(u_{i-1}(\tau))]D^k_{\xi}v_{i-1}(\tau)d\tau\\
&\quad+\sum_{i=1}^{m}S^{-}(t+m-i)\int_{0}^{1} S^{-}(1-\tau)[DG_{R(\theta_{i-1}\omega)}(v_{i-1}(\tau))\\
&\quad-DG_{R(\theta_{i-1}\omega)}(u_{i-1}(\tau))]D^k_{\xi}v_{i-1}(\tau)d\theta_{i-1}\omega(\tau)\\
&\quad+\int_{0}^{t} S^{-}(t-\tau)[DF_{R(\theta_m\omega)}(v_{m}(\tau))-DF_{R(\theta_m\omega)}(u_{m}(\tau))]D^k_{\xi}v_{m}(\tau)d\tau
\\
&\quad+\int_{0}^{t} S^{-}(t-\tau)[DG_{R(\theta_m\omega)}(v_{m}(\tau))-DG_{R(\theta_m\omega)}(u_{m}(\tau))]D^k_{\xi}v_{m}(\tau)d\theta_{m}\omega(\tau)\\
&\quad-\sum_{i=m+2}^{\infty}S^{+}(t+m-i)\int_{0}^{1} S^{+}(1-\tau)[DF_{R(\theta_{i-1}\omega)}(v_{i-1}(\tau))\\
&\quad-DF_{R(\theta_{i-1}\omega)}(u_{i-1}(\tau))]D^k_{\xi}v_{i-1}(\tau)d\tau\\
&\quad-\sum_{i=m+2}^{\infty}
S^{+}
(t+m-i)\int_{0}^{1}S^{+}
(1-\tau)[DG_{R(\theta_{i-1}\omega)}(v_{i-1}(\tau))
\\&\quad-DG_{R(\theta_{i-1}\omega)}(u_{i-1}(\tau))]D^k_{\xi}v_{i-1}(\tau)d\theta_{i-1}\omega(\tau)\\
&\quad-\int_{t}^{1} S^{+}(t-\tau)[DF_{R(\theta_{m}\omega)}(v_m(\tau))-DF_{R(\theta_{m}\omega)}(u_m(\tau))]D^k_{\xi}v_{m}(\tau)d\tau\\
&\quad-\int_{t}^{1} S^{+}(t-\tau)[DG_{R(\theta_{m}\omega)}(v_m(\tau))-DG_{R(\theta_{m}\omega)}(u_m(\tau))]D^k_{\xi}v_{m}(\tau)d\theta_{m}\omega(\tau).
\end{aligned}
\end{equation}
Then similar to the estimate of $\mathcal{T}$, one can derive by a straightforward computation that  $\|\mathcal{T}^k\|_{\mathcal{H}_{k\kappa+\gamma}}\to0$ as $\xi\to\xi_0$ for $\gamma\in[0,\gamma*)$. Consequently, regarding the previous deliberations, the proof of the $C^k$-smoothness of the local stable manifold is completed.

\section*{Data availability}
No data was used for the research described in the article.

\section*{Acknowledgments} 
We thank the referee for carefully reading the manuscript and for the valuable
suggestions.



	
	

\end{document}